\documentclass[10pt]{article} % For LaTeX2e
\usepackage[preprint]{tmlr}
\usepackage{float}

\usepackage{amsmath,amsfonts,bm}

\def\eqref#1{equation~\ref{#1}}
\def\1{\bm{1}}

\DeclareMathAlphabet{\mathsfit}{\encodingdefault}{\sfdefault}{m}{sl}
\SetMathAlphabet{\mathsfit}{bold}{\encodingdefault}{\sfdefault}{bx}{n}

\DeclareMathOperator*{\argmin}{arg\,min}

\usepackage{hyperref}
\usepackage{url}


\title{Preference-Based Reward Learning under Partial Observability with Inexact Dynamics}

\author{\name Reza Zolnouri \email zolnouri@mathc.rwth-aachen.de \\
      \addr Department of Mathematics\\
      RWTH Aachen University
      \\
      \medskip
      \\
      \name Semih Cayci \email cayci@mathc.rwth-aachen.de \\
      \addr Department of Mathematics \\
      RWTH Aachen University}

\usepackage{amsmath}
\usepackage{amssymb}
\usepackage{mathtools}
\usepackage{amsthm}
\usepackage{mathrsfs}
\usepackage{bm}
\usepackage{bbm}
\usepackage{algorithm}
\usepackage{algorithmic}
\usepackage{graphicx}
\usepackage{subcaption}
\usepackage[textsize=tiny]{todonotes}

\usepackage[capitalize,noabbrev]{cleveref}

\usepackage{placeins}
\theoremstyle{plain}
\newtheorem{theorem}{Theorem}[section]
\newtheorem{proposition}[theorem]{Proposition}
\newtheorem{lemma}[theorem]{Lemma}
\newtheorem{corollary}[theorem]{Corollary}
\theoremstyle{definition}
\newtheorem{definition}[theorem]{Definition}
\newtheorem{assumption}[theorem]{Assumption}
\theoremstyle{remark}
\newtheorem{remark}[theorem]{Remark}

\theoremstyle{plain}
\newtheorem{example}[theorem]{Example}
 
\usepackage{array}
\usepackage{makecell}
\usepackage{wrapfig}
\usepackage{float}

\usepackage{xcolor}
\usepackage{enumitem}

\usepackage[textsize=tiny]{todonotes}

\usepackage{algorithmic}

\usepackage[most]{tcolorbox}

\usepackage{longtable}
\usepackage{tikz}

\usetikzlibrary{calc}
\usetikzlibrary{positioning}
\usetikzlibrary{arrows.meta}
\usetikzlibrary{decorations.pathreplacing}
\usetikzlibrary{decorations.markings}
\usetikzlibrary{angles,quotes}

\definecolor{ForestGreen}{rgb}{0.13,0.55,0.13}
\definecolor{maroon}{RGB}{190,0,42}

\newcommand{\bA}{\mathbb{A}}
\newcommand{\bB}{\mathbb{B}}

\newcommand{\bE}{\mathbb{E}}

\newcommand{\bN}{\mathbb{N}}

\newcommand{\bP}{\mathbb{P}}

\newcommand{\bR}{\mathbb{R}}
\newcommand{\bS}{\mathbb{S}}

\newcommand{\cF}{\mathcal{F}}

\def\month{MM}  
\def\year{YYYY} 
\def\openreview{\url{https://openreview.net/forum?id=XXXX}}

\begin{document}

\maketitle

\begin{abstract}
In this paper, we study how partial observability and inexact latent-state inference affect reward learning from preferences. We consider preference-based reward learning under partial observability, where the learner forms latent-state estimates using an inexact learned POMDP model, so model error can accumulate over time. For finite log-linear POMDPs, we characterize this error term by establishing the stability of the belief filter to parametric model error under certain mixing conditions, yielding bounds on the belief mismatch in expectation and in high probability. For the same log-linear class, we analyze covariance-preconditioned first-order EM and establish local population contraction under observed-information identifiability and finite-sample convergence to a statistical neighborhood of the true POMDP parameter. We further extend the filtering-stability mechanism beyond the log-linear setting to neural-softmax POMDP models with overparameterized neural networks. We then propagate these errors into trajectory-level feature perturbations and derive finite-sample guarantees for constrained Bradley--Terry reward estimation from preferences.
Our results separate the preference-sampling error from the model-mismatch error induced by a fixed learned dynamics model, and clarify how finite-sample POMDP estimation, filtering stability, and reward learning interact under partial observability.
\end{abstract}

\section{Introduction}
\label{sec:introduction}
 
Reinforcement Learning from Human Feedback (RLHF) combines \emph{preference-based reward learning} and \emph{policy optimization}. A common pipeline collects pairwise (or K-wise) comparisons between trajectory segments, fits a parametric preference model, e.g., Bradley--Terry model \citep{Bradley1952}, to cumulative trajectory features, and then optimizes a policy using the learned reward \citep{ChristianoLBMLA17, Wirth2017, kaufmann2025a}. Recent theory for this mechanism in fully observed MDPs analyzes regularized preference maximum likelihood estimation (MLE) under self-normalized concentration \citep{ZhuPrincipledComparisons, Du2024Exploration-DrivenUtilization}. Consider an interactive agent such as an LLM: the user-facing observations are text, clicks, or other surface signals, while the \emph{real} state includes latent human factors (e.g., beliefs, goals, interpretations, accuracy and affective state) that are not accessible to the system. In particular, we argue that RLHF must recover an \emph{unknown} reward from preference feedback while simultaneously operating under latent-state uncertainty.

This motivates modeling RLHF interaction as a Partially Observable Markov Decision Process (POMDP), where an unobserved latent state evolves Markovianly and the agent receives noisy observations. In the motivating RLHF setting, the finite POMDP spaces need not enumerate the raw interaction space directly. High-dimensional observations, such as user utterances or textual feedback, may be mapped to task or dataset specific semantic categories, so that the finite observation alphabet represents abstract observation classes rather than individual raw realizations. Likewise, the latent-state space may encode task-relevant abstractions of user intent, belief, or interaction context rather than every microscopic state. Such constructions have precedent in POMDP-based dialogue systems, where raw utterances are mapped to dialogue acts and large hidden user-goal spaces are represented through partitions of equivalent states over which beliefs are maintained \citep{YOUNG2010150}. Under the POMDP model, the current observation need not contain all information relevant to the latent state and need not be a sufficient statistic, whereas the posterior belief provides the standard sufficient statistic for the interaction history and yields the induced belief-state process (the belief-MDP) \citep{Astrom1965, KAELBLING199899, Uesato2022}. Consequently, treating the noisy observation as if it were the true state can lead to state aliasing, i.e., two histories may yield the same observation but different beliefs, and hence potentially different expected rewards conditioned on the history,\(
r_b(b,a)=\sum_{s\in\bS} b(s)r(s,a).
\) A memoryless observation-based reward model cannot distinguish such cases and can therefore be structurally misspecified under partial observability. Using exact beliefs allows the reward model to distinguish histories that an observation-only (memoryless) model cannot. This identifies a potential representational benefit over observation-only models. Whether this benefit translates into more accurate reward estimates depends on the accuracy of the estimated beliefs and the available preference data. Our bounds quantify these statistical effects.
Accordingly, we lift the preference dataset (trajectory pairs with preference labels) into the belief-MDP representation by constructing beliefs along each recorded trajectory under the learned model and forming belief-based comparison features. However, this reduction introduces an additional, orthogonal difficulty: belief states are not directly observed and must be computed via Bayesian filtering. In practice, filtering uses an \emph{inexact learned} POMDP model estimated from data, so preference features are evaluated on \emph{approximate} belief trajectories, creating a structured, history-dependent model-mismatch channel. Therefore, approximation error propagates from learned dynamics to belief error, then to trajectory-feature perturbations, and finally to error in reward-parameter estimation (Figure~\ref{fig:error-propagation-pipeline}).

\begin{figure}[H]
\centering
\resizebox{\columnwidth}{!}{%
\begin{tikzpicture}[
  font=\scriptsize,
  node distance=0.95cm,
  box/.style={draw, rounded corners=1.5pt, inner xsep=4pt, inner ysep=3pt, align=center},
  arr/.style={-{Stealth[length=1.7mm]}, semithick}
]
\node[box] (est) {POMDP model\\ estimation};
\node[box, right=of est] (bel) {belief\\ error};
\node[box, right=of bel] (feat) {feature\\ perturbation};
\node[box, right=of feat] (mle) {reward-estimation\\ error};

\draw[arr] (est) -- (bel);
\draw[arr] (bel) -- (feat);
\draw[arr] (feat) -- (mle);
\end{tikzpicture}%
}
\caption{Error-propagation pipeline}
\label{fig:error-propagation-pipeline}
\vspace{-0.6em}
\end{figure}

For the log-linear model, we make the first stage of this pipeline explicit through covariance-preconditioned gradient EM. Combining this result with the filtering and reward-learning guarantees yields an end-to-end bound from incomplete POMDP trajectories to reward estimation. The central technical obstacle is that Bayesian filtering is generally \emph{not} contractive in total variation and may amplify errors over time; our stability guarantee is therefore established under a sufficient Dobrushin condition that jointly controls the transition and observation channels. Building on a recent stochastic filter stability theorem \cite{McYuk}, we establish expectation-level stability and quantify the error introduced by parametric model mismatch at each filtering step. This analysis yields an explicit Lipschitz stability bound for the parameter-to-belief mapping, along with high-probability control over the time-averaged belief error using a martingale argument. We first develop this mechanism for finite log-linear POMDPs and then extend it to neural-softmax POMDP models in the kernel regime, where finite-width linearization errors enter additively into the belief perturbation bound. These results can then be extended to define a uniform belief-accuracy event applicable across all comparisons within an RLHF dataset.

At the reward-learning stage, the resulting belief mismatch perturbs the trajectory-level feature differences used in preference comparisons and the associated empirical feature covariance. We propagate these perturbations through a Bradley--Terry MLE analysis, obtaining a finite-sample reward recovery bound that separates (i) a statistical term vanishing with the number of human comparisons $N_{\mathrm{HF}}$, (ii) a model-mismatch error controlled by filtering stability and by the accuracy of the learned POMDP model, and (iii) the regularization bias introduced to ensure well-posedness when the feature covariance is ill-conditioned or singular. This decomposition clarifies when preference-based RLHF remains statistically feasible under partial observability with an \emph{inexact} learned POMDP model, and which stability and modeling choices govern the achievable accuracy.

\paragraph{Contributions.}
\begin{itemize}[leftmargin=*]
\item \textbf{A theoretical framework for preference-based reward learning under partial observability.}
We establish, to the best of our knowledge, the first finite-sample reward-learning
guarantees for preference-based RLHF under partial observability with inexact learned dynamics. We lift
the preference dataset into the belief-MDP representation and provide an explicit
error-propagation analysis from POMDP model mismatch to reward-estimation error. This
setting is not covered by existing fully observed RLHF theory, where trajectory
features are assumed to be directly observed and exactly specified.

\item \textbf{Belief stability under learned log-linear dynamics and conditional neural-softmax extensions.}
Under partial observability, preference features are built from inferred beliefs rather than observed states, so learned-model error can propagate through the Bayesian filter into reward learning. Existing filter-stability results mainly address sensitivity to initialization error and do not directly quantify this recurring model-mismatch effect. We close this gap by proving that, for log-linear POMDP dynamics, the belief filter is Lipschitz-stable in expectation under parametric model error and a Dobrushin-type mixing condition (Theorem~\ref{thm:param-belief-lip}). We further give a conditional neural-softmax extension under a local NTK-style linearization event.

\item \textbf{Covariance-preconditioned POMDP estimation.}
For the log-linear model, we analyze covariance-preconditioned first-order EM and establish local population contraction under observed-information identifiability and finite-sample convergence to a statistical neighborhood. Combined with the filtering and reward-learning guarantees, this yields an end-to-end bound from incomplete POMDP trajectories to reward estimation.
\item \textbf{Robust Bradley--Terry reward learning from approximate beliefs.}
We derive a finite-sample guarantee for the constrained Bradley--Terry MLE
(Theorem~\ref{thm:main-mu}) in the partially observed setting, where trajectory features are constructed from approximate beliefs rather than observed states. The resulting bound decomposes into a statistical error, a model-mismatch error induced by belief error, and a covariance-regularization contribution. This extends fully observed preference-learning analyses by quantifying how latent-state inference error propagates into reward estimation, and yields an explicit preference-sample complexity conditional on the accuracy of the reconstructed beliefs.
\end{itemize}

\paragraph{Organization.}
Section~\ref{sec:setup} introduces the POMDP model, log-linear dynamics, Bayesian filtering, and the induced belief-MDP representation. Section~\ref{sec:belief-stability} derives expectation-level and high-probability stability guarantees for Bayesian filtering under learned log-linear dynamics. Section~\ref{sec:pomdp-estimation-em} develops population and finite-sample guarantees for covariance-preconditioned gradient EM and connects POMDP parameter estimation to the preceding belief-stability analysis. Appendix~\ref{sec:nn-extension} extends the belief-stability mechanism to neural-softmax models through an NTK-style linearization. Section~\ref{sec:rlhf} formalizes the preference-learning setup and establishes our main reward estimation guarantee under belief approximation error, together with the resulting end-to-end consequence for finite-sample POMDP estimation. Section~\ref{sec:conclusion} discusses implications, limitations, and future directions. Technical proofs are deferred to the appendix. Appendix~\ref{sec:Experiments} reports synthetic experiments for filtering stability, covariance-preconditioned EM, and reward learning under belief perturbations.

\section{Related Work}
\label{sec:related_work}

Our setting draws on four adjacent literatures: learning under partial observability, stability of Bayesian filtering, latent-variable estimation by EM, and preference-based RLHF. POMDPs formalize decision-making with latent states and noisy observations \citep{Astrom1965}, but are statistically hard to learn in full generality \citep{Krishnamurthy2016, Jin2020}, motivating work under structural assumptions and algorithmic frameworks for planning and learning under partial observability \citep{KAELBLING199899, Du2019, liu2022sample, Golowich2022, Zhan2022, cayci2024recurrent}.

A closely related technical thread studies \emph{filter stability} under initialization error. Classical analyses typically require strong mixing assumptions \citep{VanHandel2008}, while recent Dobrushin-type results provide explicit exponential stability by combining transition and observation contractions \citep{McYuk}. However, these results do not directly address the learned-model regime considered here, where parametric errors in both the transition and observation models introduce fresh perturbations at every Bayesian update rather than only an initial filtering discrepancy. This repeated error-injection mechanism prevents a direct application of existing initialization-stability bounds. Our stability theorem fills this gap by showing that, under a stronger sufficient Dobrushin condition, learned-model error induces a controlled and explicit belief-mismatch bound.

Related EM analyses typically focus on latent-variable models rather than the joint estimation of transition and observation models from partially observed trajectories, conditional on observed actions. EM and its first-order variants are standard methods for likelihood-based estimation with latent or missing data, including hidden Markov and state-space models \citep{HMMCappe,van2008hidden}. \citet{balakrishnan2017statistical} develop a population-to-sample framework in which local curvature and gradient stability yield population contraction, while uniform empirical-gradient control gives finite-sample convergence to a statistical neighborhood. We adapt this framework to controlled log-linear POMDP trajectories and carry out the analysis in the complete-data covariance geometry. Observed-information identifiability yields the local gradient-stability behavior used for population contraction, while covariance-relative filtering and smoothing stability control the pilot complete-data covariance entering the finite-sample preconditioner.

RLHF learns rewards from human feedback and then optimizes policies using standard RL machinery \citep{ChristianoLBMLA17, Wirth2017, Ji2023}. Reward learning is often instantiated via Bradley--Terry model \citep{Bradley1952, Plackett1975} and combined with policy optimization methods such as PPO \citep{Schulman2017}, underpinning successes in robotics and interactive systems \citep{Brown2019} and LLM alignment \citep{Ziegler2019, Stiennon2020, Bai2022, Ouyang2022, Achiam2023}. Recent theory analyzes preference-based MLE and RLHF-style learning in fully observed MDPs, from tabular regimes to function approximation and exploration-driven policy optimization \citep{Novoseller2020, Chen2022, ZhuPrincipledComparisons, Du2024Exploration-DrivenUtilization, kaufmann2025a, Saha2023}.

A smaller line of work connects RLHF to latent-state structure and non-Markovian rewards. Most closely, \citet{Kausik2024} model RLHF via \emph{partially-observed reward-states} (PORRL/PORMDP), where the environment state is observed but feedback depends on an additional latent ``internal'' reward-state. In contrast, we consider latent environment states with no access to rewards, and the agent interacts only through noisy observations and preference feedback. We quantify how belief computation under an \emph{inexact learned} dynamics model perturbs trajectory features and induces a principled error in Bradley--Terry reward estimation.

\section{Setup and Preliminaries}
\label{sec:setup}

\subsection{POMDP model and trajectories}
\label{subsec:pomdp}

We work with a finite POMDP \(
\bigl(\bS,\ \bA,\ \hat\bS,\ P_\theta,\ \Phi_w,\ r,\ \nu_0\bigr),
\) where \(\bS\) is the finite latent state space, \(\hat\bS\) is the finite observation space, \(\bA\) is the finite action set, \(P_\theta(\cdot\mid s,a)\in\Delta_{\bS}\) is the transition kernel, and \(\Phi_w(\cdot\mid s)\in\Delta_{\hat\bS}\) is the observation model. The reward function is \(r:\bS\times\bA\to\bR\), and the initial state distribution \(\nu_0\in\Delta_{\bS}\) is assumed known. The stability analysis below depends on the transition and observation channels only through their Dobrushin coefficients and the resulting explicit cardinality-dependent constants.

\begin{assumption}[Model structure and bounded features]\label{as:Model-LogLinear}

Model dynamics admit the following log-linear structure
\begin{align}
&P_\theta(s'|s,a)
=\frac{\exp\{\theta^\top\phi_p(s,a,s')\}}{\sum_{\bar s\in\bS}\exp\{\theta^\top\phi_p(s,a,\bar s)\}},\qquad
&\Phi_{w}(\hat s\,|\,s)=\frac{\exp\{ w^\top \phi_\Phi(s,\hat s)\}}{\sum_{\bar s\in\hat \bS}\exp\{ w^\top\phi_\Phi(s,\bar s)\}}.\notag
\end{align}
\medskip\noindent
Here $\phi_p$ is a feature map for the transition kernel and $\phi_\Phi$ is a feature map for the observation model, and we assume 
\(
\sup_{s,a,s'}\|\phi_p(s,a,s')\|_2\le B,\text{ and }\sup_{s,\hat s}\|\phi_\Phi(s,\hat s)\|_2\le B.
\)
For some $B<\infty$. The initialization also satisfies
\begin{equation}\label{eq:mu-min-max}
1>\nu_{\max}\ge\nu_0(s)\ge\nu_{\min}>0, \qquad \forall s \in \bS. 
\end{equation}
\end{assumption}

\begin{remark}[Exponential-family modeling and model-based RL]\label{rem:expfam-mbrl}
The log-linear parametrizations in Assumption \ref{as:Model-LogLinear} are finite exponential-family models, standard in statistical modeling (e.g., \citet{bishop2007, wainwright2008, van2008hidden}). One can use the assumed learned models for belief filtering and downstream policy optimization. In this work we focus on the preference-based reward learning guarantees, treating the policy-optimization stage as modular and leaving its integration to future work.
\end{remark}

\begin{assumption}[Realizability and local parameter error]\label{as:realizability-bounded-error}
There exists a true parameter \(\Theta^\star=(\theta^\star,w^\star)\in\mathbb{R}^{d_\theta+d_w}\)
such that the data-generating POMDP has transition kernel \(P_{\theta^\star}\) and observation model \(\Phi_{w^\star}\). Fix \(r_\Theta>0\), and suppose the learned parameter \(\Theta=(\theta,w)\) satisfies
\(
\Theta \in \mathbb{B}_2(\Theta^\star,r_\Theta):= \{\Theta' \in \mathbb{R}^{d_\theta+d_w} : \|\Theta'-\Theta^\star\|_2 \le r_\Theta\}.
\)
Equivalently,
\(
\|\Theta-\Theta^\star\|_2^2=\|\theta-\theta^\star\|_2^2 + \|w-w^\star\|_2^2\le r_\Theta^2.
\) We write \(
\delta(\theta):=\|\theta-\theta^\star\|_2,\,
\delta(w):=\|w-w^\star\|_2,
\) so that \(\delta(\theta)^2+\delta(w)^2\le r_\Theta^2\).
\end{assumption}

Given trajectories of observations and actions, a natural way to estimate the model parameter $\Theta=(\theta,w)$ in this class is via maximum-likelihood or an EM-type procedure. Fix a horizon $T\in\bN$. For a single trajectory, let the latent state sequence be
$Z := (s_{0:T})\in\bS^{T+1}$ and let the observation sequence be
$Y := (\hat s_{1:T},\, a_{0:T-1})\in \hat\bS^{T}\times\bA^{T}$.
We treat the action sequence $a_{0:T-1}$ as observed and work with the likelihood conditional on actions.
Under $\Theta=(\theta,w)$, the joint conditional trajectory likelihood factorizes as
\begin{align}\label{eq:traj-joint}
P_\Theta(Y,Z)
&=\nu_0(s_0)\,
\prod_{k=0}^{T-1} P_\theta(s_{k+1}\mid s_k,a_k)\,
\prod_{k=1}^{T}\Phi_w(\hat s_k\mid s_k).
\end{align}
In this case, the model parameter may be estimated by maximizing the marginal likelihood \(P_\Theta(Y)=\sum_ZP_\Theta(Y,Z)\). The downstream belief-stability and reward-learning results are estimator-agnostic and require only that the learned parameter lie in the local neighborhood of Assumption~\ref{as:realizability-bounded-error}. Section~\ref{sec:pomdp-estimation-em} instantiates this requirement for the log-linear model by deriving population and finite-sample parameter-error guarantees for covariance-preconditioned gradient EM.

Given the observation--action history up to time $t$,
\(Y_t := (\hat s_{1:t},\, a_{0:t-1}),\) the corresponding belief is defined as
\(
b_t^\Theta(s)\;:=\;\bP_\Theta(s_t=s \mid Y_t)\in\Delta_{\bS}, \text{ and } b_0^\Theta(s)=\nu_0(s).
\) For parameters $\Theta=(\theta,w)$, the Bayesian filter update is
\begin{align}\label{eq:belief-formula}
b_{t+1}^{\Theta}(s')
:=\frac{\Phi_w(\hat s_{t+1} \mid s')\sum_{s\in\bS}P_\theta(s'\mid s,a_t)\,b_t^{\Theta}(s)}
{\sum_{\bar s\in\bS}\Phi_w(\hat s_{t+1}\mid \bar s)\sum_{s\in\bS}P_\theta(\bar s\mid s,a_t)\,b_t^{\Theta}(s)}.
\end{align}
Moreover, the induced belief-MDP reward can be written as
\[
r_b(b_t^\Theta,a_t)
:=\bE_{s_t\sim b_t^\Theta}\!\big[r(s_t,a_t)\big]
=\sum_{s\in\bS} b_t^\Theta(s)\,r(s,a_t).
\]
The recursion Eq.~\ref{eq:belief-formula} defines a nonlinear operator on $\Delta_{\bS}$ obtained by composing a prediction step (via $P_\theta$) with a Bayes reweighting step (via $\Phi_w$). In general, this operator can amplify small perturbations in the model parameters or in the current belief over time. To obtain quantitative \emph{stability} guarantees for the belief process, we impose an explicit \emph{mixing} criterion that controls how strongly transitions forget the previous latent states. This is captured by the following Dobrushin coefficients, which will serve as the key constants governing one-step contraction and the resulting stability rate.

\medskip

\begin{definition}[Dobrushin coefficient]\label{def:dobrushin-controlled}
Let $P_\theta(\cdot\mid s,a)\in\Delta_{\bS}$ be the transition kernel and $\Phi_w\in\Delta_{\hat\bS}$ the observation model. Define the uniform Dobrushin coefficient of the models over the feasible choice of parameter set by

%$\bB_2(\theta^\star,r_\Theta)\times \bB_2(w^\star,r_\Theta)$, $s,s'\in\bS$ and $a\in \bA$

\[
\kappa_P:=\inf_{\theta}\inf_{s,s',a}
\sum_{x\in\bS}\min \{P_\theta(x\mid s,a),P_\theta(x\mid s',a)\},\qquad \kappa_\Phi := \inf_{w}\inf_{s,s'\in\bS}
\sum_{\hat s\in\hat\bS}\min\{\Phi_w(\hat s\mid s), \Phi_w(\hat{s}\mid s')\}.
\]

Equivalently, $\kappa_P=1-\sup_{\theta}\sup_{a,s,s'}
\frac12\|P_\theta(\cdot\mid s,a)-P_\theta(\cdot\mid s',a)\|_1.$
\end{definition}

\begin{assumption}[Dobrushin stability]\label{as:Dobrushin-stability}
Let \(\kappa_P,\kappa_\Phi\in(0,1]\) be the uniform Dobrushin coefficients in Definition~\ref{def:dobrushin-controlled}. Assume that
\(
\alpha := (1-\kappa_P)(4-3\kappa_\Phi) < 1.
\)
\end{assumption}

\medskip
\noindent
Assumption~\ref{as:Dobrushin-stability} is a sufficient condition tailored to recurring model mismatch and is stronger than the initialization-stability condition of \citet{McYuk}. The stronger requirement reflects the fresh transition and observation model error injected at each filtering step. A detailed comparison, scalability discussion, and a concrete example are given in Appendix~\ref{app:dobrushin-discussion}.

\noindent
The next section quantifies the effect of model mismatch by proving that, under $\alpha<1$, the belief mapping $\Theta\mapsto b_t^\Theta$ is Lipschitz in expectation around $\Theta^\star$ (Theorem~\ref{thm:param-belief-lip}). This stability bound is the key technical input for transferring POMDP parameter error into controlled feature perturbations in the downstream RLHF analysis (Theorem~\ref{thm:main-mu}).
\medskip

%%%%%%%%%%%%%%%%%%%%%%%%%%%%%%%%%%%%%%%%%%%%%%%%%%%%%%%%%%%%%%%%%%%%%%%%%%%%
\section{Belief Stability Under Model Mismatch}
\label{sec:belief-stability}

Let the true parameters be \(\Theta^\star=(\theta^\star,w^\star)\) and the estimate
\(\Theta=(\theta,w)\), living in a compact feasible parameter set as in
Assumption~\ref{as:realizability-bounded-error}.

\label{subsec:belief-expectation}

\begin{theorem}\label{thm:param-belief-lip}
Consider the POMDP with log-linear dynamics, under Assumptions
\ref{as:Model-LogLinear}, \ref{as:realizability-bounded-error}, \ref{as:Dobrushin-stability}.
Fix an action sequence $(a_0,\dots,a_{t-1})$, estimated model parameters $\Theta=(\theta,w)$, true parameters $\Theta^\star=(\theta^\star,w^\star)$ and corresponding belief processes  $(b_k^\Theta)_{k=0}^t$ and $(b_k^{\Theta^\star})_{k=0}^t$ with priors $b_0^\Theta=b_0^{\Theta^\star}=\nu_0\in\Delta_{\bS}$.
Then, the mapping from parameters $\Theta =(\theta, w)\to b_t^\Theta$ is Lipschitz at $\Theta^\star$ in expectation. In particular, for every $t\ge1$,
\begin{align}\label{eq:belief-lipschitz}
\bE\Big[\|b_{t}^\Theta-b_{t}^{\Theta^\star}\|_{\mathrm{TV}}\Big]
\le \frac{(1-\alpha^{t})}{1-\alpha}c_b\,(\delta(\theta),\delta(w)),
\end{align}
where $c_b\, (\delta(\theta),\delta(w)):=B\bigl(\,\delta(w)+\frac{(3-2\kappa_\Phi)}{2}\,\delta(\theta)\bigr),$ and the expectation is over $\hat s_{1:t}\sim P_{\Theta^\star}(\,\cdot\,|a_{0:t-1}).$
\end{theorem}

\paragraph{Proof sketch.}
The proof decomposes one-step belief error into (I) observation-model perturbation, (II) transition-model perturbation, and (III) propagation under filter mixing, then unrolls the resulting recursion. The detailed proof is provided in Appendix \ref{proof-thm-belief}.

\begin{remark}[Error in belief stability]\label{rem:belief-error-decomp}
The bound in Eq.~\ref{eq:belief-lipschitz} splits into a \emph{time-propagation} factor and a
\emph{one-step mismatch} term:
\begin{align*}
\bE\!\left[\|b_t^\Theta-b_t^{\Theta^\star}\|_{\mathrm{TV}}\right]
\;\le\;\underbrace{\frac{1-\alpha^t}{1-\alpha}}_{\text{propagation over time}}
\;\times\;\underbrace{\Big(B\,\delta(w)+\tfrac{B}{2}(3-2\kappa_\Phi)\,\delta(\theta)\Big)}_{\text{model mismatch}}.
\end{align*}

When $\alpha<1$, the propagation factor is uniformly bounded by $(1-\alpha)^{-1}$, so the
belief error scales linearly with the parameter deviations $\delta(w)$ and $\delta(\theta)$.
\end{remark}
\noindent

\begin{remark}[Filter stability challenges]
\label{rem:stability-context}
Analyzing filter stability over long horizons is notoriously difficult because the Bayesian update operator $\psi$ is generally \emph{not} a contraction and can expand the total variation distance between beliefs (see \citet{McYuk}, Example 3.4). Standard results often impose restrictive strong mixing conditions on the transition kernel $P(\cdot|\cdot,\cdot)$ and assume that it is sufficiently ergodic \cite{VanHandel2008}. Using the Dobrushin coefficient $\kappa_\Phi$ of the observation model can account for the \emph{joint} contraction properties of the transition and measurement steps, allowing for stability even when the transition kernel alone is not sufficiently mixing. The price for this generality is that the contraction holds in \emph{expectation} rather than almost surely. Moreover, while \citet{McYuk} establishes stability with respect to incorrect \emph{initialization}, our Theorem~\ref{thm:param-belief-lip} extends this machinery to the learning setting. We establish stability with respect to \emph{parametric model mismatch} ($\Theta$ vs.\ $\Theta^\star$), deriving explicit perturbation bounds specific to the log-linear family. This quantifies how parameter error injects bias at every step (terms (I) and (II) in Eq.~\ref{eq:belief-theorem-term-I} in Appendix \ref{proof-thm-belief}), a distinct challenge from the pure initialization decay studied in classical filtering literature.
\end{remark}

Theorem~\ref{thm:param-belief-lip} controls the belief error in \emph{expectation} at a fixed time $t$ under model mismatch. For the RLHF reduction we need control of the accumulated belief mismatch along a trajectory, since downstream feature and preference observations depend on the entire history. We therefore convert the per-step expected contraction into a \emph{high-probability bound} on the \emph{time-average} belief error.

\begin{remark}[Parallel notation for the neural-softmax extension]\label{rem:parallel-nn-notation}
The log-linear model is the primary setting in this work, with parameter $\Theta=(\theta,w)$ and associated quantities $\kappa_P,\kappa_\Phi,\alpha,c_b$, and $\epsilon_b$. The neural-softmax results are parallel extensions of the same filtering argument. We distinguish them by using the network parameter $W=(W_p,W_\Phi)$ and the superscript $\mathrm{NN}$ for the corresponding constants, namely $\kappa_P^{\mathrm{NN}},\kappa_\Phi^{\mathrm{NN}},\alpha_{\mathrm{NN}},c_b^{\mathrm{NN}}$, and $\epsilon_b^{\mathrm{NN}}$. The downstream reward-learning analysis depends on either model only through the resulting belief-accuracy bound.
\end{remark}

\begin{definition}[Neural-softmax transition and observation models]
\label{def:nn-softmax-model}
Let the transition and observation kernels be parameterized by neural-network scores
\begin{align}\label{eq:nn-softmax-transition}
P_{W_p}(s' \mid s,a)
&=\frac{\exp\{F_p(s,a,s';W_p)\}}{\sum_{\bar s\in\bS}\exp\{F_p(s,a,\bar s;W_p)\}}\,\,,\qquad \Phi_{W_\Phi}(\hat s \mid s)=\frac{\exp\{F_\Phi(s,\hat s;W_\Phi)\}}
{\sum_{\bar s\in\hat\bS}\exp\{F_\Phi(s,\bar s;W_\Phi)\}}
\end{align}

We take both score functions to be two-layer ReLU networks with width \(m\). 
\end{definition}

\begin{corollary}[Belief perturbation bound for neural-softmax models]
\label{cor:nn-param-belief-lip}
Consider the POMDP introduced earlier with neural-softmax transition and observation models in Definition~\ref{def:nn-softmax-model}, under Assumption~\ref{as:nn-model-linearization}. Let \(W=(W_p,W_\Phi)\in\mathcal D\) be an estimated neural parameter and let \(W^\star=(W_p^\star,W_\Phi^\star)\in\mathcal D\) be the true parameter generating the POMDP. Consider an action sequence \((a_0,\dots,a_{t-1})\) and the corresponding belief processes \((b_k^W)_{k=0}^t\) and \((b_k^{W^\star})_{k=0}^t\), initialized from the same prior \(b_0^W=b_0^{W^\star}=\nu_0\in\Delta_{\bS}\). Let \(\kappa_P^{\mathrm{NN}}\) and \(\kappa_\Phi^{\mathrm{NN}}\) be the uniform neural Dobrushin coefficients from Lemma~\ref{lem:nn-Dobrushin-row-col}, and assume
\(
\alpha_{\mathrm{NN}}:=(1-\kappa_P^{\mathrm{NN}})(4-3\kappa_\Phi^{\mathrm{NN}})<1.
\)
For any \(\delta_{\mathrm{NN}}\in(0,1)\), on the event $\mathcal E_{\mathrm{lin}}^{\mathrm{NN}}(\delta_{\mathrm{NN}})$, and for every \(t\ge1\) it holds
\begin{align}\label{eq:nn-belief-lipschitz}
\bE\Big[\|b_t^W-b_t^{W^\star}\|_{\mathrm{TV}}\Big]
\le\frac{1-\alpha_{\mathrm{NN}}^t}{1-\alpha_{\mathrm{NN}}}\,
c_b^{\mathrm{NN}}(W,W^\star;\delta_{\mathrm{NN}}),
\end{align}
where the expectation is taken with respect to \(\hat s_{1:t}\sim P_{W^\star}(\,\cdot\,|a_{0:t-1})\), and
\begin{align}\label{eq:cb-nn-def}
c_b^{\mathrm{NN}}(W,W^\star;\delta_{\mathrm{NN}})
&:=B^{\mathrm{NN}}\|W_\Phi-W_\Phi^\star\|_F
+2\varepsilon_\Phi^{\mathrm{NN}}(m,\delta_{\mathrm{NN}})+(3-2\kappa_\Phi^{\mathrm{NN}})
\Big(\frac{B^{\mathrm{NN}}}{2}\|W_p-W_p^\star\|_F
+\varepsilon_{p}^{\mathrm{NN}}(m,\delta_{\mathrm{NN}})\Big).
\end{align}
In particular, as \(m\to\infty\), this recovers a Lipschitz-type belief perturbation bound in the neural parameters.
\end{corollary}

\begin{proof}
The full proof is given in Appendix \ref{proof:belief-lip-nn}.
\end{proof}

% \paragraph{Proof sketch.}
% The proof follows the same one-step decomposition as Theorem~\ref{thm:param-belief-lip}. The only changes are that the log-linear row-wise kernel perturbation bounds are replaced by Lemma~\ref{lem:nn-kernel-lip}, and the Dobrushin constants are replaced by the neural constants \(\kappa_P^{\mathrm{NN}}\) and \(\kappa_\Phi^{\mathrm{NN}}\) from Lemma~\ref{lem:nn-Dobrushin-row-col}. The resulting recursion has additive term \(c_b^{\mathrm{NN}}(W,W^\star;\delta_{\mathrm{NN}})\) and contraction coefficient \(\alpha_{\mathrm{NN}}\).

\subsection{High-probability time-average belief error}
\label{subsec:belief-highprob}

The expectation bounds in Theorems~\ref{thm:param-belief-lip} and Corollary ~\ref{cor:nn-param-belief-lip} control the belief mismatch at each fixed time. For preference-based reward learning, however, the relevant objects are trajectory-level feature sums, and hence the error
depends on the accumulated belief mismatch along the whole rollout. Therefore, it suffices to control the \emph{time-average} belief error. The following corollary converts the one-step
stability recursion into a high-probability time-average bound via a martingale
argument. This bound will later be union-bounded over all trajectory pairs in the
preference dataset to define a belief-accuracy event for the reward learning analysis.

\begin{corollary}\label{cor:belief-good-condition-cor}
Consider the setting of Theorem \ref{thm:param-belief-lip} and fix $\Theta=(\theta,w)\in\bB_2(\Theta^\star,r_\Theta)$ such that $\|\theta-\theta^\star\|_2\le\delta(\theta)$ and $\|w-w^\star\|_2\le\delta(w)$. For $t=0,1,\dots,T-1$ consider the random variables
\(
X_t :=\bigl\|b_t^\Theta-b_t^{\Theta^\star}\bigr\|_{\mathrm{TV}}.\) For the filtrations $\{\mathcal F_t^-\}_{t\ge 0}$ and $\{\mathcal F_t\}_{t\ge 0}$ defined by
\(
\mathcal F_t^- := \sigma(\hat s_{1:t},a_{0:t-1}),\mathcal F_t := \sigma(\hat s_{1:t},a_{0:t}),
\)
with $\mathcal F_{-1}^-:=\sigma(\varnothing,\Omega)$, the random variable $X_t$ is $\mathcal F_t^-$-measurable.  For any $\delta_b\in(0,1)$, with probability at least $1-\delta_b$,

\begin{align}\label{eq:hp-belief-freedman}
\frac{1}{T}\sum_{t=0}^{T-1}X_t &\le \frac{2c_b\,(\delta(\theta),\delta(w))}{1-\alpha} + \frac{2}{1-\alpha}\sqrt{\frac{2 c_b\,(\delta(\theta),\delta(w)) \log(1/\delta_b)}{T}}+ \frac{\log(1/\delta_b)}{T(1-\alpha)}\left(\frac{4}{3} + \frac{2\alpha}{1-\alpha}\right)\,=:\, \epsilon_b(\delta_b).
\end{align}

Moreover, the same statement holds for the neural-softmax model of
Corollary~\ref{cor:nn-param-belief-lip}. In particular, for any
\(\delta_{\mathrm{NN}}\in(0,1)\), on the linearization event
\(\mathcal E_{\mathrm{lin}}^{\mathrm{NN}}(\delta_{\mathrm{NN}})\) from
Definition~\ref{def:nn-linearized-scores}, if \(
X_t^{\mathrm{NN}}:=\bigl\|b_t^W-b_t^{W^\star}\bigr\|_{\mathrm{TV}},
\) and if \(\alpha_{\mathrm{NN}}<1\), then with conditional probability at least \(1-\delta_b\), equivalently with joint probability at least \(1-(\delta_b+\delta_{\mathrm{NN}})\),

\begin{align}\label{eq:hp-belief-freedman-nn}
\frac{1}{T}\sum_{t=0}^{T-1}X_t^{\mathrm{NN}}
&\le \frac{2c_b^{\mathrm{NN}}(W,W^\star;\delta_{\mathrm{NN}})}{1-\alpha_{\mathrm{NN}}} +\frac{2}{1-\alpha_{\mathrm{NN}}}\sqrt{
\frac{2c_b^{\mathrm{NN}}(W,W^\star;\delta_{\mathrm{NN}})\log(1/\delta_b)}{T}}
\notag\\
&\qquad\qquad+\frac{\log(1/\delta_b)}{T(1-\alpha_{\mathrm{NN}})}
\left(\frac{4}{3} +\frac{2\alpha_{\mathrm{NN}}}{1-\alpha_{\mathrm{NN}}}\right)=:\epsilon_b^{\mathrm{NN}}(\delta_b,\delta_{\mathrm{NN}}),
\end{align}
where \(c_b^{\mathrm{NN}}(W,W^\star;\delta_{\mathrm{NN}})\) is defined in Eq.~\ref{eq:cb-nn-def}.
\end{corollary}

\begin{proof}
The full proof is given in Appendix \ref{proof:col-belief}
\end{proof}

% \paragraph{Proof sketch.} Starting from a one-step conditional drift inequality for $X_t$, define a martingale difference sequence and apply Freedman’s inequality to control deviations of $\sum_t X_t$ from its conditional expectation, yielding the stated time-average bound.
\medskip

\begin{remark}[Learning Beliefs Directly]\label{rem:alternative-belief-approx}
Our analysis is modular with respect to the source of the approximate beliefs. 
Although Sections~\ref{sec:setup}--\ref{sec:belief-stability} instantiate $\tilde b_t=b_t^\Theta$ through Bayesian filtering under a learned POMDP model, one may instead use a direct belief-inference module that maps histories $(\hat s_{1:t},a_{0:t-1})$ to approximate beliefs. Recent examples include Deep Belief Markov Models \cite{DBeliefNN}, which learn belief-transition and belief-inference operators via variational inference, and flow-based recurrent belief models, which use normalizing flows to represent flexible belief distributions \citep{flowBelief}. This provides an alternative approach to bypass explicit POMDP model estimation when convenient. In our reward-learning analysis in the following section, such a module can be substituted for the Bayesian filter whenever it provides a belief-accuracy guarantee of the form
\(
\frac1T\sum_{t=0}^{T-1}\|\tilde b_t-b_t^{\Theta^\star}\|_{\mathrm{TV}}\le \epsilon_b^{\mathrm{approx}}
\). Then the belief-approximation error enters in the downstream reward-learning analysis, only through the trajectory-feature perturbation term, and the corresponding belief-accuracy level can be replaced by $\epsilon_b^{\mathrm{approx}}$ to obtain the same bound on the accumulated feature difference $\|\tilde\phi_i-\phi_i\|_2$. The subsequent Bradley--Terry estimation analysis then proceeds unchanged with this substituted error level.
\end{remark}

%%%%%%%%%%%%%%%%%%%%%%%%%%%%%%
%%%%%%%%%%%%%%%%%%%%%%%%%%%%%%
%%%%%%%%%%%%%%%%%%%%%%%%%%%%%%
%%%%%%%%%%%%%%%%%%%%%%%%%%%%%%
%%%%%%%%%%%%%%%%%%%%%%%%%%%%%%
%%%%%%%%%%%%%%%%%%%%%%%%%%%%%%

\section{POMDP Model Learning via Covariance-Preconditioned Gradient EM}
\label{sec:pomdp-estimation-em}

\medskip\noindent
We return to the log-linear model of Assumption~\ref{as:Model-LogLinear} and learn its transition and observation parameters using covariance-preconditioned gradient EM from incomplete trajectories, where observations and actions are available but the latent state sequence is unobserved. The neural-softmax extension is not considered in the present EM analysis. Following the population-to-sample framework of \citet{balakrishnan2017statistical}, we establish local curvature and gradient-stability properties. Observed-information identifiability provides the gradient-stability control, while Theorem~\ref{thm:param-belief-lip} and the resulting smoothing bounds control the pilot complete-data covariance. Following the weighted-gradient treatment of incomplete-data likelihoods in \citet[Section~10.1.3]{HMMCappe}, we use the population complete-data covariance as a fixed local geometry and replace it in finite samples by a regularized pilot estimate. Supporting results are given in Appendix~\ref{sec:covariance-gradient-em}. Fix an estimation horizon $K\in\bN$, let $Z=(s_{0:K})$ and $Y=(\hat s_{1:K},a_{0:K-1})$, and condition throughout on the fixed action sequence. Define
\begin{align*}
Q_Y(\Theta_1,\Theta_2):=\bE_{Z\mid Y,\Theta_2}\!\left[\log P_{\Theta_1}(Y,Z)\right],\qquad
&Q(\Theta_1,\Theta_2):=\bE_{Y\sim P_{\Theta^\star}}\!\left[Q_Y(\Theta_1,\Theta_2)\right],\qquad
q(\Theta):=Q(\Theta,\Theta^\star).
\end{align*}

\medskip

\begin{assumption}[Complete-data covariance identifiability]
\label{as:complete-data-covariance-identifiability}
The transition and observation parameters are restricted to fixed linear subspaces
\(\mathcal U_\theta\subseteq\bR^{d_\theta}\) and
\(\mathcal U_w\subseteq\bR^{d_w}\) on which the corresponding log-linear
parameterizations are identifiable. All gradients, Hessians, and inverses below are
considered on \(\mathcal U_\theta\times\mathcal U_w\), and the restricted complete-data
covariance satisfies \(
\Sigma_\star\succ0.
\)
\end{assumption}

\begin{assumption}[Observed-information identifiability]\label{as:observed-information-identifiability}
For the fixed action sequence, define the population observed-data log-likelihood and observed information by
\(
\ell_{\mathrm{obs}}(\Theta):=\bE_{Y\sim P_{\Theta^\star}}\!\left[\log P_\Theta(Y)\right],\text{ and }
\mathcal I_{\mathrm{obs}}:=-\nabla^2\ell_{\mathrm{obs}}(\Theta^\star).
\)
On $\mathcal U_\theta\times\mathcal U_w$, assume that $\mathcal I_{\mathrm{obs}}\succ0$, and define
\(
\nu_{\mathrm{obs}}:=\lambda_{\min}\!\left(\Sigma_\star^{-1/2}\mathcal I_{\mathrm{obs}}\Sigma_\star^{-1/2}\right)>0.
\)
\end{assumption}
\medskip

Let $\Sigma_\star=\operatorname{diag}(\Sigma_{\theta,\star},\Sigma_{w,\star})$ be the population complete-data covariance defined in Eq.~\ref{eq:complete-data-covariance}, and consider
\[
\mathbb B_{\Sigma_\star}(\Theta^\star,r):=\left\{\Theta\in\mathbb B_2(\Theta^\star,r_\Theta):\|\Theta-\Theta^\star\|_{\Sigma_\star}\le r\right\}.
\] By Lemma~\ref{lem:local-covariance-curvature}, $\Sigma_\star$ induces the local curvature geometry of $q$ on a sufficiently small covariance ball. Throughout this section, we take \( r\le\sqrt{\lambda_{\min}(\Sigma_\star)}\,r_\Theta,
\) so that $\|\Theta-\Theta^\star\|_{\Sigma_\star}\le r$ implies $\Theta\in\mathbb B_2(\Theta^\star,r_\Theta)$.

Our population analysis follows the first-order EM contraction framework of \citet[Theorem~1]{balakrishnan2017statistical}, specialized to controlled log-linear POMDPs and the complete-data covariance geometry. Assumption~\ref{as:observed-information-identifiability} yields the local covariance-weighted gradient-stability behavior characterized in Lemma~\ref{lem-3condition-1}. Together with the local covariance curvature above, this gives the following population contraction result.

\medskip

\begin{proposition}[Population contraction of covariance-preconditioned gradient EM]\label{prop:covariance-gradient-em-convergence}
Let $\varepsilon\in(0,1)$, let $r_\varepsilon>0$ be the local radius in Lemma~\ref{lem:local-covariance-curvature}, and suppose $r\le r_\varepsilon$. Let $\eta>0$ and define
\(
\rho_\eta:=\max\{|1-\eta((1-\varepsilon))|,\ |1-\eta(1+\varepsilon)|\}.
\) Consider the following parameter update rule
\begin{equation}\label{eq:covariance-gradient-em-update}
\Theta^{(t+1)}
=\Theta^{(t)}+\eta\Sigma_\star^{-1}\nabla_1Q(\Theta^{(t)},\Theta^{(t)}).
\end{equation}
If $\Theta^{(0)}\in\mathbb B_{\Sigma_\star}(\Theta^\star,r)$ and \(
\kappa_{\mathrm{EM}}:=\rho_\eta+\eta L_Q^\Sigma(r)<1,\) where \(L_Q^\Sigma(r)\) is the local covariance-weighted gradient-stability coefficient defined in Definition~\ref{def:covariance-gradient-stability}, then all iterates remain in $\mathbb B_{\Sigma_\star}(\Theta^\star,r)$ and, for every $t\ge0$,
\begin{equation}\label{eq:covariance-gradient-em-contraction}
\|\Theta^{(t)}-\Theta^\star\|_{\Sigma_\star}
\le\kappa_{\mathrm{EM}}^t\|\Theta^{(0)}-\Theta^\star\|_{\Sigma_\star}.
\end{equation}

\end{proposition}
\begin{proof}
The full proof is given in  Appendix~\ref{proof:prop-covgradEM}.
\end{proof}
For the finite-sample procedure, let the pilot sample $\{Y_0^{(i)}\}_{i=1}^{M_0}$ and the estimation sample $\{Y^{(i)}\}_{i=1}^{M}$ be independent collections of trajectories generated under $P_{\Theta^\star}$. Using a fixed pilot parameter $\widetilde\Theta\in\mathbb B_{\Sigma_\star}(\Theta^\star,r)$, obtained independently of both the pilot and estimation samples, let $\widehat\Sigma_{M_0}(\widetilde\Theta)$ be the pilot complete-data covariance estimator in Eq.~\ref{eq:pilot-complete-data-covariance}, and define
\(
\widehat\Sigma_\lambda:=\widehat\Sigma_{M_0}(\widetilde\Theta)+\lambda I,\text{ and }
Q_M(\Theta_1,\Theta_2):=\frac1M\sum_{i=1}^{M}Q_{Y^{(i)}}(\Theta_1,\Theta_2),
\)
where $\lambda>0$ and $\widehat\Sigma_\lambda$ is kept fixed throughout the iterations.

\medskip

\begin{remark}[Pilot initialization and preconditioning]\label{rem:pilot-parameter}
Our analysis assumes a preliminary estimator $\widetilde\Theta$ in the local neighborhood of $\Theta^\star$, following \citet{balakrishnan2017statistical}. Spectral or moment-based methods can provide such initializers \citep{mattila2017identification}. Given $\widetilde\Theta$, an independent pilot sample is used to estimate the complete-data geometry and form the fixed preconditioner $\widehat\Sigma_\lambda(\widetilde\Theta)$, after which covariance-preconditioned gradient EM is run on the estimation sample. The use of complete-data curvature for gradient preconditioning is classical in incomplete-data optimization; see \citet[Section~10.1.3]{HMMCappe}.
\end{remark}

\medskip

\begin{definition}[Finite-sample EM concentration events]\label{def:finite-sample-em-deviations}
For $\delta_{\mathrm{cov}},\delta_Q\in(0,1)$, let $\varepsilon_{\mathrm{cov}}$ and $\epsilon_Q^{\mathrm{unif},\Sigma}(M,\delta_Q)$ be deterministic deviation bounds satisfying
\begin{align}
&\bP\!\left(\left\|\Sigma_\star^{-1/2}(\widehat\Sigma_\lambda-\Sigma_\star)\Sigma_\star^{-1/2}\right\|_{\mathrm{op}}\le\varepsilon_{\mathrm{cov}}\right)\ge1-\delta_{\mathrm{cov}},\label{eq:empirical-covariance-deviation-main}\\
&\bP\!\left(\sup_{\Theta\in\mathbb B_{\Sigma_\star}(\Theta^\star,r)}
\|\nabla_1Q_M(\Theta,\Theta)-\nabla_1Q(\Theta,\Theta)\|_{\Sigma_\star^{-1}}
\le\epsilon_Q^{\mathrm{unif},\Sigma}(M,\delta_Q)\right)\ge1-\delta_Q.\label{eq:uniform-empirical-gradient-deviation-main}
\end{align}
An explicit choice of $\epsilon_Q^{\mathrm{unif},\Sigma}(M,\delta_Q)$ is given in Lemma~\ref{lem:explicit-uniform-empirical-Q-gradient}.
\end{definition}

\begin{theorem}[Finite-sample covariance-preconditioned gradient EM]\label{thm:finite-sample-covariance-gradient-em}
Let $\varepsilon\in(0,1)$, $r_\varepsilon>0$ be the local radius in Lemma~\ref{lem:local-covariance-curvature}, $r\le r_\varepsilon$, $\varepsilon_{\mathrm{cov}}<1$, and $\eta>0$. Also, assume the conditions of Theorem~\ref{thm:param-belief-lip} holds. Define
\[
\widehat\rho_\eta:=\max\left\{\left|1-\eta\frac{(1-\varepsilon)}{1+\varepsilon_{\mathrm{cov}}}\right|,\left|1-\eta\frac{1+\varepsilon}{1-\varepsilon_{\mathrm{cov}}}\right|\right\},\qquad
\widehat\kappa_{\mathrm{EM}}:=\widehat\rho_\eta+\frac{\eta L_Q^\Sigma(r)}{1-\varepsilon_{\mathrm{cov}}}.
\]

Consider the update rule
\begin{equation}\label{eq:finite-sample-covariance-gradient-em-update}
\widehat\Theta^{(t+1)}
=\widehat\Theta^{(t)}+\eta\widehat\Sigma_\lambda^{-1}
\nabla_1Q_M\bigl(\widehat\Theta^{(t)},\widehat\Theta^{(t)}\bigr).
\end{equation}
Suppose $\widehat\kappa_{\mathrm{EM}}<1$ and
\begin{align}
\|\widehat\Theta^{(0)}-\Theta^\star\|_{\Sigma_\star}
&\le r\sqrt{\frac{1-\varepsilon_{\mathrm{cov}}}{1+\varepsilon_{\mathrm{cov}}}},
\label{eq:finite-sample-em-initialization}\\
\epsilon_Q^{\mathrm{unif},\Sigma}(M,\delta_Q)
&\le\frac{(1-\widehat\kappa_{\mathrm{EM}})(1-\varepsilon_{\mathrm{cov}})r}{\eta}.
\label{eq:finite-sample-em-deviation-condition}
\end{align}
Then, with probability at least $1-\delta_{\mathrm{cov}}-\delta_Q$, all iterates remain in $\mathbb B_{\Sigma_\star}(\Theta^\star,r)$ and, for every $t\ge0$,
\begin{align}
\|\widehat\Theta^{(t)}-\Theta^\star\|_{\Sigma_\star}
&\le\sqrt{\frac{1+\varepsilon_{\mathrm{cov}}}{1-\varepsilon_{\mathrm{cov}}}}\,
\widehat\kappa_{\mathrm{EM}}^t\|\widehat\Theta^{(0)}-\Theta^\star\|_{\Sigma_\star}+\frac{\eta(1-\widehat\kappa_{\mathrm{EM}}^t)}
{(1-\widehat\kappa_{\mathrm{EM}})(1-\varepsilon_{\mathrm{cov}})}
\epsilon_Q^{\mathrm{unif},\Sigma}(M,\delta_Q).
\label{eq:finite-sample-em-contraction-population}
\end{align}

\end{theorem}
\paragraph{Proof sketch.}
On the events in Definition~\ref{def:finite-sample-em-deviations}, Lemma~\ref{lem:empirical-covariance-geometry-transfer} transfers the local curvature and gradient-stability bounds to the empirical covariance geometry. Applying the same one-step argument as in Proposition~\ref{prop:covariance-gradient-em-convergence}, together with the uniform empirical-gradient error, yields an affine recursion with coefficient $\widehat\kappa_{\mathrm{EM}}$.Unrolling the recursion gives Eq.~\ref{eq:finite-sample-em-contraction-population}; see Appendix~\ref{sec:covariance-gradient-em}.

% \medskip
% \begin{remark}[Finite-sample EM complexity]\label{rem:em-sample-complexity}
% Suppose $\varepsilon_{\mathrm{cov}}\le\bar\varepsilon_{\mathrm{cov}}<1$ and $\widehat\kappa_{\mathrm{EM}}<1$. Let \(
% m_{\mathrm{cov}}:=\bar\varepsilon_{\mathrm{cov}}-L_{\mathrm{pilot}}^\Sigma\|\widetilde\Theta-\Theta^\star\|_{\Sigma_\star}-\lambda\|\Sigma_\star^{-1}\|_{\mathrm{op}}>0.
% \)
% Then, up to logarithmic factors,
% \[
% M=\widetilde O\!\left(C_{\mathrm{EM}}\frac{d_\Theta+\log(1/\delta_Q)}{\epsilon_\Theta^2}\right),\qquad
% M_0=\widetilde O\!\left(\frac{K^2\chi_\Sigma^2\log(d_\Theta/\delta_{\mathrm{cov}})}{m_{\mathrm{cov}}^2}\right),\qquad
% T_{\mathrm{EM}}=O\!\left(\frac{\log(1/\epsilon_\Theta)}{\log(1/\widehat\kappa_{\mathrm{EM}})}\right),
% \]
% are sufficient to attain parameter error of order $\epsilon_\Theta$, for some problem-dependent constant $C_{\mathrm{EM}}>0$, subject to the stated local initialization and pilot conditions.
% \end{remark}

The covariance-preconditioned EM procedure is summarized in Algorithm~\ref{alg:covariance-preconditioned-gradient-em} in Appendix~\ref{sec:covariance-gradient-em}; the pilot preconditioner is computed once and kept fixed throughout the iterations.

\begin{remark}[From POMDP estimation to reward-learning error]\label{rem:em-belief-reward-interface}
The EM guarantees provide the parameter-error control required by the preceding belief-stability analysis. Indeed, Corollary~\ref{cor:covariance-filter-stability} gives, for every filtering time $h$ and every parameter $\Theta$ in the local covariance ball,
\(
\bE\!\left[\|b_h^\Theta-b_h^{\Theta^\star}\|_{\mathrm{TV}}\right]
\le\frac{1}{2}L_{\alpha,h}^\Sigma\|\Theta-\Theta^\star\|_{\Sigma_\star}.
\)
Consequently, the population iterates satisfy
\(
\bE\!\left[\|b_h^{\Theta^{(j)}}-b_h^{\Theta^\star}\|_{\mathrm{TV}}\right]
\le\frac{1}{2}L_{\alpha,h}^\Sigma\kappa_{\mathrm{EM}}^j
\|\Theta^{(0)}-\Theta^\star\|_{\Sigma_\star},
\)
while, on the event of Theorem~\ref{thm:finite-sample-covariance-gradient-em}, the corresponding finite-sample belief error is obtained by substituting the right-hand side of Eq.~\ref{eq:finite-sample-em-contraction-population}. If the model-estimation trajectories are independent of the preference data, the belief guarantees may be applied conditionally on the learned parameter and then inserted into Corollary~\ref{cor:belief-good-condition-cor} and Theorem~\ref{thm:main-mu}. 
\end{remark}

%%%%%%%%%%%%%%%%%%%%%%%%%%%%%%%%%%%%%%%%%%%%%%%%%%%%%%%%%%%%%%%%%%%%%%%%%%%%
\FloatBarrier

\section{Preference-Based Reward Learning Under Belief Error}
\label{sec:rlhf}

\noindent
In this section, we focus on learning a reward model from pairwise preferences when trajectories are represented in the belief-MDP. Note that the model mismatch affects reward learning through the induced belief features, since preference labels come from rollouts, but the learner’s feature construction relies on $b_t^\Theta$ rather than the oracle beliefs $b_t^{\Theta^\star}$. Accordingly, this section formalizes the Bradley--Terry preference model \citep{Bradley1952, Wirth2017} over belief-based trajectory features, and quantifies how the belief error bounds from Section~\ref{sec:belief-stability} translate into an explicit perturbation term in reward estimation, alongside the usual finite-sample statistical error. For the finite-sample estimator of Section~\ref{sec:pomdp-estimation-em}, the required belief-error input is obtained by combining Eq.~\ref{eq:finite-sample-em-contraction-population} with Remark~\ref{rem:em-reward-interface}. The reward analysis itself remains modular and continues to apply to any learned model or belief estimator satisfying the same belief-accuracy condition.

\noindent

\subsection{Reward model and preference data}
\label{subsec:reward-pref-model}

\begin{assumption}[Reward model structure and preference data for belief-MDP]\label{assum:reward-rlhf}
\leavevmode
\begin{enumerate}
\item \textbf{Linear latent reward.}
For all $(s,a)\in\bS\times\bA$, the reward is linear:
\(r(s,a)=\phi_r(s,a)^\top \mu^\star\) with the corresponding feature map satisfying \(\|\phi_r(s,a)\|_2\le B_r < \infty\),
and the unknown reward parameter $\mu^\star\in\bB_2(r_\mu)\subset\bR^{d}$.

\item \textbf{Belief feature map.}
Define
\(\phi_b(b,a):=\sum_{s\in\bS}b(s)\,\phi_r(s,a)\) for \((b,a)\in\Delta_{\bS}\times\bA\).
Then \(r_b(b,a) :=\bE_{s\sim b}\big[r(s,a)\big]
=\sum_{s\in\bS} b(s)\,\phi_r(s,a)^\top\mu^\star
=\phi_b(b,a)^\top\mu^\star.\)

\item \textbf{Preference realizability.}
Let $\{(\tau_i^{(1)},\tau_i^{(2)},y_i)\}_{i=1}^{N_{\mathrm{HF}}}$ be $N_{\mathrm{HF}}$ independent preference comparisons with fixed trajectory horizon $T\in\bN$. Here $y_i\in\{0,1\}$ and $\sigma(z)=1/(1+e^{-z})$. Let $\phi_i\in\bR^{d}$ be the clean trajectory-level feature difference defined in Eq.~\ref{eq:accu-feature-exact}. Then, for every $i$,
\(\bP(y_i=1\mid \phi_i)=\sigma(\phi_i^\top\mu^\star).\)
\end{enumerate}
\end{assumption}

\begin{remark}
Since $\phi_b$ is linear in $b$, it inherits boundedness:
\(
\|\phi_b(b,a)\|_2 \le \sum_{s\in\bS} b(s)\,\|\phi_r(s,a)\|_2 \le B_r, \forall (b,a)\in\Delta_{\bS}\times\bA.
\) Moreover, for any $b,b'\in\Delta_{\bS}$ and all $a\in\bA$,
\begin{align}\label{eq:belief-feature-lip}
\|\phi_b(b,a)&-\phi_b(b',a)\|_2
\le2B_r\|b-b'\|_{\mathrm{TV}}.
\end{align}
\end{remark}

\begin{remark}[Generalization to non-linear rewards]
Linear reward parameterizations are standard in preference-based RLHF and are widely used as a tractable baseline (see, e.g., \cite{ZhuPrincipledComparisons, Du2024Exploration-DrivenUtilization, cen2025valueincentivized}). 
Our analysis isolates how model-mismatch propagates through belief errors into perturbed trajectory features, which in turn affects reward learning in Theorem \ref{thm:main-mu}. 
More general reward function classes can be accommodated by introducing an approximation error term (relative to the linear class) and carrying it through the final bound. 
In particular, over-parameterized neural networks are a plausible choice as discussed in Appendix ~\ref{sec:nn-extension} for POMDP model approximation.
\end{remark}
 
\subsection{Clean vs.\ perturbed trajectory features and MLE}
\label{subsec:features-mle}

Throughout the following, we work on the belief state space $\Delta_{\bS}$ and use the belief feature map
$\phi_b$. Let $N_{\mathrm{HF}}$ denote the number of \emph{independent} preference comparisons, and consider a dataset
\(
\bigl\{(\tau_i^{(1)},\tau_i^{(2)},y_i)\bigr\}_{i=1}^{N_{\mathrm{HF}}}
\),
where $y_i\in\{0,1\}$ is the observed preference label and $y_i=1$ indicates a preference for $\tau_i^{(1)}$ over $\tau_i^{(2)}$. For each $i\in\{1,\dots,N_{\mathrm{HF}}\}$ and $j\in\{1,2\}$, let the observed history be \(Y_i^{(j)} \;:=\; \bigl(\hat s_{i,1:T}^{(j)},\,a_{i,0:T-1}^{(j)}\bigr)
\) and define the corresponding belief sequences $(b_{i,h}^{(j),\Theta})_{h=0}^{T-1}$ and
$(b_{i,h}^{(j),\Theta^\star})_{h=0}^{T-1}$ as the filtering distributions computed under
$\Theta$ and $\Theta^\star$, respectively, on the \emph{same} history $Y_i^{(j)}$ via the Bayesian update
Eq.~\ref{eq:belief-formula}. We then write the history-indexed belief trajectories as
\(
\tau_i^{(j),\Theta}:=\bigl(b_{i,h}^{(j),\Theta},\,a_{i,h}^{(j)}\bigr)_{h=0}^{T-1}, \tau_i^{(j),\Theta^\star}:=\bigl(b_{i,h}^{(j),\Theta^\star},\,a_{i,h}^{(j)}\bigr)_{h=0}^{T-1}
\), and the action sequence $(a_{i,h}^{(j)})_{h=0}^{T-1}$ is treated as observed. Throughout this work, we assume that all trajectories have a fixed horizon $T<\infty$.

 \begin{remark}[Single realized action sequence vs.\ fixed-history analysis]\label{rem:one-action-seq}
In an RLHF interaction there is only one physically realized history $(\hat s_{1:T},a_{0:T-1})$.
Our stability guarantees compare $b_t^\Theta$ and $b_t^{\Theta^\star}$ computed on this \emph{same realized history}
(i.e., conditioning on the realized action sequence and coupling on the realized observations).
This fixed-history coupling is exactly what controls feature perturbations in Eq.~\ref{eq:accu-feature-exact}--Eq.~\ref{eq:accu-feature-perturbed}.
Analyzing additional divergence caused by two models inducing different closed-loop action sequences is beyond scope.
\end{remark}

\paragraph{Reward parameter learning from human preferences.}
We consider the estimation of the reward parameter $\mu^\star\in\bR^d$ from human preference data under a trajectory-level Bradley--Terry model \cite{Bradley1952}. Human feedback is provided as pairwise trajectory comparisons $\{(\tau_i^{(1)},\tau_i^{(2)},y_i)\}_{i=1}^{N_{\mathrm{HF}}}$. For each comparison, we define the trajectory-level exact and perturbed accumulated feature differences as

\begin{equation}\label{eq:accu-feature-exact}
\phi_i:=\sum_{h=0}^{T-1}\phi_b \bigl(b_{i,h}^{(1),\Theta^\star},a_{i,h}^{(1)}\bigr)-\sum_{h=0}^{T-1}\phi_b \bigl(b_{i,h}^{(2),\Theta^\star},a_{i,h}^{(2)}\bigr),
\end{equation}
\begin{equation}\label{eq:accu-feature-perturbed}
\tilde\phi_i:=\sum_{h=0}^{T-1} \phi_b \bigl(b_{i,h}^{(1),\Theta}, a_{i,h}^{(1)}\bigr)
-\sum_{h=0}^{T-1} \phi_b \bigl(b_{i,h}^{(2),\Theta}, a_{i,h}^{(2)}\bigr).
\end{equation}
The feature vector $\phi_i$ corresponds to the trajectory-level reward feature difference evaluated along the true belief process $(b_{i,h}^{(j),\Theta^\star})_{h}$.
The perturbed feature $\tilde\phi_i$ is computed using approximate beliefs $(b_{i,h}^{(j),\Theta})_{h}$, where the differences arise from model mismatch. Lemma \ref{lem:Delta-Bound} bounds  $\|\phi_i-\tilde\phi_i\|_2$ on the event $\mathcal E_{b}^{\mathrm{RLHF}}(\delta_b)$ in the Appendix. We estimate the reward parameter using the constrained maximum likelihood estimator

\begin{align}\label{eq:reward-mle}
\tilde\mu\;&:=\;\argmin_{\|\mu\|_2\le r_\mu}\Bigg\{-\sum_{i=1}^{N_{\mathrm{HF}}}\Big[y_i\log\sigma(\tilde\phi_i^\top\mu)
+(1-y_i)\log \bigl(1-\sigma(\tilde\phi_i^\top\mu)\bigr)\Big]\Bigg\}.
\end{align}
In the following analysis, we study the statistical behavior of $\tilde\mu$ and its robustness to trajectory-level perturbations arising from belief approximation. For clarity, the complete computational pipeline from recursive belief filtering to the optimization of the estimator in Eq.~\ref{eq:reward-mle} is summarized in Algorithm \ref{alg:reward-learning}.

\begin{remark}[Belief-based features]
Classical Bradley--Terry reward learning typically assumes access to true states and uses state-based trajectory features \citep{Wirth2017, ZhuPrincipledComparisons, Du2024Exploration-DrivenUtilization, cen2025valueincentivized}. Under partial observability, we instead construct $\phi_i$ and $\tilde\phi_i$ from belief--action pairs; hence model mismatch affects reward learning through belief-induced feature perturbations. The resulting POMDP framework therefore allows the analysis below to account explicitly for simultaneous transition and observation model misspecification through the induced belief error and to upper bound its downstream contribution to reward-estimation error. This should be viewed as an extension to partially observed settings rather than as a claim of outperforming fully observed formulations.
\end{remark}

\begin{algorithm}[t]
\caption{Preference-Based Reward Learning via Learned Beliefs}
\label{alg:reward-learning}
\begin{algorithmic}[1]
\STATE \textbf{Input:}
Horizon $T$;
preference data $\mathcal{D} = \{(Y_i^{(1)},Y_i^{(2)},y_i)\}_{i=1}^{N_{\mathrm{HF}}}$ with
$Y_i^{(j)}=(\hat s^{(j)}_{i,1:T},a^{(j)}_{i,0:T-1})$;
labels $y_i\in\{0,1\}$;
learned POMDP model $\Theta=(\theta,w)$; belief feature map $\phi_b$; constraint radius $r_\mu$.

\noindent\textit{Phase 1: Reconstruct Beliefs and Features}
\FOR{$i=1,\dots,N_{\mathrm{HF}}$}
    \FOR{$j\in\{1,2\}$}
        \STATE $b^{(j),\Theta}_{i,0}\gets \nu_0$
        \FOR{$h=0,\dots,T-1$}
            \STATE $b^{(j),\Theta}_{i,h+1}\gets \text{BeliefUpdate}\!\left(
            b^{(j),\Theta}_{i,h},a^{(j)}_{i,h},\hat s^{(j)}_{i,h+1};\Theta\right)$
            \hfill{\small (Eq.~\ref{eq:belief-formula})}
        \ENDFOR
    \ENDFOR
    \STATE Compute perturbed feature difference $\tilde\phi_i$ (Eq.~\ref{eq:accu-feature-perturbed}):
    \STATE $\tilde\phi_i \gets \sum_{h=0}^{T-1}\phi_b\!\left(b^{(1),\Theta}_{i,h},a^{(1)}_{i,h}\right)
                 \;-\; \sum_{h=0}^{T-1}\phi_b\!\left(b^{(2),\Theta}_{i,h},a^{(2)}_{i,h}\right)$
\ENDFOR

\noindent\textit{Phase 2: Maximum Likelihood Estimation}
\STATE Solve for $\tilde\mu$ (Eq.~\ref{eq:reward-mle}):
 $$\tilde\mu \gets \argmin_{\|\mu\|_2\le r_\mu}\;
-\sum_{i=1}^{N_{\mathrm{HF}}}\Big[y_i\log\sigma(\tilde\phi_i^\top\mu)
+(1-y_i)\log\big(1-\sigma(\tilde\phi_i^\top\mu)\big)\Big]$$

\STATE \textbf{Output:} Estimated reward parameter $\tilde\mu$
\end{algorithmic}
\end{algorithm}

\begin{remark}[Computational scaling]\label{rem:computational-scaling}
For finite-state models, filtering over horizon $T$ costs $\mathcal O(T|\bS|^2)$, and a forward--backward pass over a length-$K$ trajectory costs $\mathcal O(K|\bS|^2)$. Thus, belief reconstruction for $2N_{\mathrm{HF}}$ preference trajectories costs $\mathcal O(N_{\mathrm{HF}}T|\bS|^2)$, while one batch gradient-EM iteration over $M$ trajectories costs $\mathcal O(MK|\bS|^2)$, excluding parameter-gradient and linear-algebra costs.
For larger or continuous state spaces, computational costs can be mitigated by exploiting sparse, low-rank, or factored model structure (e.g., \citealp{ghahramani1997factorial}), or through approximate filtering and smoothing methods, including particle-based approaches; see Remark~\ref{rem:alternative-belief-approx}.

\end{remark}

\subsection{Belief accuracy event for preference learning}
\label{subsec:belief-event}
In this subsection we lift the time-average belief error guarantee from
Corollary~\ref{cor:belief-good-condition-cor} to the full preference dataset.
Note that we do \emph{not} need a uniform-in-time control such as
$\sup_{0\le h\le T-1}\|b_h^\Theta-b_h^{\Theta^\star}\|_{\mathrm{TV}}$.
Since Bradley--Terry comparisons are driven by accumulated trajectory-level  feature differences (\ref{eq:accu-feature-exact} and \ref{eq:accu-feature-perturbed}), the relevant quantity is the \emph{per-trajectory time-average} belief
mismatch, which then translates into a bound on the induced feature perturbation.
We therefore define a high-probability event on which every rollout used to form the
comparisons satisfies a small average belief error.

\textbf{Belief accuracy event.}
Consider preference comparisons $\{(\tau_i^{(1)},\tau_i^{(2)},y_i)\}_{i=1}^{N_{\mathrm{HF}}}$. For any confidence level $\delta_b\in(0,1)$, define the belief-accuracy events for $i\in\{1,\dots,N_{\mathrm{HF}}\},\ j\in\{1,2\}$ as
\begin{equation}\label{eq:Eb-RLHF-uniform}
\mathcal E_{b,i}^{(j)}(\delta_b)\;:=\;\left\{\frac{1}{T}\sum_{h=0}^{T-1}
\bigl\|b_{i,h}^{(j),\Theta}-b_{i,h}^{(j),\Theta^\star}\bigr\|_{\mathrm{TV}}
\ \le\ \epsilon_b(\delta_b)\right\},\quad \mathcal E_{b}^{\mathrm{RLHF}}(\delta_b)\;:=\;\bigcap_{i=1}^{N_{\mathrm{HF}}}\ \bigcap_{j\in\{1,2\}}
\mathcal E_{b,i}^{(j)} \left(\frac{\delta_b}{2N_{\mathrm{HF}}}\right).
\end{equation}
By Corollary~\ref{cor:belief-good-condition-cor} and a union bound,
\(
\bP\bigl(\mathcal E_{b}^{\mathrm{RLHF}}(\delta_b)\bigr)\ \ge\ 1-\delta_b.
\)
Throughout the following analysis, we work on the event $\mathcal E_{b}^{\mathrm{RLHF}}(\delta_b)$ for log-linear POMDP model and use the shorthand
\begin{align}\label{eq:eps-b-bar-def}
\epsilon_b(\delta_b, 2N_{\mathrm{HF}}):=\;\epsilon_b \left(\frac{\delta_b}{2N_{\mathrm{HF}}}\right)&=\frac{2c_b\,(\delta(\theta),\delta(w))}{1-\alpha}+ \frac{2}{1-\alpha}\sqrt{\frac{2 c_b\,(\delta(\theta),\delta(w)) \log(2N_{\mathrm{HF}}/\delta_b)}{T}}\notag\\[6pt]
&\qquad\qquad+\frac{\log(2N_{\mathrm{HF}}/\delta_b)}{T(1-\alpha)}\left(\frac{4}{3}+ \frac{2\alpha}{1-\alpha}\right).
\end{align}
That is, $\epsilon_b(\delta_b, 2N_{\mathrm{HF}})$ is given by Eq.~\ref{eq:hp-belief-freedman} with the substitution
$\log\bigl(2N_{\mathrm{HF}}/\delta_b\bigr)$. Analogously, for the neural-softmax POMDP model under the setting of Corollary~\ref{cor:nn-param-belief-lip}, fix \(\delta_{\mathrm{NN}}\in(0,1)\). On the event \(\mathcal E_{\mathrm{lin}}^{\mathrm{NN}}(\delta_{\mathrm{NN}})\), define \(\mathcal E_{b}^{\mathrm{RLHF},\mathrm{NN}}(\delta_b,\delta_{\mathrm{NN}})\) similarly to Eq.~\ref{eq:Eb-RLHF-uniform}, with \(b_{i,h}^{(j),\Theta}\) and \(b_{i,h}^{(j),\Theta^\star}\) replaced by \(b_{i,h}^{(j),W}\) and \(b_{i,h}^{(j),W^\star}\), respectively, and with \(\epsilon_b\) replaced by the neural bound \(\epsilon_b^{\mathrm{NN}}\). More explicitly, we use the shorthand
\[
\epsilon_b^{\mathrm{NN}}(\delta_b,2N_{\mathrm{HF}},\delta_{\mathrm{NN}})
:=\epsilon_b^{\mathrm{NN}}\!\left(\frac{\delta_b}{2N_{\mathrm{HF}}},\delta_{\mathrm{NN}}\right),
\]
where \(\epsilon_b^{\mathrm{NN}}(\cdot,\delta_{\mathrm{NN}})\) is defined in Eq.~\ref{eq:hp-belief-freedman-nn} using the linearization-error levels
\(\varepsilon_p^{\mathrm{NN}}(m,\delta_{\mathrm{NN}})\) and
\(\varepsilon_\Phi^{\mathrm{NN}}(m,\delta_{\mathrm{NN}})\). By the neural high-probability belief bound and a union bound over the \(2N_{\mathrm{HF}}\) trajectories, together with \(\mathcal E_{\mathrm{lin}}^{\mathrm{NN}}(\delta_{\mathrm{NN}})\), we have
\[
\bP\!\left(\mathcal E_{b}^{\mathrm{RLHF},\mathrm{NN}}(\delta_b,\delta_{\mathrm{NN}})\right)
\ge 1-(\delta_b+\delta_{\mathrm{NN}}).
\]

\begin{remark}[Preference realizability and conditional independence]\label{remark:label-cond-indep}
Assumption~\ref{assum:reward-rlhf} (item~3) assumes a correctly specified Bradley--Terry preference model, where each label $y_i$ is generated from an independent trajectory pair and is conditionally independent given the corresponding clean feature difference $\phi_i$, with \(
\bP(y_i=1\mid \phi_i)=\sigma(\phi_i^\top\mu^\star).
\) This provides the link from trajectory-level features to observable preference feedback. (See, also Proof \ref{lem:clean-like-grad-bound})
\end{remark}

\subsection{Reward estimation guarantee}\label{subsec:main-reward-thm}

We now state our main theorem for estimating the Bradley--Terry reward parameter
$\mu^\star$. Our analysis is carried out in the \emph{local geometry} induced by the
empirical design covariance, which captures the curvature of the empirical MLE objective. Accordingly, we measure error in the
covariance-weighted norm $\|\cdot\|_{\tilde\Sigma+\zeta I}$, (See Definition \ref{def:clean-pert-cov}), which is a natural
scale for evaluating the parameter optimality gap. However, our learner features are
\emph{perturbed} through belief mismatch, so the usual self-normalized arguments \cite{ZhuPrincipledComparisons, NIPS2011_e1d5be1c} cannot be applied directly to the gradient formed with $\tilde\phi_i$.
This introduces an additional technical difficulty; see Remark~\ref{remark:why-not-exact-cov}.

\begin{definition}[Clean and perturbed empirical covariances]\label{def:clean-pert-cov}
Given trajectory-level feature differences $\{\phi_i\}_{i=1}^{N_{\mathrm{HF}}}$ and their perturbed counterparts
$\{\tilde\phi_i\}_{i=1}^{N_{\mathrm{HF}}}$, define the corresponding empirical covariance matrices by
\begin{equation}\label{eq:sigma-cov-hf}
\Sigma \;:=\; \frac{1}{N_{\mathrm{HF}}}\sum_{i=1}^{N_{\mathrm{HF}}}\phi_i\phi_i^\top,
\qquad
\tilde\Sigma \;:=\; \frac{1}{N_{\mathrm{HF}}}\sum_{i=1}^{N_{\mathrm{HF}}}\tilde\phi_i\tilde\phi_i^\top.
\end{equation}
\end{definition}

\begin{theorem}[Reward parameter estimation under belief error]\label{thm:main-mu}
Assume the conditions of Theorem~\ref{thm:param-belief-lip} and Assumption \ref{assum:reward-rlhf}. Let $\Theta=(\theta,w)\in\bB_2(\Theta^\star,r_\Theta)$ and let
$\{(Y_i^{(1)},Y_i^{(2)},y_i)\}_{i=1}^{N_{\mathrm{HF}}}$ be $N_{\mathrm{HF}}$ independent preference comparisons. Form perturbed trajectory feature differences $\{\tilde\phi_i\}_{i=1}^{N_{\mathrm{HF}}}$ via Eq.~\ref{eq:accu-feature-perturbed}
using beliefs $(b_{i,h}^{(j),\Theta})_{h=0}^{T}$ computed from Eq.~\ref{eq:belief-formula}, and let
$\tilde\mu$ be the constrained Bradley--Terry MLE in Eq.~\ref{eq:reward-mle}.
Fix $\delta_b,\delta_c\in(0,1)$. Then, for any factor $c_{\zeta}>1$, with probability at least $1-(\delta_b+\delta_c)$,
the estimation error satisfies

\begin{align}
\|\tilde\mu-\mu^\star\|_{\tilde\Sigma+\zeta I}
\le &\frac{2\sqrt{c_{\zeta}}}{\rho\sqrt{N_{\mathrm{HF}}(c_{\zeta}-1)}}\sqrt{d\log \Big(1+\frac{4T^2 B_r^2}{\zeta d}\Big)+2\log \Big(\frac{1}{\delta_c}\Big)} \notag\\[8pt]
&\qquad\qquad+ \frac{8T B_r\,\epsilon_b(\delta_b, 2N_{\mathrm{HF}})}{\rho\sqrt{\zeta }}\, \Bigl(1+\tfrac12\,T B_r\,r_\mu\Bigr)+2r_\mu\sqrt{\zeta }:=\,\epsilon(\delta_b,\delta_c,\zeta),
\label{eq:main_mu_bound}
\end{align}
where $\rho\,=\;1/\big(2+\exp(2T B_r r_\mu)+\exp(-2T B_r r_\mu)\big)$, and if $\epsilon_b(\delta_b,2N_{\mathrm{HF}})>0$, let $\zeta:=c_{\zeta} 16T^2B_r^2\epsilon_b(\delta_b, 2N_{\mathrm{HF}})(1+\epsilon_b(\delta_b, 2N_{\mathrm{HF}}))$, otherwise take any $\zeta>0$.

The same conclusion holds for the neural-softmax POMDP model of
Corollary~\ref{cor:nn-param-belief-lip}. In that case, additionally fix \(\delta_{\mathrm{NN}}\in(0,1)\). With probability at least
\(1-(\delta_b+\delta_c+\delta_{\mathrm{NN}})\), the bound
Eq.~\ref{eq:main_mu_bound} holds after replacing
\(b_{i,h}^{(j),\Theta},b_{i,h}^{(j),\Theta^\star}\) by
\(b_{i,h}^{(j),W},b_{i,h}^{(j),W^\star}\), respectively, and replacing
\(\epsilon_b(\delta_b,2N_{\mathrm{HF}})\) everywhere, including in the definition of
\(\zeta\), by \(
\epsilon_b^{\mathrm{NN}}(\delta_b,2N_{\mathrm{HF}},\delta_{\mathrm{NN}}).
\)

\end{theorem}

\paragraph{Proof sketch}
We show $\tilde L$ is $\rho$-strongly convex on $\bB_2(r_\mu)$ in the $\tilde\Sigma$-geometry, and then control the estimation error by bounding the gradient at $\mu^\star$ \(
\|\tilde\mu-\mu^\star\|_{\tilde\Sigma+\zeta I} \;\le\; \frac{2}{\rho}\,\|\nabla \tilde L(\mu^\star)\|_{(\tilde\Sigma+\zeta I)^{-1}} \;+\;2r_\mu\sqrt{\zeta }\). Then, we decompose $\nabla \tilde L(\mu^\star)$ into (i) a clean term handled by an elliptical-potential bound, and (ii) a perturbation term controlled on $\mathcal E_b^{\mathrm{RLHF}}(\delta_b)$. Finally, we transfer the bound from the clean norm $(\Sigma+\zeta I)$ to the perturbed norm $(\tilde\Sigma+\zeta I)$ by an inverse-comparison argument. More details on $\zeta$ are provided in Remarks \ref{rem:zeta-explicit}-\ref{rem:zeta-choice}. 

\begin{proof}
The full proof is given in Appendix \ref{proof:thm-rlhf}.
\end{proof}

\begin{remark}[Cost of partial observability and preference sample complexity]\label{rem:reward-sample-complexity}
Compared with the fully observed oracle setting, where the belief-induced mismatch vanishes and one recovers the standard $\widetilde O(\sqrt{d/(\rho^2N_{\mathrm{HF}})})$ statistical rate (e.g., Lemma~5.1 of \citet{ZhuPrincipledComparisons}), Theorem~\ref{thm:main-mu} separates the reward-parameter error into three error terms as
\[
\|\widetilde\mu-\mu^\star\|_{\widetilde\Sigma+\zeta I}
\lesssim
\underbrace{\widetilde O\left(\frac{\sqrt d}{\rho\sqrt{N_{\mathrm{HF}}}}\right)}_{\text{statistical error}}
+\underbrace{O\left(\frac{TB_r}{\sqrt\zeta}\epsilon_b(\delta_b,2N_{\mathrm{HF}})\right)}_{\text{model-mismatch error}}
+\underbrace{O(r_\mu\sqrt\zeta)}_{\text{regularization error}}.
\]
Thus, increasing $N_{\mathrm{HF}}$ reduces the statistical term, while the belief-mismatch and regularization terms in this bound are not directly reduced by additional preference samples. This bound should be interpreted as an upper-bound decomposition. Under the choice $\zeta=c_\zeta16T^2B_r^2\epsilon_b(1+\epsilon_b)$, the last two terms are of order $O(\sqrt{\epsilon_b})$ for $\epsilon_b\le1$, up to problem-dependent constants. Thus, for a fixed certified belief-error level $\bar\epsilon_b$, the statistical term reaches the scale of the belief-dependent contribution at a preference-sample size of order $\widetilde O(1/\bar\epsilon_b)$, with the precise threshold determined by $d$, $\rho$, and the problem-dependent constants in \eqref{eq:main_mu_bound}. In particular, if $\epsilon_b(\delta_b,2N_{\mathrm{HF}})\le\bar\epsilon_b$ and $\mathcal B(\bar\epsilon_b)$ denotes the last two terms in \eqref{eq:main_mu_bound} with the belief error bounded by $\bar\epsilon_b$, then for any target accuracy $\epsilon_\mu>\mathcal B(\bar\epsilon_b)$, it is sufficient, up to logarithmic factors, to take
\[
N_{\mathrm{HF}}=\widetilde O\left(\frac{d+\log(1/\delta_c)}{\rho^2\bigl(\epsilon_\mu-\mathcal B(\bar\epsilon_b)\bigr)^2}\right).
\]
Figure~\ref{fig:reward-adaptive-main} provides an empirical validation of the behavior predicted by Theorem~\ref{thm:main-mu}.
\end{remark}

\begin{figure}[h]
    \centering
    \includegraphics[width=0.62\linewidth]{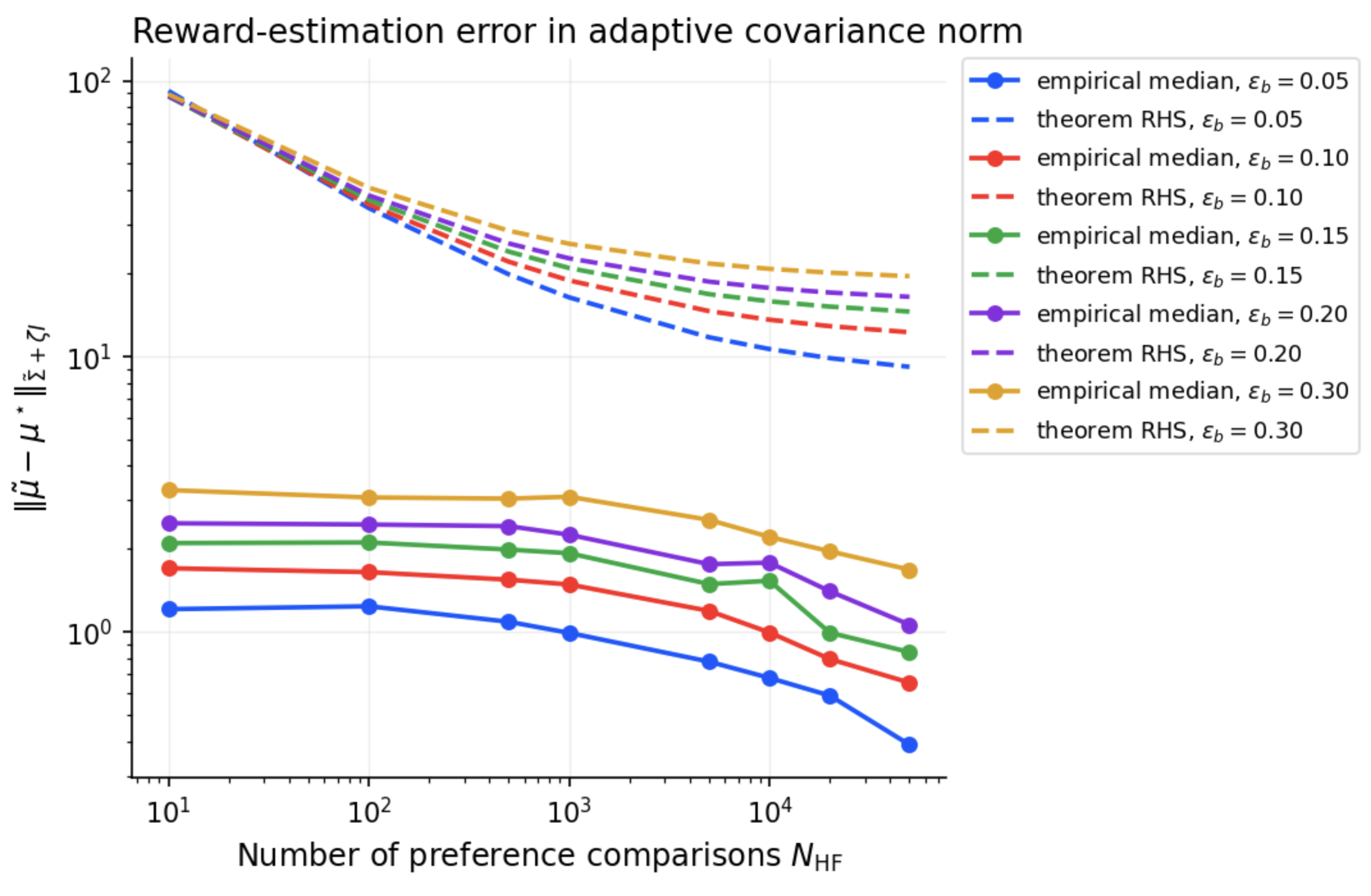}
    \caption{Empirical validation of Theorem~\ref{thm:main-mu} in a synthetic setting. Solid curves show empirical errors, while dashed curves show the corresponding theoretical bounds; see Appendix~\ref{subsec:exp-reward-learning}.}
    \label{fig:reward-adaptive-main}
\end{figure}

% Figure~\ref{fig:reward-adaptive-main} illustrates the error decomposition in Remark~\ref{rem:reward-sample-complexity}. As \(N_{\mathrm{HF}}\) increases, the empirical adaptive-norm error decreases, consistently with the \(N_{\mathrm{HF}}^{-1/2}\) statistical term in Theorem~\ref{thm:main-mu}; for a fixed sample size, larger prescribed belief error \(\epsilon_b\) generally leads to larger estimation error. Across the tested regimes, the empirical curves remain below the corresponding theorem RHS.

\begin{corollary}[End-to-end consequence of finite-sample POMDP estimation]\label{cor:end-to-end-em-reward-main}
Suppose the conditions of Theorems~\ref{thm:finite-sample-covariance-gradient-em} and \ref{thm:main-mu} hold, the POMDP-estimation trajectories are independent of the preference sample, and $T-1\le K$. Fix an EM iteration $\ell\ge0$, and construct the approximate beliefs, perturbed features, and constrained Bradley--Terry estimator using $\widehat\Theta^{(\ell)}$. For some finite-sample EM error $\mathcal R_{\mathrm{EM}}^{(\ell)}$ given by the right-hand side of Eq.~\ref{eq:finite-sample-em-contraction-population}, let $\epsilon_{b,\mathrm{EM}}^{(\ell)}$ denote the resulting belief-accuracy level obtained through Remark~\ref{rem:em-reward-interface} and Corollary \ref{cor:belief-good-condition-cor}. Then, for the corresponding regularization level $\zeta_\ell$, with probability at least $1-\delta_{\mathrm{cov}}-\delta_Q-\delta_b-\delta_c$,
\[
\|\tilde\mu^{(\ell)}-\mu^\star\|_{\tilde\Sigma^{(\ell)}+\zeta_\ell I}
\lesssim
\mathcal O\!\left(\frac{1}{\sqrt{N_{\mathrm{HF}}}}\right)
+\mathcal O\!\left(\frac{TB_r}{\sqrt{\zeta_\ell}}\,\epsilon_{b,\mathrm{EM}}^{(\ell)}\right)
+\mathcal O\!\left(r_\mu\sqrt{\zeta_\ell}\right),
\]
up to the curvature and logarithmic factors in Theorem~\ref{thm:main-mu}. The complete expressions and proof are given in Corollary~\ref{cor:end-to-end-em-reward-learning}.
\end{corollary}

\section{Conclusion}
\label{sec:conclusion}
We analyzed preference-based reward learning in belief-MDPs under partial observability, where the true environment state is latent, and the agent instead acts on belief states obtained by Bayesian filtering from observation histories under an inexact learned POMDP. For finite log-linear POMDPs, we proved explicit stability of the filtering recursion to parametric model mismatch under a Dobrushin-type mixing condition, yielding bounds on $\bE[\|b_t^\Theta-b_t^{\Theta^\star}\|_{\mathrm{TV}}]$ and high-probability time-average control. For this model, covariance-preconditioned gradient EM yields local population contraction under observed-information identifiability and finite-sample convergence to a statistical neighborhood; combined with the filtering and Bradley--Terry guarantees, this gives an end-to-end bound from incomplete POMDP trajectories to reward estimation. We further extended the same belief-stability mechanism to neural-softmax POMDP models through an NTK-style linearization, where finite-width linearization errors enter additively into the belief perturbation bound. We then propagated belief mismatch into trajectory-level feature perturbations and derived finite-sample Bradley--Terry guarantees that separate an $N_{\mathrm{HF}}^{-1/2}$ statistical term from a belief-dependent model-mismatch term. Moreover, empirical experiments on filtering stability, covariance-preconditioned EM, and reward learning are reported in Appendix~\ref{sec:Experiments}.
A limitation is that our filter comparison treats the action sequence as given and studies stability under a common logged action--observation history, which is the relevant regime for offline RLHF on a fixed dataset \cite{kaufmann2025a}. Also, the Dobrushin contraction assumption can be conservative in weakly mixing or nearly deterministic regimes and may only be valid in a local neighborhood of the true parameters. For the neural-softmax extension, the guarantee additionally relies on a lazy-training regime, so the resulting constants depend on the width-dependent approximation errors.

As future directions, one can consider end-to-end closed-loop guarantees that couple belief-filter error with policy-induced trajectory drift, develop weaker stability notions beyond uniform Dobrushin contraction, and extend the analysis to broader learned-dynamics function classes with explicit approximation terms. For EM, natural extensions include neural-softmax dynamics, weaker local initialization and contraction conditions, and adaptive covariance preconditioners without independent pilot data.

\newpage

\bibliography{main}
\bibliographystyle{tmlr}

\newpage
\appendix
\section{Appendix}

\begin{longtable}{@{}p{0.27\textwidth}p{0.69\textwidth}@{}}
\caption{Summary of notation.}
\label{tab:notation}\\
\hline
\textbf{Symbol} & \textbf{Description}\\
\hline
\endfirsthead

\multicolumn{2}{c}{\tablename~\thetable\ -- continued}\\
\hline
\textbf{Symbol} & \textbf{Description}\\
\hline
\endhead

\hline
\endfoot

\multicolumn{2}{@{}l}{\textbf{POMDP model and belief stability}}\\
\(\bS,\bA,\hat\bS\) & Latent-state, action, and observation spaces.\\
\(P_\theta,\Phi_w\) & Log-linear transition and observation kernels.\\
\(\Theta=(\theta,w)\), \(\Theta^\star=(\theta^\star,w^\star)\) & Learned and true POMDP parameters.\\
\(\phi_p,\phi_\Phi\) & Transition and observation feature maps.\\
\(B\) & Uniform transition and observation feature-norm bound.\\
\(\nu_0\) & Initial latent-state distribution.\\
\(Y_t\) & Observation--action history up to time \(t\).\\
\(b_t^\Theta\) & Filtering belief computed under parameter \(\Theta\).\\
\(r_b(b,a)\) & Belief-MDP reward.\\
\(r_\Theta\) & Radius of the local POMDP parameter neighborhood.\\
\(\delta(\theta),\delta(w)\) & Transition and observation parameter errors.\\
\(\kappa_P,\kappa_\Phi\) & Transition and observation Dobrushin coefficients.\\
\(\alpha\) & Belief-stability contraction coefficient.\\
\(c_b(\delta(\theta),\delta(w))\) & One-step belief-mismatch term.\\
\(\epsilon_b(\delta_b)\) & High-probability time-average belief-error bound.\\
\(\mathcal E_b^{\mathrm{RLHF}}(\delta_b)\) & Belief-accuracy event over the preference dataset.\\

\hline
\multicolumn{2}{@{}l}{\textbf{Neural-softmax extension}}\\
\(W=(W_p,W_\Phi)\), \(W^\star\) & Learned and true neural-softmax parameters.\\
\(F_p,F_\Phi\) & Neural transition and observation score functions.\\
\(W_{p,0},W_{\Phi,0}\) & Reference network initializations.\\
\(\mathcal D_p,\mathcal D_\Phi\) & Local lazy-training neighborhoods.\\
\(m\) & Network width.\\
\(\phi_p^{\mathrm{NN}},\phi_\Phi^{\mathrm{NN}}\) & NTK tangent feature maps.\\
\(B^{\mathrm{NN}}\) & Uniform tangent-feature bound.\\
\(\varepsilon_p^{\mathrm{NN}},\varepsilon_\Phi^{\mathrm{NN}}\) & Neural score-linearization errors.\\
\(\mathcal E_{\mathrm{lin}}^{\mathrm{NN}}(\delta_{\mathrm{NN}})\) & Neural linearization event.\\
\(\kappa_P^{\mathrm{NN}},\kappa_\Phi^{\mathrm{NN}}\) & Neural Dobrushin coefficients.\\
\(\alpha_{\mathrm{NN}}\) & Neural belief-stability contraction coefficient.\\
\(c_b^{\mathrm{NN}},\epsilon_b^{\mathrm{NN}}\) & Neural belief-mismatch and belief-error bounds.\\

\hline
\multicolumn{2}{@{}l}{\textbf{Covariance-preconditioned EM}}\\
\(K\) & POMDP-estimation trajectory horizon.\\
\(Z,Y\) & Latent-state trajectory and observed action--observation trajectory.\\
\(Q_Y(\Theta_1,\Theta_2)\) & Single-trajectory EM auxiliary function.\\
\(Q(\Theta_1,\Theta_2)\) & Population EM auxiliary function.\\
\(q(\Theta)\) & Oracle auxiliary function \(Q(\Theta,\Theta^\star)\).\\
\(\gamma_k^\Theta\) & One-slice smoothing marginal.\\
\(\xi_k^\Theta\) & Two-slice smoothing marginal.\\
\(R_{k,\Theta}\) & Backward smoothing kernel.\\
\(\Sigma_p,\Sigma_\Phi\) & Local transition and observation feature covariances.\\
\(\Sigma_{\theta,\star},\Sigma_{w,\star}\) & Complete-data covariance blocks.\\
\(\Sigma_\star\) & Population complete-data covariance.\\
\(\mathcal U_\theta,\mathcal U_w\) & Identifiable parameter subspaces.\\
\(\mathcal I_{\mathrm{obs}}\) & Population observed-data information matrix.\\
\(\nu_{\mathrm{obs}}\) & Relative observed-information identifiability constant.\\
\(\mathbb B_{\Sigma_\star}(\Theta^\star,r)\) & Local covariance-norm parameter neighborhood.\\
\(\chi_p,\chi_\Phi\) & Local covariance-domination constants.\\
\(L_Q^\Sigma(r)\) & Covariance-weighted EM gradient-stability coefficient.\\
\(\kappa_R\) & Backward-smoothing Dobrushin coefficients.\\
\(L_\alpha^\Sigma\) & Covariance-relative filtering-stability constant.\\
\(L_R^\Sigma\) & Backward-kernel stability constant.\\
\(L_\gamma^\Sigma\) & One-slice smoothing stability constant.\\
\(c_b^\Sigma,C_b^\Sigma\) & Covariance-relative belief-mismatch quantities.\\
\(B_{p,\Sigma},B_{\Phi,\Sigma}\) & Covariance-relative feature-diameter bounds.\\
\(L_{\mathrm{pilot}}^\Sigma\) & Pilot-covariance bias constant.\\
\(\eta\) & EM gradient step size.\\
\(\rho_\eta\) & Population preconditioned-gradient contraction term.\\
\(\kappa_{\mathrm{EM}}\) & Population EM contraction coefficient.\\
\(\widetilde\Theta\) & Preliminary parameter used for pilot preconditioning.\\
\(M_0,M\) & Pilot and estimation sample sizes.\\
\(\widehat\Sigma_{M_0}\) & Empirical pilot complete-data covariance.\\
\(\widehat\Sigma_\lambda\) & Regularized pilot covariance preconditioner.\\
\(\lambda\) & Pilot covariance ridge parameter.\\
\(\varepsilon_{\mathrm{cov}}\) & Pilot covariance approximation error.\\
\(\epsilon_Q^{\mathrm{unif},\Sigma}\) & Uniform empirical auxiliary-gradient error.\\
\(\widehat\rho_\eta,\widehat\kappa_{\mathrm{EM}}\) & Finite-sample contraction quantities.\\
\(T_{\mathrm{EM}}\) & Number of EM iterations.\\

\hline
\multicolumn{2}{@{}l}{\textbf{Preference-based reward learning}}\\
\(\phi_r(s,a)\) & Latent-state reward feature.\\
\(B_r\) & Uniform reward-feature bound.\\
\(\mu^\star\) & True reward parameter.\\
\(r_\mu\) & Reward-parameter constraint radius.\\
\(\phi_b(b,a)\) & Belief-based reward feature.\\
\(N_{\mathrm{HF}}\) & Number of independent preference comparisons.\\
\(T\) & Preference-trajectory horizon.\\
\(y_i\) & Binary preference label.\\
\(\phi_i\) & Clean trajectory-level feature difference.\\
\(\tilde\phi_i\) & Feature difference constructed from approximate beliefs.\\
\(\Delta_{\phi,i}\) & Feature perturbation \(\tilde\phi_i-\phi_i\).\\
\(\Sigma,\tilde\Sigma\) & Clean and perturbed empirical feature covariances.\\
\(\tilde\mu\) & Constrained Bradley--Terry estimator.\\
\(\rho\) & Bradley--Terry curvature constant.\\
\(\zeta\) & Covariance regularization level used in the reward bound.\\
\(c_\zeta\) & Constant controlling the choice of \(\zeta\).\\
\(\epsilon(\delta_b,\delta_c,\zeta)\) & Reward-parameter estimation bound.\\
\(\mathcal R_{\mathrm{EM}}^{(\ell)}\) & POMDP parameter-error bound after EM iteration \(\ell\).\\
\(\epsilon_{b,\mathrm{EM}}^{(\ell)}\) & Belief-error level induced by the EM estimate.\\
\(\zeta_\ell\) & Corresponding reward covariance regularization level.\\

\hline
\multicolumn{2}{@{}l}{\textbf{Confidence parameters}}\\
\(\delta_b,\delta_c\) & Belief-error and reward-estimation confidence levels.\\
\(\delta_{\mathrm{NN}}\) & Neural linearization confidence level.\\
\(\delta_{\mathrm{cov}},\delta_Q\) & Pilot-covariance and empirical-gradient confidence levels.\\

\end{longtable}

\newpage

\subsection{Discussion of the Dobrushin Condition and Filter Stability}
\label{app:dobrushin-discussion}
This subsection first discusses the Dobrushin coefficients and the resulting contraction condition, then prepares the technical tools used in the proof of the belief-filter stability results.

\begin{remark}[Role of \(\alpha\)]\label{rem:alpha-control}
In the log-linear model, Lemma~\ref{lem:Dobrushin-row-col} yields explicit lower bounds on $\kappa_P$ and $\kappa_\Phi$. This condition is a non-trivial stability requirement since the Bayesian filtering update need not be a contraction in general (see Example 3.4, \cite{McYuk}). For comparison, the initialization-stability condition of \citet{McYuk} is $(1-\kappa_P)(2-\kappa_\Phi)<1$, whereas here
\[
(1-\kappa_P)(4-3\kappa_\Phi)=(1-\kappa_P)(2-\kappa_\Phi)+2(1-\kappa_P)(1-\kappa_\Phi).
\]
The additional term arises from the perturbation analysis under model mismatch; see the decomposition in the proof of Theorem~\ref{thm:param-belief-lip}. Thus, Assumption~\ref{as:Dobrushin-stability} is a sufficient condition tailored to recurring parametric perturbations and is not claimed to be necessary. The condition also makes explicit the tradeoff between transition mixing and observation overlap. When $\kappa_P<3/4$, $\alpha<1$ is equivalent to \(
\kappa_\Phi>\frac{1}{3}\left(4-\frac{1}{1-\kappa_P}\right),
\) whereas for $\kappa_P\ge3/4$ it holds for every $\kappa_\Phi\in(0,1]$. Hence the condition becomes more restrictive as transition mixing weakens. Finally, with $B$, $r_\Theta$, and the true parameter norms fixed, Lemma~\ref{lem:Dobrushin-row-col} gives, as $|\bS|,|\hat\bS|\to\infty$,
\[
\liminf_{|\bS|\to\infty}\kappa_P
\ge e^{-2B(\|\theta^\star\|_2+r_\Theta)} > 0,
\qquad
\liminf_{|\hat\bS|\to\infty}\kappa_\Phi
\ge e^{-2B(\|w^\star\|_2+r_\Theta)} > 0.
\]
Thus, cardinality alone does not force the sufficient certificate to degenerate.
\end{remark}

\medskip

\begin{example}\label{example:dobrushin}
Consider a four-state model with transition and observation kernels
\[
P=
\begin{pmatrix}
0.23 & 0.07 & 0.18 & 0.52\\
0.12 & 0.13 & 0.41 & 0.34\\
0.15 & 0.36 & 0.17 & 0.32\\
0.37 & 0.11 & 0.21 & 0.31
\end{pmatrix},
\qquad
\Phi=
\begin{pmatrix}
0.04 & 0.66 & 0.20 & 0.10\\
0.50 & 0.06 & 0.12 & 0.32\\
0.10 & 0.14 & 0.68 & 0.08\\
0.36 & 0.31 & 0.05 & 0.28
\end{pmatrix}.
\]
For simplicity, the transition kernel is assumed action-independent in this example, so the infimum over \(a\in\bA\) in the definition of \(\kappa_P\) is immaterial here. The corresponding Dobrushin coefficients are $\kappa_P=0.71$ and $\kappa_\Phi=0.32$, therefore $\alpha=0.88$. The transition rows are heterogeneous, while the conditional observation distributions are well separated; for example, the first and second rows of $\Phi$ have total-variation distance $0.68$.
\end{example}

\paragraph{Practical note.} Although $\kappa_P$ and $\kappa_\Phi$ are worst-case quantities, in practice rather than plugging in analytic lower bounds Eq.~\ref{eq:dobrushin-row} and Eq.~\ref{eq:dobrushin-Phi} that must hold uniformly over all $(s,s',a)$ in the feasible parameter set, we can estimate $\kappa_P$ and $\kappa_\Phi$ from the realized transition and observation kernels. This yields a substantially less pessimistic stability certificate and better matches the observed contraction.

\begin{lemma}\label{lem:model-lip}
Given model dynamic assumption as in \ref{as:Model-LogLinear}, 
and finite feature bounds $B<\infty$, for all feasible parameter pairs $(\theta,w)$, $(\theta',w')$ with
$\|\theta-\theta'\|_2\le\delta(\theta)$, $\|w-w'\|_2\le\delta(w)$:
\begin{align}
&\|P_\theta(\cdot|s,a)-P_{\theta'}(\cdot|s,a)\|_1\le B\delta(\theta),
\qquad \forall (s,a)\in\bS\times\bA,
\label{eq:lip-in-param-kernel}\\[4pt]
&\|\Phi_{w}(\cdot\,|\,s)-\Phi_{w'}(\cdot\,|\,s)\|_1\le B\delta(w),
\qquad \forall s\in\bS,
\label{eq:lip-in-param-obs}\\[4pt]
&\|P_\theta(s'|\cdot,a)-P_{\theta'}(s'|\cdot,a)\|_1
\le
2B\frac{|\bS|}
{1+(|\bS|-1)\exp\bigl(-2B(\|\theta^\star\|_2+r_\Theta)\bigr)}
\,\delta(\theta),
\qquad \forall (a,s')\in\bA\times\bS,
\label{eq:lip-in-param-col-rowKernel}\\[4pt]
&\big\|\Phi_w(\hat s\mid \cdot)-\Phi_{w'}(\hat s\mid \cdot)\big\|_1
\le
2B\frac{|\bS|}
{1+(|\hat\bS|-1)\exp\bigl(-2B(\|w^\star\|_2+r_\Theta)\bigr)}
\,\delta(w),
\qquad \forall \hat s\in\hat\bS 
\label{eq:lip-in-param-col-rowObs}
% \\[4pt]
% &\|\overleftarrow P_\theta(\cdot\mid s',a)-\overleftarrow P_{\theta'}(\cdot\mid s',a)\|_1
% \;\le\;2B\delta(\theta).\label{eq:lip-in-param-colkernel}
\end{align}

% where
% \begin{equation}\label{eq:pred-op-kernel-reverse}
% (\overleftarrow P_\theta^{a_t} q)(s)
% := \sum_{s'\in\bS}
% \frac{P_\theta(s'\mid s,a_t)}{\sum_{u\in\bS}P_\theta(s'\mid u,a_t)}\,q(s'),
% \qquad \Delta_\bS\to \Delta_\bS.
% \end{equation}

\end{lemma}

\begin{proof}
Fix $(s,a)\in\bS\times\bA$ and consider the probability vector \(
P_\theta(\cdot\mid s,a)\in\Delta_{\bS}
\). Denote the Jacobian of $\theta \rightarrow P_\theta(\cdot|s, a)$ by $J_{a,s}(\theta)\in \bR^{|\bS|\times d_\theta}$ where $d_\theta$ is the dimension of $\theta$, and the $s'$-th row is $\nabla_\theta P_\theta(s'|s,a)^\top$. We first establish the row-wise operator bound in \ref{eq:lip-in-param-kernel}. For $v\in\bR^{d_\theta}$ with $\|v\|_2=1$, we have
\begin{align*}
\|J_{a,s}(\theta)v\|_1
&= \sum_{s'\in\bS} \Big|\left(J_{a,s}(\theta)v\big)_{s'}\right|\\
&=\sum_{s'\in\bS} \left| P_\theta(s'\mid s,a)\left(\phi_p(s,a,s')^\top v-\sum_{s''\in\bS}P_\theta(s''\mid s,a)\phi_p(s,a,s'')^\top v
\right)\right|\\
&=\sum_{s'\in\bS}P_\theta(s'\mid s,a)\left| \phi_p(s,a,s')^\top v-\sum_{s''\in\bS}P_\theta(s''\mid s,a)\phi_p(s,a,s'')^\top v \right|\\
&=\bE_{S'\sim P_\theta(\cdot\mid s,a)}
\left[\left|\phi_p(s,a,S')^\top v - \bE_{\bar S\sim P_\theta(\cdot\mid s,a)}\big[\phi_p(s,a,\bar S)^\top v\big]\right|\right]\\
&\le B,
\end{align*}
where the last inequality follows from Cauchy--Schwarz and the fact that
\[
\bE\Big[\big|\phi_p(s,a,S')^\top v-\bE[\phi_p(s,a,S')^\top v]\big|\Big]
\le \sqrt{\operatorname{Var}(\phi_p(s,a,S')^\top v)} \le B.
\]
Therefore 
\[
\|J_{a,s}(\theta)\|_{2\to1}=\sup_{\|v\|_2=1}\|J_{a,s}(\theta)v\|_1\le B,
\]
and by the mean-value theorem, we have
\begin{align*}
\big\|P_\theta(\cdot\mid s,a)-P_{\theta'}(\cdot\mid s,a)\big\|_1
&\le \sup_{\hat\theta \text{ between } \theta \text{ and } \theta'}\|J_{a,s}(\hat\theta)\|_{2\to1}\,\|\theta-\theta'\|_2\\
&\le B\delta(\theta).
\end{align*}
The proof for $\Phi_w(\cdot\mid s)$ is identical, yielding $\|\Phi_{w}(\cdot\,|\,s)-\Phi_{w'}(\cdot\,|\,s)\|_1\le B\delta(w)$.

\medskip
\noindent
We next compute the coordinate-wise gradient bound needed for the column-wise estimates. For $s'\in\bS$,
\begin{align}
\Big\|\nabla_\theta P_\theta(s'\mid s,a)\Big\|_2
&=\Big\|\sum_{s''\in\bS}\frac{\partial\, P_\theta(s'\mid s,a)}{\partial(\theta^\top\phi_p(s,a,s''))}\;\nabla_\theta \theta^\top\phi_p(s,a,s'')\Big\|_2\notag\\
&=\Big\|\sum_{s''\in\bS}\big(\mathrm{Diag}(P_\theta(\cdot\mid s,a))-P_\theta(\cdot\mid s,a)P_\theta(\cdot\mid s,a)^\top\big)_{s's''}\,\phi_p(s,a,s'')\Big\|_2\notag\\
&=\Big\|P_\theta(s'\mid s,a)\Big(\phi_p(s,a,s')-\sum_{s''\in\bS} P_\theta(s''\mid s,a)\,\phi_p(s,a,s'')\Big)\Big\|_2\notag\\
&\le P_\theta(s'\mid s,a)\Big(\|\phi_p(s,a,s')\|_2+\sum_{s''}P_\theta(s''\mid s,a)\,\|\phi_p(s,a,s'')\|_2\Big)\notag\\
&\le P_\theta(s'\mid s,a)\,(B+B)\notag\\[4pt]
&=2B\,P_\theta(s'\mid s,a). \label{eq:transition-kernel-grad-bound}
\end{align}
where the first inequality follows by the triangle inequality and Jensen’s inequality.

\medskip
\noindent
Then, we establish a column–wise Lipschitz bound for the transition kernel (\ref{eq:lip-in-param-col-rowKernel}). Assume $\theta,\theta'\in\bB_2(r_\Theta;\theta^\star)$ and fix
$a\in\bA$, $s'\in\bS$. Let $J_{a,s'}(\theta)\in\bR^{|\bS|\times d_\theta}$ denote the Jacobian of $\theta \mapsto P_\theta(s'\mid \cdot\,,\, a)$, whose $s$–th row is
$\nabla_\theta P_\theta(s'\mid s,a)^\top$. By \ref{eq:transition-kernel-grad-bound}, for every $s\in\bS$,
\[
\big\|\nabla_\theta P_\theta(s'\mid s,a)\big\|_2
\le 2B\,P_\theta(s'\mid s,a).
\]
Therefore,
\begin{align*}
\|J_{a,s'}(\theta)\|_{2\to1}
&=\sup_{\|v\|_2=1}
  \sum_{s\in\bS}\big|\langle\nabla_\theta P_\theta(s'\mid s,a),v\rangle\big|\\
&\le \sum_{s\in\bS}\big\|\nabla_\theta P_\theta(s'\mid s,a)\big\|_2\\
&\le 2B\sum_{s\in\bS}P_\theta(s'\mid s,a).
\end{align*}
By Lemma \ref{lem:Dobrushin-row-col}, the column–sum is uniformly bounded on $\bB_2(r_\Theta;\theta^\star)$, hence
\[
\sum_{s\in\bS}P_\theta(s'\mid s,a)
\le
\frac{|\bS|}
{1+(|\bS|-1)\exp\bigl(-2B(\|\theta^\star\|_2+r_\Theta)\bigr)},
\]
and thus
\[
\|J_{a,s'}(\theta)\|_{2\to1}
\le
2B\frac{|\bS|}
{1+(|\bS|-1)\exp\bigl(-2B(\|\theta^\star\|_2+r_\Theta)\bigr)}.
\]
Applying the mean–value theorem gives
\begin{align*}
\sum_{s\in\bS}\big|P_\theta(s'\mid s,a)-P_{\theta'}(s'\mid s,a)\big|
&=\left\|\big(P_\theta(s'\mid s,a)\big)_{s\in\bS}
-\big(P_{\theta'}(s'\mid s,a)\big)_{s\in\bS}
\right\|_1\\
&\le \sup_{\hat\theta\text{ between }\theta,\theta'}
\|J_{a,s'}(\hat\theta)\|_{2\to1}\,\|\theta-\theta'\|_2\\
&\le
2B\frac{|\bS|}
{1+(|\bS|-1)\exp\bigl(-2B(\|\theta^\star\|_2+r_\Theta)\bigr)}
\,\|\theta-\theta'\|_2.
\end{align*}
If, in addition, $\|\theta-\theta'\|_2\le\delta(\theta)$, this yields
\[
\sum_{s\in\bS}\big|P_\theta(s'\mid s,a)-P_{\theta'}(s'\mid s,a)\big|
\le
2B\frac{|\bS|}
{1+(|\bS|-1)\exp\bigl(-2B(\|\theta^\star\|_2+r_\Theta)\bigr)}
\,\delta(\theta),
\qquad\forall\,s'\in\bS.
\]

\medskip
\noindent
To show the claim in \ref{eq:lip-in-param-col-rowObs}, fix $\hat s\in\hat\bS$ and consider $w\to\Phi_w(\hat s|\cdot)$ from $\bR^{d_w}$. Following a similar argument, we apply the mean-value theorem to the aforementioned map and bound the operator norm $\|J_{\hat s}(w)\|_{2\to 1}$ of its Jacobian $J_{\hat s}(w)\in\bR^{|\bS|\times d_w}$. First, we bound the $\ell_2$-norm of $\nabla_w \Phi_w(\hat s\mid s)^\top$ which is the $s$-th row of $J_{\hat s}(w)$,
\begin{align*}
\big\|\nabla_w \Phi_w(\hat s\mid s)\big\|_2
&= \|\Phi_w(\hat s\mid s)\Big(\phi_\Phi(s,\hat s)-\sum_{y\in\hat\bS}\Phi_w(y\mid s)\,\phi_\Phi(s,y)\Big)\|_2\\
&\le \Phi_w(\hat s\mid s)\Big(\|\phi_\Phi(s,\hat s)\|_2+\sum_{y\in\hat\bS}\Phi_w(y\mid s)\,\|\phi_\Phi(s,y)\|_2\Big)\\
&\le \Phi_w(\hat s\mid s)\,(B+B)
=2B\,\Phi_w(\hat s\mid s).
\end{align*}
\noindent Then, by Lemma~\ref{lem:Dobrushin-row-col}, inequality \ref{eq:Phi-column-sum-upperbound},
\begin{align*}
\|J_{\hat s}(w)\|_{2\to 1}
&=\sup_{\|v\|_2=1}\sum_{s\in\bS}\big|\langle\nabla_w \Phi_w(\hat s\mid s),v\rangle\big|
\;\le\;\sum_{s\in\bS}\big\|\nabla_w \Phi_w(\hat s\mid s)\big\|_2\\
&\le 2B\sum_{s\in\bS}\Phi_w(\hat s\mid s)\\[4pt]
&\le
2B\frac{|\bS|}
{1+(|\hat\bS|-1)\exp\bigl(-2B(\|w^\star\|_2+r_\Theta)\bigr)}.
\end{align*}
Now applying the mean-value theorem, for some $\tilde w$ on the line segment between $w$ and $w'$,
\begin{align*}
\big\|\Phi_w(\hat s\mid \cdot)-\Phi_{w'}(\hat s\mid \cdot)\big\|_1
&\le \|J_{\hat s}(\tilde w)\|_{2\to 1}\,\|w-w'\|_2\\
&\le
2B\frac{|\bS|}
{1+(|\hat\bS|-1)\exp\bigl(-2B(\|w^\star\|_2+r_\Theta)\bigr)}
\,\delta(w),
\end{align*}
and the result follows.

% \medskip
% \noindent
% For the last claim, fix $(a,s')\in\bA\times\bS$ and let $\overleftarrow J_{a,s'}(\theta)$ denote the Jacobian of $\theta\mapsto\overleftarrow P_\theta(\cdot\mid s',a)$.
% For $\|v\|_2=1$, define
% \begin{align*}
% g_{\theta,v}(s)
% &:=\left\langle\nabla_\theta\log P_\theta(s'\mid s,a),v\right\rangle\\
% &=\left\langle\phi_p(s,a,s')-\sum_{\bar s\in\bS}P_\theta(\bar s\mid s,a)\phi_p(s,a,\bar s),v\right\rangle.
% \end{align*}

%  By the coordinate-wise gradient calculation expressed earlier in this proof, $|g_{\theta,v}(s)|\le 2B$. Repeating the centered-Jacobian argument used to prove \eqref{eq:lip-in-param-kernel}, we obtain
% \begin{align*}
% \|\overleftarrow J_{a,s'}(\theta)v\|_1
% &=\bE_{S\sim\overleftarrow P_\theta(\cdot\mid s',a)}
% \left[\left|g_{\theta,v}(S)-\bE_{\bar S\sim\overleftarrow P_\theta(\cdot\mid s',a)}
% [g_{\theta,v}(\bar S)]\right|\right]\\
% &\le \sqrt{\operatorname{Var}_{S\sim\overleftarrow P_\theta(\cdot\mid s',a)}
% \bigl(g_{\theta,v}(S)\bigr)}
% \le 2B.
% \end{align*}
% Thus, $\|\overleftarrow J_{a,s'}(\theta)\|_{2\to1}\le 2B$, and the mean-value theorem yields
% \[
% \bigl\|\overleftarrow P_\theta(\cdot\mid s',a)-\overleftarrow P_{\theta'}(\cdot\mid s',a)\bigr\|_1
% \le 2B\|\theta-\theta'\|_2
% \le 2B\delta(\theta).
% \]
% Since $(a,s')\in\bA\times\bS$ were arbitrary, the result follows.

\end{proof}

In the following lemma, for simplicity, we consider the parameter set to be $\bB_2(\theta^\star,r_\Theta)\times \bB_2(w^\star,r_\Theta)$. This result also holds for the feasible parameter set, $\bB_2(\Theta^\star,r_\Theta)$, assumed in Assumption \ref{as:realizability-bounded-error}.

\begin{lemma}\label{lem:Dobrushin-row-col}
Under Assumption~\ref{as:Model-LogLinear}, in the neighborhoods
\(\theta\in\mathbb B_2(r_\Theta;\theta^\star)\), \(w\in\mathbb B_2(r_\Theta;w^\star)\), the following properties hold:
\begin{equation}\label{eq:column-sum-upperbound}
\sup_{\theta\in\mathbb B_2(r_\Theta;\theta^\star)}\sup_{a\in\bA}\sup_{s'\in\bS}\sum_{s\in\bS}P_\theta(s'\mid s,a)
\;\le\; \frac{|\bS|}{1+(|\bS|-1)\exp\bigl(-2B(\|\theta^\star\|_2+r_\Theta)\bigr)} ,
\end{equation}
\begin{equation}\label{eq:column-sum-lowerbound}
\inf_{\theta\in\mathbb B_2(r_\Theta;\theta^\star)}\inf_{a\in\bA}\inf_{s'\in\bS}\sum_{s\in\bS}P_\theta(s'\mid s,a)\;\ge\;\frac{|\bS|}
{1+(|\bS|-1)\exp\bigl(2B(\|\theta^\star\|_2+r_\Theta)\bigr)} .
\end{equation}
\begin{equation}\label{eq:Phi-column-sum-upperbound}
\sup_{w\in\mathbb B_2(r_\Theta;w^\star)}\sup_{\hat s\in\hat\bS}\sum_{s\in\bS}\Phi_w(\hat s\mid s)\;\le\;\frac{|\bS|}
{1+(|\hat\bS|-1)\exp\bigl(-2B(\|w^\star\|_2+r_\Theta)\bigr)} ,
\end{equation} 
\begin{equation}\label{eq:Phi-column-sum-lowerbound}
\inf_{w\in\mathbb B_2(r_\Theta;w^\star)}\inf_{\hat s\in\hat\bS}\sum_{s\in\bS}\Phi_w(\hat s\mid s)\;\ge\;\frac{|\bS|}
{1+(|\hat\bS|-1)\exp\bigl(2B(\|w^\star\|_2+r_\Theta)\bigr)} .
\end{equation}
\medskip\noindent
Furthermore, the uniform Dobrushin coefficients in Definition~\ref{def:dobrushin-controlled} satisfy
\begin{equation}\label{eq:dobrushin-row}
\kappa_P\;\ge\;\frac{|\bS|}
{1+(|\bS|-1)\exp\bigl(2B(\|\theta^\star\|_2+r_\Theta)\bigr)} ,
\end{equation}
\begin{equation}\label{eq:dobrushin-Phi}
\kappa_\Phi\;\ge\;\frac{|\hat\bS|}
{1+(|\hat\bS|-1)\exp\bigl(2B(\|w^\star\|_2+r_\Theta)\bigr)},
\end{equation}

% \begin{equation}\label{eq:dobrushin-col}
% \kappa_{\overleftarrow P}\;\ge\;
% \frac{1+(|\bS|-1)\exp\bigl(-2B(\|\theta^\star\|_2+r_\Theta)\bigr)}
% {1+(|\bS|-1)\exp\bigl(2B(\|\theta^\star\|_2+r_\Theta)\bigr)} .
% \end{equation}

\end{lemma}

\begin{remark}
By Lemma \ref{lem:Dobrushin-row-col} and Definition \ref{def:dobrushin-controlled}, for all $\nu,\nu'\in\Delta_{\bS}$ and all $a\in\bA$, the operators $P_\theta^{a}$ and $\Phi_w$ defined in \ref{eq:pred-op-kernel} satisfy
\begin{align}
\bigl\|P_\theta^{a}\nu-P_\theta^{a}\nu'\bigr\|_1
&\le\left(1-\frac{|\bS|}{1+(|\bS|-1)\exp\bigl(2B(\|\theta^\star\|_2+r_\Theta)\bigr)}\right)\|\nu-\nu'\|_1,
\label{eq:row-Dobrushin-Ineq-lemma}\\[4pt]
\bigl\|\Phi_w\nu-\Phi_w\nu'\bigr\|_1
&\le\left(1-\frac{|\hat\bS|}{1+(|\hat\bS|-1)\exp\bigl(2B(\|w^\star\|_2+r_\Theta)\bigr)}\right)\|\nu-\nu'\|_1,
\label{eq:row-Dobrushin-ineq-Phi-lemma}\\[4pt]
% \bigl\|\overleftarrow P_\theta^{a}\nu-\overleftarrow P_\theta^{a}\nu'\bigr\|_1
% &\le\left(1-
% \frac{1+(|\bS|-1)\exp\bigl(-2B(\|\theta^\star\|_2+r_\Theta)\bigr)}
% {1+(|\bS|-1)\exp\bigl(2B(\|\theta^\star\|_2+r_\Theta)\bigr)}
% \right)\|\nu-\nu'\|_1.
% \label{eq:column-Dobrushin-ineq-lemma}
\end{align}
In particular, at $\theta=\theta^\star$, the term $r_\Theta$ will be disappeared form the right hand side of the inequalities.
% \begin{align}
% \bigl\|\overleftarrow P_{\theta^\star}^{a}\nu-\overleftarrow P_{\theta^\star}^{a}\nu'\bigr\|_1
% &\le\left(1-
% \frac{1+(|\bS|-1)\exp\bigl(-2B\|\theta^\star\|_2\bigr)}
% {1+(|\bS|-1)\exp\bigl(2B\|\theta^\star\|_2\bigr)}
% \right)\|\nu-\nu'\|_1.
% \label{eq:column-Dobrushin-ineq-star-lemma}
% \end{align}

\end{remark}

\begin{proof}[Proof of Lemma~\ref{lem:Dobrushin-row-col}]
For any \((s,a,s')\in\bS\times\bA\times\bS\) and any \(\theta\in\bB_2(r_\Theta;\theta^\star)\), we have
\[
|\theta^\top\phi_p(s,a,s')|\le\|\theta\|_2\,\|\phi_p(s,a,s')\|_2\le B(\|\theta^\star\|_2+r_\Theta).
\]
Hence,
\begin{equation}\label{eq:score-funct-bounds}
\exp\bigl(-B(\|\theta^\star\|_2+r_\Theta)\bigr)\le\exp(\theta^\top\phi_p(s,a,s'))\le\exp\bigl(B(\|\theta^\star\|_2+r_\Theta)\bigr).
\end{equation}
Using the fact that the numerator term also appears in the denominator, we have
\begin{align}\label{eq:kernel-uni-upperbound}
P_\theta(s'\mid s,a)
&=\frac{\exp(\theta^\top\phi_p(s,a,s'))}{\sum_{\bar s\in\bS}\exp(\theta^\top\phi_p(s,a,\bar s))}\notag\\
&\le \frac{\exp\bigl(B(\|\theta^\star\|_2+r_\Theta)\bigr)}
{\exp\bigl(B(\|\theta^\star\|_2+r_\Theta)\bigr)+(|\bS|-1)\exp\bigl(-B(\|\theta^\star\|_2+r_\Theta)\bigr)}\notag\\
&=
\frac{1}
{1+(|\bS|-1)\exp\bigl(-2B(\|\theta^\star\|_2+r_\Theta)\bigr)} .
\end{align}
Summing over \(s\in\bS\) for fixed \((a,s')\) yields
\[
\sum_{s\in\bS}P_\theta(s'\mid s,a)\le\frac{|\bS|}{1+(|\bS|-1)\exp\bigl(-2B(\|\theta^\star\|_2+r_\Theta)\bigr)} .
\]
Taking the supremum over \((\theta,a,s')\) proves \ref{eq:column-sum-upperbound}. Similarly,
\begin{align}\label{eq:kernel-uni-lowerbound}
P_\theta(s'\mid s,a)
&\ge\frac{\exp\bigl(-B(\|\theta^\star\|_2+r_\Theta)\bigr)}{\exp\bigl(-B(\|\theta^\star\|_2+r_\Theta)\bigr)+(|\bS|-1)\exp\bigl(B(\|\theta^\star\|_2+r_\Theta)\bigr)}\notag\\
&=\frac{1}{1+(|\bS|-1)\exp\bigl(2B(\|\theta^\star\|_2+r_\Theta)\bigr)}=:P_{\min}.
\end{align}
Taking the infimum over \((\theta,a,s')\) and summing \(P_{\min}\) over \(s\in\bS\) proves \ref{eq:column-sum-lowerbound}. For the Dobrushin coefficient of \(P_\theta\) in \ref{eq:dobrushin-row}, every entry of \(P_\theta(\cdot\mid s,a)\) is bounded from below by \(P_{\min}\). Hence,
\begin{align}\label{eq:dobcoef-p-lowerbound}
\sum_{s' \in\bS}\min\{P_\theta(s'\mid s,a),P_\theta(s'\mid s'',a)\}
&\ge|\bS|\,P_{\min}\notag\\
&=\frac{|\bS|}
{1+(|\bS|-1)\exp\bigl(2B(\|\theta^\star\|_2+r_\Theta)\bigr)} .
\end{align}
Taking the infimum over \(\theta\in\mathbb B_2(r_\Theta;\theta^\star)\), \(a\in\bA\), and \(s,s''\in\bS\) yields \ref{eq:dobrushin-row}.

\medskip\noindent
Then, we prove \ref{eq:Phi-column-sum-lowerbound} and \ref{eq:Phi-column-sum-upperbound}, which provide uniform column-sum bounds for the observation model \(\Phi_w\). The proof is similar to the derivation of \ref{eq:column-sum-upperbound} and \ref{eq:column-sum-lowerbound} for \(P_\theta\). Fix \((s,\hat s)\in\bS\times\hat\bS\) and let \(w\in\mathbb B_2(r_\Theta;w^\star)\). By the feature bound \(\|\phi_\Phi(s,y)\|_2\le B\) and the radius assumption,
\[
|w^\top\phi_\Phi(s,y)|
\le\|w\|_2\,\|\phi_\Phi(s,y)\|_2
\le B(\|w^\star\|_2+r_\Theta),
\qquad \forall y\in\hat\bS.
\]
Hence, for all \(y\in\hat\bS\),
\[
\exp\bigl(-B(\|w^\star\|_2+r_\Theta)\bigr)\le\exp \big(w^\top\phi_\Phi(s,y)\big)\le\exp\bigl(B(\|w^\star\|_2+r_\Theta)\bigr).
\]
Recalling the log-linear observation model in Assumption~\ref{as:Model-LogLinear}, we obtain the uniform pointwise bounds
\begin{align}
\Phi_w(\hat s\mid s)
&\le\frac{\exp\bigl(B(\|w^\star\|_2+r_\Theta)\bigr)}
{\exp\bigl(B(\|w^\star\|_2+r_\Theta)\bigr)+(|\hat\bS|-1)\exp\bigl(-B(\|w^\star\|_2+r_\Theta)\bigr)}
\notag\\
&=\frac{1}
{1+(|\hat\bS|-1)\exp\bigl(-2B(\|w^\star\|_2+r_\Theta)\bigr)}
=:\Phi_{\max},\label{eq:Phi-pointwise-upper}\\[4pt]
\Phi_w(\hat s\mid s)&\ge\frac{\exp\bigl(-B(\|w^\star\|_2+r_\Theta)\bigr)}
{\exp\bigl(-B(\|w^\star\|_2+r_\Theta)\bigr)+(|\hat\bS|-1)\exp\bigl(B(\|w^\star\|_2+r_\Theta)\bigr)}\notag\\
&=\frac{1}{1+(|\hat\bS|-1)\exp\bigl(2B(\|w^\star\|_2+r_\Theta)\bigr)}
=:\Phi_{\min}.\label{eq:Phi-pointwise-lower}
\end{align}
Summing \ref{eq:Phi-pointwise-upper} and \ref{eq:Phi-pointwise-lower} over \(s\in\bS\) yields, for every fixed \(\hat s\in\hat\bS\),
\[
\sum_{s\in\bS}\Phi_w(\hat s\mid s)\le\frac{|\bS|}{1+(|\hat\bS|-1)\exp\bigl(-2B(\|w^\star\|_2+r_\Theta)\bigr)} ,
\]
\[
\sum_{s\in\bS}\Phi_w(\hat s\mid s)\ge\frac{|\bS|}{1+(|\hat\bS|-1)\exp\bigl(2B(\|w^\star\|_2+r_\Theta)\bigr)} .
\]
Taking the supremum and infimum over \((w,\hat s)\) gives \ref{eq:Phi-column-sum-upperbound} and \ref{eq:Phi-column-sum-lowerbound}, respectively. To obtain the Dobrushin coefficient in \ref{eq:dobrushin-Phi}, using \(\Phi_{\min}\) by the exact same argument used to obtain \(\kappa_P\), we have
\begin{align}\label{eq:dobcoef-Phi-lowerbound}
\sum_{\hat s \in\hat\bS}\min\{\Phi_w(\hat s\mid s),\Phi_w(\hat s\mid s'')\}
&\ge|\hat\bS|\,\Phi_{\min}\notag\\
&=\frac{|\hat\bS|}
{1+(|\hat\bS|-1)\exp\bigl(2B(\|w^\star\|_2+r_\Theta)\bigr)} .
\end{align}
Taking the infimum over \(w\in\mathbb B_2(r_\Theta;w^\star)\) and \(s,s''\in\bS\) proves \ref{eq:dobrushin-Phi}.

% \medskip
% \noindent
% Finally, we establish \ref{eq:dobrushin-col}. For every feasible $\theta$, $a\in\bA$, and $s,s'\in\bS$, by \ref{eq:kernel-uni-lowerbound} and \ref{eq:column-sum-upperbound},
% \begin{align}
% \overleftarrow P_\theta(s\mid s',a)
% &=\frac{P_\theta(s'\mid s,a)}{\sum_{u\in\bS}P_\theta(s'\mid u,a)}\notag\\
% &\ge
% \frac{\displaystyle \frac{1}{1+(|\bS|-1)\exp\bigl(2B(\|\theta^\star\|_2+r_\Theta)\bigr)}}
% {\displaystyle \frac{|\bS|}{1+(|\bS|-1)\exp\bigl(-2B(\|\theta^\star\|_2+r_\Theta)\bigr)}}\notag\\
% &=
% \frac{1+(|\bS|-1)\exp\bigl(-2B(\|\theta^\star\|_2+r_\Theta)\bigr)}
% {|\bS|\bigl(1+(|\bS|-1)\exp\bigl(2B(\|\theta^\star\|_2+r_\Theta)\bigr)\bigr)}
% =:\overleftarrow P_{\min}.
% \label{eq:reverse-kernel-pointwise-lower}
% \end{align}
% Therefore, for every $s',s''\in\bS$,
% \begin{align}
% \sum_{s\in\bS}\min\bigl\{\overleftarrow P_\theta(s\mid s',a),\overleftarrow P_\theta(s\mid s'',a)\bigr\}
% &\ge |\bS|\,\overleftarrow P_{\min}\notag\\
% &=
% \frac{1+(|\bS|-1)\exp\bigl(-2B(\|\theta^\star\|_2+r_\Theta)\bigr)}
% {1+(|\bS|-1)\exp\bigl(2B(\|\theta^\star\|_2+r_\Theta)\bigr)} .
% \end{align}
% Taking the infimum over $\theta\in\mathbb B_2(r_\Theta;\theta^\star)$, $a\in\bA$, and $s',s''\in\bS$ proves \ref{eq:dobrushin-col}.

\end{proof}

\section{Proof of Theorem \ref{thm:param-belief-lip}}

\begin{remark}
First, recall \ref{eq:mu-min-max} where we assumed \(1>\nu_{\max}\ge\nu_0(s)\ge\nu_{\min}>0\). Plus, similar to the argument in Lemma \ref{lem:Dobrushin-row-col} that we have uniformly bounded \(P_\theta(s'|s,a)\) in Eq.~\ref{eq:kernel-uni-upperbound} and \ref{eq:kernel-uni-lowerbound}, and given the similar log-linear structure of models \(P_\theta\) and \(\Phi_w\) assumed in \ref{as:Model-LogLinear} and bounded feature maps in \ref{as:Model-LogLinear}, we can establish uniform bounds over \(\Phi_w\) and \(P_\theta\). Denote
\begin{align}
P_{\min}
&:=\frac{1}{1+(|\bS|-1)\exp\bigl(2B(\|\theta^\star\|_2+r_\Theta)\bigr)},
&
P_{\max}
&:=\frac{1}{1+(|\bS|-1)\exp\bigl(-2B(\|\theta^\star\|_2+r_\Theta)\bigr)},
\label{eq:model-elemenst-bounds-P}\\[2pt]
\Phi_{\min}
&:=\frac{1}{1+(|\hat\bS|-1)\exp\bigl(2B(\|w^\star\|_2+r_\Theta)\bigr)},
&
\Phi_{\max}
&:=\frac{1}{1+(|\hat\bS|-1)\exp\bigl(-2B(\|w^\star\|_2+r_\Theta)\bigr)}.
\label{eq:model-elemenst-bounds-Phi}
\end{align}
where for all \((s,s',a,\hat s)\) and \(\Theta\in\bB_2(\Theta^\star,r_\Theta)\) we have
\begin{align}
0<\nu_{\min}\;\le\;\nu_0(s)\;\le\;\nu_{\max},\quad
&0<P_{\min}\;\le\;P_\theta(s'\mid s,a)\;\le\;P_{\max}, \quad
0<\Phi_{\min}\;\le\;\Phi_w(\hat s\mid s)\;\le\;\Phi_{\max}.
\end{align}
\end{remark}

Besides, define the prediction and observation operators
\begin{equation}\label{eq:pred-op-kernel}
(P_\theta^{a_{t}}q)(s')
=\sum_{s\in\bS}P_\theta(s'\,|\,s,a_t)\,q(s),
\quad \Delta_\bS\to \Delta_\bS,
\end{equation}
\begin{equation}\label{eq:pred-op-observ}
(Q_w q)(\hat s)
=\sum_{s'\in\bS}\Phi_w(\hat s\mid s')\,q(s'),
\quad \Delta_\bS\to\Delta_{\hat\bS},
\end{equation}
and the Bayesian update operator
\[
\psi_{\Phi_w}:\Delta_\bS\times\hat\bS\to\Delta_\bS\cup\{0\},
\]
\[
\psi_{\Phi_w}(b,\hat{s})(s')=
\frac{\Phi_w(\hat s| s')b(s')}{\sum_{\bar s\in\bS}\Phi_w(\hat s| \bar s)b(\bar s)}
=\frac{\Phi_w(\hat s| s')b(s')}{(Q_w b)(\hat s)}.
\]
Then, the belief update \ref{eq:belief-formula} can be expressed as
\[
b_{t+1}^\Theta=\psi_{\Phi_w}(P_\theta^{a_t}b^\Theta_t,\, \hat s_{t+1}).
\]

Now, we can prove Theorem \ref{thm:param-belief-lip}.

\begin{proof}[Proof of Theorem \ref{thm:param-belief-lip}]\label{proof-thm-belief}
For two parameter sets \(\Theta\) and \(\Theta^\star\), we have
\[
b_{k+1}^\Theta=\psi_{\Phi_w} \big(P_\theta^{a_k} b_k^\Theta,\hat s_{k+1}\big),\qquad
b_{k+1}^{\Theta^\star}=\psi_{\Phi_{w^\star}} \big(P_{\theta^\star}^{a_k} b_k^{\Theta^\star},\hat s_{k+1}\big),
\qquad b_0^\Theta=b_0^{\Theta^\star}=\nu_0.
\]
We work under the conditional law \(\bP_{\Theta^\star}(\,\cdot\,|a_{0:t-1})\), i.e., the action sequence \((a_0,\dots,a_{t-1})\) is treated as \emph{fixed}. We use the filtrations
\(
\mathcal F_k^- := \sigma(\hat s_{1:k},a_{0:k-1}),\text{ and }
\mathcal F_k := \sigma(\hat s_{1:k},a_{0:k}),
\)
so that \(b_k^{\Theta^\star}\) and \(b_k^\Theta\) are \(\mathcal F_k^-\)-measurable and \(a_k\) is \(\mathcal F_k\)-measurable. In particular, the conditional law of \(\hat s_{k+1}\) given \(\mathcal F_k\) is \(
\hat s_{k+1}\mid \mathcal F_k \sim Q_{w^\star}\!\Big(P_{\theta^\star}^{a_k} b_k^{\Theta^\star}\Big).
\) Fix \(k\in\{0,\dots,t-1\}\). Using the filter recursions
\(b_{k+1}^{\Theta}=\psi_{\Phi_w}\!\big(P_\theta^{a_k} b_k^{\Theta},\hat s_{k+1}\big)\) and
\(b_{k+1}^{\Theta^\star}=\psi_{\Phi_{w^\star}}\!\big(P_{\theta^\star}^{a_k} b_k^{\Theta^\star},\hat s_{k+1}\big)\),
add and subtract the intermediate terms
\(\psi_{\Phi_{w^\star}}\!\big(P_{\theta}^{a_k} b_k^{\Theta},\hat s_{k+1}\big)\) and
\(\psi_{\Phi_{w^\star}}\!\big(P_{\theta^\star}^{a_k} b_k^{\Theta},\hat s_{k+1}\big)\):
\begin{align} 
\|b_{k+1}^\Theta-b_{k+1}^{\Theta^\star}\|_{\mathrm{TV}}
&=\big\|\psi_{\Phi_w}(P_\theta^{a_k} b_k^\Theta,\hat s_{k+1}) -\psi_{\Phi_{w^\star}}(P_{\theta^\star}^{a_k} b_k^{\Theta^\star},\hat s_{k+1})\big\|_{\mathrm{TV}}\\[8pt]
&\le \underbrace{\big\|\psi_{\Phi_w}(P_\theta^{a_k} b_k^\Theta,\hat s_{k+1}) -\psi_{\Phi_{w^\star}}(P_\theta^{a_k} b_k^{\Theta},\hat s_{k+1})\big\|_{\mathrm{TV}}}_{\text{(I)}} \label{eq:belief-theorem-term-I}\\[2pt]
&+\underbrace{\big\|\psi_{\Phi_{w^\star}}(P_\theta^{a_k} b_k^{\Theta},\hat s_{k+1}) -\psi_{\Phi_{w^\star}}(P_{\theta^\star}^{a_k} b_k^{\Theta},\hat s_{k+1})\big\|_{\mathrm{TV}}}_{\text{(II)}} \label{eq:belief-theorem-term-II}\\[2pt]
&+\underbrace{\big\|\psi_{\Phi_{w^\star}}(P_{\theta^\star}^{a_k} b_k^{\Theta},\hat s_{k+1}) -\psi_{\Phi_{w^\star}}(P_{\theta^\star}^{a_k} b_k^{\Theta^\star},\hat s_{k+1})\big\|_{\mathrm{TV}}}_{\text{(III)}}. \label{eq:belief-theorem-term-III}
\end{align}

\medskip
\noindent
\paragraph{Term (I).} This term corresponds to the perturbation of the observation model. First, consider the following observation for any generic \(b\in\Delta_{|\bS|}\):
\begin{align}\label{eq:bayes-op-param-perturb}
\|\psi_{\Phi_w}(b,\hat{s}) - \psi_{\Phi_{w^\star}}(b,\hat{s})\|_{\mathrm{TV}}
&=\frac12\sum_{s'\in\bS}\left|
\frac{\Phi_w(\hat s\mid s')\,b(s')}{\sum_{\bar s}\Phi_w(\hat s\mid \bar s)\,b(\bar s)}
-\frac{\Phi_{w^\star}(\hat s\mid s')\,b(s')}{\sum_{\bar s}\Phi_{w^\star}(\hat s\mid \bar s)\,b(\bar s)}
\right|\notag\\[6pt]
&\le \frac12\sum_{s'\in\bS}\left|
\frac{\bigl(\Phi_w(\hat s\mid s')-\Phi_{w^\star}(\hat s\mid s')\bigr)\,b(s')}{\sum_{\bar s}\Phi_{w^\star}(\hat s\mid \bar s)\,b(\bar s)}
\right|\notag\\
&\qquad\quad+\frac12\sum_{s'\in\bS}\left|
\Phi_w(\hat s\mid s')\,b(s')\Big(\frac{1}{\sum_{\bar s}\Phi_w(\hat s\mid \bar s)\,b(\bar s)} -\frac{1}{\sum_{\bar s}\Phi_{w^\star}(\hat s\mid \bar s)\,b(\bar s)}\Big)
\right|\notag\\[6pt]
&\le \frac{1}{2\sum_{\bar s}\Phi_{w^\star}(\hat s\mid \bar s)\,b(\bar s)}
\sum_{s'\in\bS} b(s')\,\bigl|\Phi_w(\hat s\mid s')-\Phi_{w^\star}(\hat s\mid s')\bigr|\notag\\
&\qquad\qquad+\frac12\Big(\sum_{s'\in\bS}\Phi_w(\hat s\mid s')\,b(s')\Big)
\Big|\frac{1}{\sum_{\bar s}\Phi_w(\hat s\mid \bar s)\,b(\bar s)}
-\frac{1}{\sum_{\bar s}\Phi_{w^\star}(\hat s\mid \bar s)\,b(\bar s)}\Big|\notag\\[6pt]
&= \frac{1}{2\sum_{\bar s}\Phi_{w^\star}(\hat s\mid \bar s)\,b(\bar s)}
\sum_{s'\in\bS} b(s')\,\bigl|\Phi_w(\hat s\mid s')-\Phi_{w^\star}(\hat s\mid s')\bigr|\notag\\
&\qquad\qquad
+\frac12\,\frac{\Big|\sum_{\bar s}\Phi_w(\hat s\mid \bar s)\,b(\bar s)
-\sum_{\bar s}\Phi_{w^\star}(\hat s\mid \bar s)\,b(\bar s)\Big|}
{\sum_{\bar s}\Phi_{w^\star}(\hat s\mid \bar s)\,b(\bar s)}\notag\\[6pt]
&\le \frac{1}{2\sum_{\bar s}\Phi_{w^\star}(\hat s\mid \bar s)\,b(\bar s)}
\sum_{s'\in\bS} b(s')\,\bigl|\Phi_w(\hat s\mid s')-\Phi_{w^\star}(\hat s\mid s')\bigr|\notag\\
&\qquad\qquad
+\frac{1}{2\sum_{\bar s}\Phi_{w^\star}(\hat s\mid \bar s)\,b(\bar s)}
\sum_{s'\in\bS} b(s')\,\bigl|\Phi_w(\hat s\mid s')-\Phi_{w^\star}(\hat s\mid s')\bigr|\notag\\[6pt]
&=\frac{\sum_{s'\in\bS}b(s')\bigl|\Phi_w(\hat s\mid s')-\Phi_{w^\star}(\hat s\mid s')\bigr|}
{\sum_{\bar s}\Phi_{w^\star}(\hat s\mid \bar s)\,b(\bar s)}.
\end{align}
Now we take the expectation under \(\hat s\sim (Q_{w^\star}b)(\cdot)\). Since \((Q_{w^\star}b)(\cdot)\) is a probability distribution on \(\hat\bS\) with non-zero support, we have
\begin{align}\label{eq:exp-denominator-eliminate}
\underset{\hat s\sim (Q_{w^\star}b)}{\bE}\!\left[
\frac{\sum_{s'\in\bS}b(s')\,\big|\Phi_w(\hat s\mid s')-\Phi_{w^\star}(\hat s\mid s')\big|}{(Q_{w^\star}b)(\hat s)}
\right]
&=\sum_{\hat s\in\hat\bS}(Q_{w^\star}b)(\hat s)\,
\frac{\sum_{s'}b(s')\,|\Phi_w(\hat s\mid s')-\Phi_{w^\star}(\hat s\mid s')|}{(Q_{w^\star}b)(\hat s)}\notag\\
&=\sum_{s'\in\bS}b(s')\sum_{\hat s\in\hat\bS}\big|\Phi_w(\hat s\mid s')-\Phi_{w^\star}(\hat s\mid s')\big|\notag\\
&=\sum_{s'\in\bS}b(s')\,\big\|\Phi_w(\cdot\mid s')-\Phi_{w^\star}(\cdot\mid s')\big\|_{1}\notag\\
&\le \sup_{s'\in\bS}\big\|\Phi_w(\cdot\mid s')-\Phi_{w^\star}(\cdot\mid s')\big\|_{1}\notag\\
&\le B\,\delta(w),
\end{align}
where in the last step we used Lemma \ref{lem:model-lip}. Now, we observe that for any measurable \(f:\hat\bS\to[0,1]\) and distributions \(\nu,\nu'\) on \(\hat\bS\),
\[
\bE_{\hat s\sim\nu}[f(\hat s)] \le \bE_{\hat s\sim\nu'}[f(\hat s)] + \|\nu-\nu'\|_{\mathrm{TV}}.
\]
Applying this with
\[
f(\hat s)=\|\psi_{\Phi_w}(P_\theta^{a_k}b_k^\Theta,\hat s)-\psi_{\Phi_{w^\star}}(P_\theta^{a_k}b_k^\Theta,\hat s)\|_{\mathrm{TV}}\in[0,1],
\]
\(\nu = Q_{w^\star}(P_{\theta^\star}^{a_k}b_k^{\Theta^\star})\), and
\(\nu' = Q_{w^\star}(P_{\theta}^{a_k}b_k^{\Theta})\), yields
\begin{align}\label{eq:termI-main-belief}
\bE\big[(\mathrm{I})\mid \mathcal F_k\big]
&=\underset{\hat s\sim Q_{w^\star}(P_{\theta^\star}^{a_k}b_k^{\Theta^\star})}{\bE}\!\Big[
\|\psi_{\Phi_w}(P_\theta^{a_k}b_k^\Theta,\hat s)-\psi_{\Phi_{w^\star}}(P_\theta^{a_k}b_k^\Theta,\hat s)\|_{\mathrm{TV}}
\Big]\notag\\
&\le
\underset{\hat s\sim Q_{w^\star}(P_{\theta}^{a_k}b_k^{\Theta})}{\bE}\!\Big[
\|\psi_{\Phi_w}(P_\theta^{a_k}b_k^\Theta,\hat s)-\psi_{\Phi_{w^\star}}(P_\theta^{a_k}b_k^\Theta,\hat s)\|_{\mathrm{TV}}
\Big]\notag\\
&\qquad\qquad
+\Big\|Q_{w^\star}(P_{\theta^\star}^{a_k}b_k^{\Theta^\star})-Q_{w^\star}(P_{\theta}^{a_k}b_k^{\Theta})\Big\|_{\mathrm{TV}}.
\end{align}
The first term is controlled by Eq.~\ref{eq:bayes-op-param-perturb}--Eq.~\ref{eq:exp-denominator-eliminate} with \(b=P_\theta^{a_k}b_k^\Theta\).

For the mismatch term in Eq.~\ref{eq:termI-main-belief}, since \(Q_{w^\star}\) is a kernel, its Dobrushin contraction gives
\begin{align}\label{eq:termI-mismatch-part-belief}
\Big\|Q_{w^\star}(P_{\theta^\star}^{a_k}b_k^{\Theta^\star})-Q_{w^\star}(P_{\theta}^{a_k}b_k^{\Theta})\Big\|_{\mathrm{TV}}
&\le (1-\kappa_\Phi)\,\Big\|P_{\theta^\star}^{a_k}b_k^{\Theta^\star}-P_{\theta}^{a_k}b_k^{\Theta}\Big\|_{\mathrm{TV}}\notag\\
&\le (1-\kappa_\Phi)\Big(
\Big\|P_{\theta^\star}^{a_k}b_k^{\Theta^\star}-P_{\theta^\star}^{a_k}b_k^{\Theta}\Big\|_{\mathrm{TV}}
+\Big\|P_{\theta^\star}^{a_k}b_k^{\Theta}-P_{\theta}^{a_k}b_k^{\Theta}\Big\|_{\mathrm{TV}}
\Big)\notag\\
&\le (1-\kappa_\Phi)\Big(
(1-\kappa_P)\,\|b_k^{\Theta^\star}-b_k^{\Theta}\|_{\mathrm{TV}}
+\frac{B}{2}\,\delta(\theta)
\Big).
\end{align}
Combining Eq.~\ref{eq:termI-main-belief}, Eq.~\ref{eq:exp-denominator-eliminate}, and Eq.~\ref{eq:termI-mismatch-part-belief}, we obtain
\begin{align}\label{eq:I-final-bound-thm}
\bE\big[(\mathrm{I})\mid \mathcal F_k\big]
\le
B\,\delta(w)
+\frac{B}{2}(1-\kappa_\Phi)\,\delta(\theta)
+(1-\kappa_P)(1-\kappa_\Phi)\,\|b_k^{\Theta^\star}-b_k^{\Theta}\|_{\mathrm{TV}}.
\end{align}

\paragraph{Term (II).} To control the perturbation of the transition kernel, we keep the sampling law of \(\hat s_{k+1}\) under the true model and add a change-of-measure step, as in Term (I):
\begin{align}\label{eq:termII-change-measure}
\bE\big[\mathrm{(II)}\mid \mathcal F_k\big]
&=
\bE_{\hat s\sim Q_{w^\star}(P_{\theta^\star}^{a_k} b_k^{\Theta^\star})}
\Big[\big\|\psi_{\Phi_{w^\star}}(P_{\theta}^{a_k} b_k^{\Theta},\hat s)-\psi_{\Phi_{w^\star}}(P_{\theta^\star}^{a_k} b_k^{\Theta},\hat s)\big\|_{\mathrm{TV}}\Big]\notag\\
&\le
\bE_{\hat s\sim Q_{w^\star}(P_{\theta^\star}^{a_k} b_k^{\Theta})}
\Big[\big\|\psi_{\Phi_{w^\star}}(P_{\theta}^{a_k} b_k^{\Theta},\hat s)-\psi_{\Phi_{w^\star}}(P_{\theta^\star}^{a_k} b_k^{\Theta},\hat s)\big\|_{\mathrm{TV}}\Big]\notag\\
&\qquad\quad+\Big\|Q_{w^\star}\!\big(P_{\theta^\star}^{a_k} b_k^{\Theta^\star}\big)-Q_{w^\star}\!\big(P_{\theta^\star}^{a_k} b_k^{\Theta}\big)\Big\|_{\mathrm{TV}}.
\end{align}
We now bound the first term in Eq.~\ref{eq:termII-change-measure} using Lemma 3.5 \cite{McYuk}, which gives
\begin{align}\label{eq:termII-lemma32-part}
\bE_{\hat s\sim Q_{w^\star}(P_{\theta^\star}^{a_k} b_k^{\Theta})}
\Big[\big\|\psi_{\Phi_{w^\star}}(P_{\theta}^{a_k} b_k^{\Theta},\hat s)-\psi_{\Phi_{w^\star}}(P_{\theta^\star}^{a_k} b_k^{\Theta},\hat s)\big\|_{\mathrm{TV}}\Big]
&\le (2-\kappa_\Phi)\,\|P_{\theta}^{a_k} b_k^{\Theta}-P_{\theta^\star}^{a_k} b_k^{\Theta}\|_{\mathrm{TV}}\notag\\
&= (2-\kappa_\Phi)\,\big\|P_{\theta}^{a_k}b_k^\Theta-P_{\theta^\star}^{a_k}b_k^\Theta\big\|_{\mathrm{TV}}\notag\\
&\le \frac{B}{2}(2-\kappa_\Phi)\delta(\theta),
\end{align}
where we have used
\begin{align*}
\bigl\|P_\theta^{a_k} b - P_{\theta^\star}^{a_k} b\bigr\|_{\mathrm{TV}}
&= \frac12 \sum_{s'\in\bS} \left| \sum_{s\in\bS} b(s)\,\Big(P_\theta(s'\mid s,a_k)-P_{\theta^\star}(s'\mid s,a_k)\Big) \right|\\[4pt]
&\le \frac12 \sum_{s'\in\bS} \sum_{s\in\bS} b(s)\, \Big|P_\theta(s'\mid s,a_k)-P_{\theta^\star}(s'\mid s,a_k)\Big| \\[4pt]
&= \frac12 \sum_{s\in\bS} b(s)\,\sum_{s'\in\bS} \Big|P_\theta(s'\mid s,a_k)-P_{\theta^\star}(s'\mid s,a_k)\Big|\\[4pt]
&\le\sup_{s\in\bS} \frac12 \sum_{s'\in\bS} \Big|P_\theta(s'\mid s,a_k)-P_{\theta^\star}(s'\mid s,a_k)\Big| \\[4pt]
&= \sup_{s\in\bS} \frac12\, \bigl\|P_\theta(\cdot\mid s,a_k)-P_{\theta^\star}(\cdot\mid s,a_k)\bigr\|_1\\[4pt]
&\le \frac{B}{2}\delta(\theta).
\end{align*}
It remains to bound the second term in Eq.~\ref{eq:termII-change-measure}. Using the Dobrushin coefficient of \(Q_{w^\star}\) and then of \(P_{\theta^\star}^{a_k}\) gives
\begin{align}\label{eq:termII-measure-mismatch-final}
\Big\|Q_{w^\star}\!\big(P_{\theta^\star}^{a_k} b_k^{\Theta^\star}\big)-Q_{w^\star}\!\big(P_{\theta^\star}^{a_k} b_k^{\Theta}\big)\Big\|_{\mathrm{TV}}
&\le (1-\kappa_\Phi)\,\big\|P_{\theta^\star}^{a_k} b_k^{\Theta^\star}-P_{\theta^\star}^{a_k} b_k^{\Theta}\big\|_{\mathrm{TV}}\notag\\
&\le (1-\kappa_\Phi)(1-\kappa_P)\,\|b_k^{\Theta^\star}-b_k^{\Theta}\|_{\mathrm{TV}}.
\end{align}
Combining Eq.~\ref{eq:termII-lemma32-part}--Eq.~\ref{eq:termII-measure-mismatch-final} yields
\begin{align}\label{eq:termII-final-bound}
\bE\big[\mathrm{(II)}\mid \mathcal F_k\big]
\le \frac{B}{2}(2-\kappa_\Phi)\,\delta(\theta)
+(1-\kappa_\Phi)(1-\kappa_P)\,\|b_k^{\Theta^\star}-b_k^{\Theta}\|_{\mathrm{TV}}.
\end{align}

\begin{remark}[Dominance condition]
We note that the dominance condition \(P^{a_k}_{\theta^\star}b^{\Theta}_k \ll P^{a_k}_{\theta}b^\Theta_k\) required for the application of filter stability results (e.g., Lemma 3.5 in \citet{McYuk}) is automatically satisfied for all \(k\) in our setting. Under the log-linear parameterization in Assumption \ref{as:Model-LogLinear}, the transition kernel \(P_\theta\) and observation model \(\Phi_w\), due to the softmax structure, are strictly positive everywhere, ensuring that \(P^{a_k}_{\theta}b^\Theta_k\) and \(P^{a_k}_{\theta^\star}b^{\Theta}_k\) remain mutually absolutely continuous.
\end{remark}

\paragraph{Term (III).} This term measures the propagation of error in belief for one time step. Under Assumption \ref{as:Dobrushin-stability}, and by Theorems 3.6 and 4.1 of \cite{McYuk}, for one step we have
\begin{align}\label{eq:one-step-ex-bound}
 \bE[\,\mathrm{(III)}\,|\, \mathcal{F}_k]&\le(1-\kappa_P)(2-\kappa_\Phi)\|b^\Theta_k-b_k^{\Theta^\star}\|_{\mathrm{TV}}.
\end{align}
Since \(b_0^\Theta=b_0^{\Theta^\star}=\nu_0\), we have
\[
\bE\Big[\|b_{0}^\Theta-b_{0}^{\Theta^\star}\|_{\mathrm{TV}}\Big]
=\|\nu_0-\nu_0\|_{\mathrm{TV}}=0.
\]
Putting all together, and unrolling over time yields
\begin{align}\label{eq:them-belief-lip-recursion}
\bE\,\Big[\|b_{k+1}^\Theta-b_{k+1}^{\Theta^\star}\|_{\mathrm{TV}}\Big]
&= \bE\Bigl[\bE\,\Big[\|b_{k+1}^\Theta-b_{k+1}^{\Theta^\star}\|_{\mathrm{TV}}\,\Big|\,\mathcal{F}_{k}\Big]\Bigr]\notag\\[4pt]
&\le \bE\Big[\bE\big[(\mathrm{I})\mid \mathcal F_k\big]+\bE\big[(\mathrm{II})\mid \mathcal F_k\big]+\bE\big[(\mathrm{III})\mid \mathcal F_k\big]\Big]\notag\\[4pt]
&\le \Big((1-\kappa_P)(2-\kappa_\Phi)+2(1-\kappa_P)(1-\kappa_\Phi)\Big)\,
\bE\Big[\|b_k^\Theta-b_k^{\Theta^\star}\|_{\mathrm{TV}}\Big]\notag\\
&\qquad\qquad\quad
+B\,\delta(w)+\frac{B}{2}\bigl((2-\kappa_\Phi)+(1-\kappa_\Phi)\bigr)\delta(\theta)\notag\\[4pt]
&=\alpha\,
\bE\Big[\|b_k^\Theta-b_k^{\Theta^\star}\|_{\mathrm{TV}}\Big]
+B\,\delta(w)+\frac{B}{2}(3-2\kappa_\Phi)\,\delta(\theta),
\end{align}
with \(\alpha=(1-\kappa_P)(4-3\kappa_\Phi)\). Unrolling Eq.~\ref{eq:them-belief-lip-recursion} and using \(\bE[\|b_0^\Theta-b_0^{\Theta^\star}\|_{\mathrm{TV}}]=0\) gives, for all \(k\ge 1\),
\begin{align}\label{eq:them-belief-lip-unrolled-final}
\bE\,\Big[\|b_{k}^\Theta-b_{k}^{\Theta^\star}\|_{\mathrm{TV}}\Big]
&\le \left(B\,\delta(w)+\frac{B}{2}(3-2\kappa_\Phi)\,\delta(\theta)\right)
\sum_{j=0}^{k-1}\alpha^{j}\notag\\[3pt]
&=
B\left(\delta(w)+\frac{3-2\kappa_\Phi}{2}\,\delta(\theta)\right)
\frac{1-\alpha^{k}}{1-\alpha}.
\end{align}
This concludes the proof.
\end{proof}

%%%%%%%%%%%%%%%%%%%%%%%%%%%%
%%%%%%%%%%%%%%%%%%%%%%%%%%%%
%%%%%%%%%%%%%%%%%%%%%%%%%%%%
%%%%%%%%%%%%%%%%%%%%%%%%%%%%
%%%%%%%%%%%%%%%%%%%%%%%%%%%%
%%%%%%%%%%%%%%%%%%%%%%%%%%%%
%%%%%%%%%%%%%%%%%%%%%%%%%%%%
%%%%%%%%%%%%%%%%%%%%%%%%%%%%

\section{Extension of Belief Stability to Neural-Softmax POMDP Models}\label{sec:nn-extension}

The linear approximation of the POMDP model studied earlier provides a setting in which the dependence of the POMDP kernels on the learned parameters can be controlled explicitly. However, log-linear scores may be too restrictive for complex models. We therefore consider a neural-softmax extension, where the scores defining \(P(\cdot\mid s,a)\) and \(\Phi(\cdot\mid s)\) are represented by neural networks.

This extension is motivated by standard results on over-parameterized neural networks. Such networks can interpolate training data under first-order optimization methods while still exhibiting strong generalization behavior \citep{Belkin2019,Bartlett2020,Zhang2021}. A central theoretical explanation is the lazy-training or neural tangent kernel (NTK) regime, where sufficiently wide networks trained near random initialization behave approximately like their first-order linearization around initialization \citep{JacotNTK2018,LiLiang2018,Du2019,OymakSoltanolkotabi2019,ChizatBach2019}. In this regime,
\[
F(x;W)\approx F(x;W_0)+\big\langle \nabla_W F(x;W_0),\,W-W_0\big\rangle ,
\]
so the tangent feature map \(\nabla_W F(x;W_0)\) plays the role of an effective feature representation. In the infinite-width limit, these tangent features induce an RKHS description, while in the local lazy-training regime the nonlinear network output is well approximated by its first-order linearization, with a finite-width error that vanishes as \(m\to\infty\) under standard assumptions.

In this section, we extend our belief-stability mechanism beyond the log-linear class. The neural-softmax model can be viewed as an approximate log-linear model in its NTK tangent features, with additional finite-width linearization errors \(\varepsilon_{p}^{\mathrm{NN}}(m,\,\delta_{\mathrm{NN}})\) and \(\varepsilon_\Phi^{\mathrm{NN}}(m, \delta_{\mathrm{NN}})\). These errors additively contribute to the belief perturbation bound and vanish as the lazy-training approximation improves.

\medskip
\begin{definition}[Symmetric random initialization]
\label{def:symmetric-random-initialization}
Assume the fixed width \(m\) is even. For each model component \(g\in\{p,\Phi\}\), let \(d_{x,g}\) denote the input dimension.
For \(i=1,\ldots,m/2\), sample independently
\(c_{g,i}\sim \operatorname{Rad},\,
\omega_{g,i,0}\sim \mathcal N(0,I_{d_{x,g}}),
\) and define the second half of the initialization by
\[
c_{g,i+m/2}:=-c_{g,i},
\qquad \omega_{g,i+m/2,0}:=\omega_{g,i,0}, \qquad i=1,\ldots,m/2 .
\]
Then,
\(
W_{g,0}:=(\omega_{g,1,0},\ldots,\omega_{g,m,0})
\) is called a symmetric random initialization for the two-layer ReLU network. 
\end{definition}

\medskip

\begin{definition}[NTK-linearized scores and tangent feature maps]
\label{def:nn-linearized-scores}
Let \(W_{p,0}\) and \(W_{\Phi,0}\) be reference initialization points for the transition and observation networks with parameters \(W_p\) and \(W_\Phi\), respectively. Let \(
x_p(s,a,s')\in\bR^{d_{x,p}},
x_\Phi(s,\hat s)\in\bR^{d_{x,\Phi}}
\) denote fixed input encodings of the transition and observation tuples, respectively. We write the trainable first-layer weights as
\[
W_p=(\omega_{p,1},\ldots,\omega_{p,m})\in\bR^{m\times d_{x,p}},
\qquad
W_\Phi=(\omega_{\Phi,1},\ldots,\omega_{\Phi,m})\in\bR^{m\times d_{x,\Phi}},
\]
and fix coefficients \(c_{p,i},c_{\Phi,i}\in\{-1,+1\}\). The neural scores are defined as
\begin{align}
F_p(s,a,s';W_p)
&:=\frac{1}{\sqrt m}\sum_{i=1}^m c_{p,i}\,
\operatorname{ReLU}\!\left(\omega_{p,i}^{\top}x_p(s,a,s')\right),
\label{eq:shallow-relu-score-p}\\
F_\Phi(s,\hat s;W_\Phi)
&:=\frac{1}{\sqrt m}\sum_{i=1}^m c_{\Phi,i}\,
\operatorname{ReLU}\!\left(\omega_{\Phi,i}^{\top}x_\Phi(s,\hat s)\right),
\label{eq:shallow-relu-score-Phi}
\end{align}
Throughout this section, norms on
\(W_p\) and \(W_\Phi\) are Frobenius norms. Equivalently, after vectorization, one may regard
\(W_p\in\bR^{d_p}\) and \(W_\Phi\in\bR^{d_\Phi}\), where
\(d_p=md_{x,p}\) and \(d_\Phi=md_{x,\Phi}\).
\paragraph{Linearization.} Define the first-order NTK-linearized scores around initialization by
\begin{align}
F_{p,W_p}^{\mathrm{lin}}(s,a,s')
&:=F_p(s,a,s';W_{p,0})+\big\langle\nabla_{W_p}F_p(s,a,s';W_{p,0}), W_p-W_{p,0} \big\rangle_F, \label{eq:nn-lin-score-p}\\
F_{\Phi,W_\Phi}^{\mathrm{lin}}(s,\hat s) &:=F_\Phi(s,\hat s;W_{\Phi,0})+\big\langle \nabla_{W_\Phi}F_\Phi(s,\hat s;W_{\Phi,0}), W_\Phi-W_{\Phi,0} \big\rangle_F.\label{eq:nn-lin-score-Phi}
\end{align}
The corresponding finite-width tangent feature maps are defined as
\begin{align}
\phi_p^{\mathrm{NN}}(s,a,s')
&:=\nabla_{W_p}F_p(s,a,s';W_{p,0}),\label{eq:def-tangent-p-nn}\\
\phi_\Phi^{\mathrm{NN}}(s,\hat s) &:=\nabla_{W_\Phi}F_\Phi(s,\hat s;W_{\Phi,0}).\label{eq:def-tangent-Phi-nn}
\end{align}
\end{definition}

\begin{assumption}[Over-parameterized neural-softmax model]
\label{as:nn-model-linearization}
Let \(\mathcal D=\mathcal D_p\times\mathcal D_\Phi\) be a bounded convex local lazy-training neighborhood around the symmetric initialization. The neural-softmax model in Definition~\ref{def:nn-softmax-model} satisfies the following conditions:
\begin{itemize}
    \item \textbf{Realizability.} 
    There exists \(W^\star=(W_p^\star,W_\Phi^\star)\in\mathcal D\) generating the true POMDP.

    \item \textbf{Initialization.}
    The reference points \(W_{p,0}\) and \(W_{\Phi,0}\) are symmetric initializations as in Definition \ref{def:symmetric-random-initialization}, for the transition and observation networks, respectively. Under symmetric construction we have
    \begin{align}
    F_p(s,a,s';W_{p,0})=0,
    \qquad
    F_\Phi(s,\hat s;W_{\Phi,0})=0,
    \label{eq:init-zero-nn-scores}
    \end{align}
    for all \((s,a,s')\in\bS\times\bA\times\bS\) and \((s,\hat s)\in\bS\times\hat\bS\).
    \item \textbf{Local lazy-training neighborhood.}
    There exist constants \(R_p,R_\Phi<\infty\) independent of the width \(m\), such that
    \[
    \sup_{W_p\in\mathcal D_p}\max_{i\in[m]}
    \|\omega_{p,i}-\omega_{p,i,0}\|_2
    \le \frac{R_p}{\sqrt m},
    \qquad
    \sup_{W_\Phi\in\mathcal D_\Phi}\max_{i\in[m]}
    \|\omega_{\Phi,i}-\omega_{\Phi,i,0}\|_2
    \le \frac{R_\Phi}{\sqrt m}.
    \]

    \item \textbf{Uniformly bounded tangent features.}
    For some \(B^{\mathrm{NN}}<\infty\), the tangent feature maps in Definition~\ref{def:nn-linearized-scores} satisfy
    \begin{align}
    \sup_{s\in\bS,\;a\in\bA,\;s'\in\bS}
    \|\phi_p^{\mathrm{NN}}(s,a,s')\|_F
    &\le B^{\mathrm{NN}},\label{eq:bounded-tangent-p-nn}\\
    \sup_{s\in\bS,\;\hat s\in\hat\bS}\|\phi_\Phi^{\mathrm{NN}}(s,\hat s)\|_F &\le B^{\mathrm{NN}}.
    \label{eq:bounded-tangent-Phi-nn}
    \end{align}
    \end{itemize}
\end{assumption}

\paragraph{Linearization event.} For the local lazy-training neighborhood \(\mathcal D=\mathcal D_p\times\mathcal D_\Phi\) in Assumption~\ref{as:nn-model-linearization} and any \(\delta_{\mathrm{NN}}\in(0,1)\), define the high-probability linearization event

\begin{align}
\mathcal E_{\mathrm{lin}}^{\mathrm{NN}}(\delta_{\mathrm{NN}})
:=\Bigg\{&\sup_{W_p\in\mathcal D_p}\sup_{s,a,s'}
\left|F_p(s,a,s';W_p)-F_{p,W_p}^{\mathrm{lin}}(s,a,s')\right|
\le \varepsilon_p^{\mathrm{NN}}(m,\delta_{\mathrm{NN}}), \notag\\
&\sup_{W_\Phi\in\mathcal D_\Phi}\sup_{s,\hat s} \left|F_\Phi(s,\hat s;W_\Phi)-F_{\Phi,W_\Phi}^{\mathrm{lin}}(s,\hat s)\right| \le \varepsilon_\Phi^{\mathrm{NN}}(m,\delta_{\mathrm{NN}})
\Bigg\},\label{eq:nn-linearization-event}
\end{align}
where \(\bP(\mathcal E_{\mathrm{lin}}^{\mathrm{NN}}(\delta_{\mathrm{NN}}))\ge 1-\delta_{\mathrm{NN}}\).

\medskip
\medskip

\begin{remark}[Neural linearization event]\label{rem:nn-linearization-event}
Under Assumption~\ref{as:nn-model-linearization}, the event $\mathcal E_{\mathrm{lin}}^{\mathrm{NN}}(\delta_{\mathrm{NN}})$ is used to replace the neural-softmax scores in Definition~\ref{def:nn-softmax-model} by their NTK-linearized counterparts in Definition~\ref{def:nn-linearized-scores}. This event follows from standard two-layer ReLU arguments around symmetric initialization (see Lemma~2 of \citet{cayci2022finite}) that gives the corresponding local linearization control, while Lemma~4.1(7) of \citet{satpathi2020role} provides the required uniform control over bounded inputs. Applying these bounds separately to the transition and observation score networks, and taking a union bound over the two initializations, yields $\bP(\mathcal E_{\mathrm{lin}}^{\mathrm{NN}}(\delta_{\mathrm{NN}}))\ge 1-\delta_{\mathrm{NN}}$ with, for universal constants $C_p,C_\Phi>0$,
\[
\varepsilon_p^{\mathrm{NN}}(m,\delta_{\mathrm{NN}})\le
C_p\frac{R_p}{\sqrt{m}}\left(R_p+\sqrt{\log\frac{2}{\delta_{\mathrm{NN}}}}+\sqrt{d_{x,p}\log m}\right)\]
\[
\varepsilon_\Phi^{\mathrm{NN}}(m,\delta_{\mathrm{NN}})\le C_\Phi\frac{R_\Phi}{\sqrt{m}} \left(R_\Phi+\sqrt{\log\frac{2}{\delta_{\mathrm{NN}}}}+\sqrt{d_{x,\Phi}\log m}\right).
\]
Therefore, for fixed input dimensions \(d_{x,p}\) and \(d_{x,\Phi}\),
\(
\varepsilon_p^{\mathrm{NN}}(m,\delta_{\mathrm{NN}})=\widetilde O(m^{-1/2}),\text{ and } \varepsilon_\Phi^{\mathrm{NN}}(m,\delta_{\mathrm{NN}})=\widetilde O(m^{-1/2}).
\)
Throughout the neural-softmax analysis, all bounds involving Lemma~\ref{lem:nn-kernel-lip} and Corollary~\ref{cor:nn-param-belief-lip} are understood on this event.
\end{remark}

\medskip
\medskip

\begin{lemma}[Neural-softmax kernel perturbation bounds]
\label{lem:nn-kernel-lip}
Fix \(\delta_{\mathrm{NN}}\in(0,1)\) and work on the event
\(\mathcal E_{\mathrm{lin}}^{\mathrm{NN}}(\delta_{\mathrm{NN}})\). Assume Assumption~\ref{as:nn-model-linearization}. For \(W_p,W_p'\in\mathcal D_p\) and \(W_\Phi,W_\Phi'\in\mathcal D_\Phi\), let, for \(W\in\mathcal D_p\) and \(W'\in\mathcal D_\Phi\),
\begin{equation}\label{eq:delta-w-nn}
\delta^{\mathrm{NN}}(W):=\|W-W_{p,0}\|_F,\qquad
\delta^{\mathrm{NN}}(W'):=\|W'-W_{\Phi,0}\|_F.
\end{equation}

Then,

\begin{align}
&\bigl\| P_{W_p}(\cdot\mid s,a) - P_{W_p'}(\cdot\mid s,a) \bigr\|_{\mathrm{TV}}
\le \frac{1}{2}B^{\mathrm{NN}}\|W_p-W_p'\|_F + \varepsilon_{p}^{\mathrm{NN}}(m,\,\delta_{\mathrm{NN}}),\qquad \forall (s,a)\in\bS\times\bA. \label{eq:nn-row-p-param}\\[6pt]
&\bigl\| \Phi_{W_\Phi}(\cdot\mid s) - \Phi_{W_\Phi'}(\cdot\mid s) \bigr\|_{\mathrm{TV}}
\le \frac{1}{2}B^{\mathrm{NN}}\|W_\Phi-W_\Phi'\|_F + \varepsilon_\Phi^{\mathrm{NN}}(m,\,\delta_{\mathrm{NN}}), \qquad \forall s\in\bS.
\label{eq:nn-row-Phi-param}\\[6pt]
&\bigl\| P_{W_p}(s'\mid\cdot,a) - P_{W_p'}(s'\mid\cdot,a) \bigr\|_1
\le 2B^{\mathrm{NN}}c_p\|W_p-W_p'\|_F + 4c_p\varepsilon_{p}^{\mathrm{NN}}(m,\,\delta_{\mathrm{NN}}),
\qquad\forall (s',a)\in\bS\times\bA.
\label{eq:nn-col-p-param}\\[6pt]
&\bigl\| \Phi_{W_\Phi}(\hat s\mid\cdot) - \Phi_{W_\Phi'}(\hat s\mid\cdot) \bigr\|_1
\le 2B^{\mathrm{NN}}c_\Phi\|W_\Phi-W_\Phi'\|_F
+4c_\Phi\varepsilon_\Phi^{\mathrm{NN}}(m,\,\delta_{\mathrm{NN}}),
\qquad \forall \hat s\in\hat\bS,
\label{eq:nn-col-Phi-param}
\end{align}
where
\[
c_p:=
\frac{|\bS|}
{1+(|\bS|-1)\exp\!\bigl(-2(B^{\mathrm{NN}}r_p^{\mathrm{NN}}+\varepsilon_{p}^{\mathrm{NN}}(m,\,\delta_{\mathrm{NN}}))\bigr)},\qquad
c_\Phi :=
\frac{|\bS|}
{1+(|\hat\bS|-1)\exp\!\bigl(-2(B^{\mathrm{NN}}r_\Phi^{\mathrm{NN}}+\varepsilon_\Phi^{\mathrm{NN}}(m,\,\delta_{\mathrm{NN}}))\bigr)},
\]
and
\begin{equation}\label{eq:delta-max-w-nn}
r_p^{\mathrm{NN}}:=\sup_{\bar W_p\in\mathcal D_p}\|\bar W_p-W_{p,0}\|_F,\qquad
r_\Phi^{\mathrm{NN}}:=\sup_{\bar W_\Phi\in\mathcal D_\Phi}\|\bar W_\Phi-W_{\Phi,0}\|_F.
\end{equation}

\end{lemma}

\begin{proof}
Define the linearized score mismatches
\begin{align}
\Delta_p^{\mathrm{lin}}(W_p,W_p') &:= \sup_{s\in\bS,\;a\in\bA,\;s'\in\bS} \left| F_{p,W_p}^{\mathrm{lin}}(s,a,s') - F_{p,W_p'}^{\mathrm{lin}}(s,a,s') \right|,
\label{eq:def-delta-p-lin}\\
\Delta_\Phi^{\mathrm{lin}}(W_\Phi,W_\Phi') &:= \sup_{s\in\bS,\;\hat s\in\hat\bS} \left| F_{\Phi,W_\Phi}^{\mathrm{lin}}(s,\hat s) - F_{\Phi,W_\Phi'}^{\mathrm{lin}}(s,\hat s) \right|.
\label{eq:def-delta-Phi-lin}
\end{align}
By the symmetric Xavier initialization in Assumption~\ref{as:nn-model-linearization}, Eq.~\ref{eq:nn-lin-score-p} becomes
\[
F_{p,W_p}^{\mathrm{lin}}(s,a,s') = \big\langle \phi_p^{\mathrm{NN}}(s,a,s'),\,W_p-W_{p,0}\big\rangle_F .
\]
Thus, given initialization, the transition kernel \(P_{W_p}^{\mathrm{lin}}(\cdot\mid s,a)\), obtained by using \(F_{p,W_p}^{\mathrm{lin}}\) in Eq.~\ref{eq:nn-softmax-transition}, is a log-linear model with feature map \(\phi_p^{\mathrm{NN}}\) and parameter \(W_p-W_{p,0}\).
Since \(P_{W_p}^{\mathrm{lin}}(\cdot\mid s,a)\) is a log-linear softmax model with feature map \(\phi_p^{\mathrm{NN}}\), parameter \(W_p-W_{p,0}\), and feature bound \(B^{\mathrm{NN}}\), the row-wise part of Lemma~\ref{lem:model-lip} applies with \(\theta\) replaced by \(W_p-W_{p,0}\). Therefore,
\begin{align}
\bigl\|P_{W_p}^{\mathrm{lin}}(\cdot\mid s,a) - P_{W_p'}^{\mathrm{lin}}(\cdot\mid s,a)\bigr\|_{\mathrm{TV}}
&\le \frac{B^{\mathrm{NN}}}{2}\|W_p-W_p'\|_F .
\label{eq:aux-lin-param-bound-tv}
\end{align}
Moreover, by the same row-wise argument used in Lemma~\ref{lem:model-lip}, applied to the score vectors, we also have
\begin{align}
\bigl\|P_{W_p}^{\mathrm{lin}}(\cdot\mid s,a) - P_{W_p'}^{\mathrm{lin}}(\cdot\mid s,a)\bigr\|_{\mathrm{TV}}
&\le \frac{1}{2}\Delta_p^{\mathrm{lin}}(W_p,W_p').
\label{eq:aux-lin-score-bound-tv}
\end{align}
Moreover, the definition of \(F_{p,W_p}^{\mathrm{lin}}\) and the uniform tangent-feature bound imply
\begin{align}
\Delta_p^{\mathrm{lin}}(W_p,W_p')
&\le B^{\mathrm{NN}}\|W_p-W_p'\|_F ,
\label{eq:aux-delta-param-tv}
\end{align}
then,
\[
\left|F_{p,W_p}^{\mathrm{lin}}(s,a,s')-F_{p,W_p'}^{\mathrm{lin}}(s,a,s')\right|
=
\left|\big\langle \phi_p^{\mathrm{NN}}(s,a,s'),\,W_p-W_p'\big\rangle_F\right|
\le B^{\mathrm{NN}}\|W_p-W_p'\|_F .
\]

Next, by the definition of the score-linearization error,
\[
\sup_{s,a,s'}\left| F_p(s,a,s';W_p)-F_{p,W_p}^{\mathrm{lin}}(s,a,s') \right| \le \varepsilon_{p}^{\mathrm{NN}}(m,\,\delta_{\mathrm{NN}}),
\]
and the same bound holds with \(W_p'\) in place of \(W_p\). Applying the same row-wise score argument used in Lemma~\ref{lem:model-lip} to the nonlinear and linearized score vectors gives
\begin{align}
\bigl\|P_{W_p}(\cdot\mid s,a)-P_{W_p}^{\mathrm{lin}}(\cdot\mid s,a)\bigr\|_{\mathrm{TV}}
&\le \frac{1}{2}\varepsilon_{p}^{\mathrm{NN}}(m,\,\delta_{\mathrm{NN}}),
\label{eq:aux-nonlinear-lin-p-tv}\\
\bigl\|P_{W_p'}(\cdot\mid s,a)-P_{W_p'}^{\mathrm{lin}}(\cdot\mid s,a)\bigr\|_{\mathrm{TV}}
&\le \frac{1}{2}\varepsilon_{p}^{\mathrm{NN}}(m,\,\delta_{\mathrm{NN}}).
\label{eq:aux-nonlinear-lin-p-prime-tv}
\end{align}
Combining Eq.~\ref{eq:aux-lin-score-bound-tv}, Eq.~\ref{eq:aux-nonlinear-lin-p-tv}, and Eq.~\ref{eq:aux-nonlinear-lin-p-prime-tv} by the triangle inequality yields
\begin{align*}
\bigl\|P_{W_p}(\cdot\mid s,a)-P_{W_p'}(\cdot\mid s,a)\bigr\|_{\mathrm{TV}}
&\le \bigl\|P_{W_p}(\cdot\mid s,a)-P_{W_p}^{\mathrm{lin}}(\cdot\mid s,a)\bigr\|_{\mathrm{TV}} \\
&\quad+ \bigl\|P_{W_p}^{\mathrm{lin}}(\cdot\mid s,a)-P_{W_p'}^{\mathrm{lin}}(\cdot\mid s,a)\bigr\|_{\mathrm{TV}} \\
&\quad+ \bigl\|P_{W_p'}^{\mathrm{lin}}(\cdot\mid s,a)-P_{W_p'}(\cdot\mid s,a)\bigr\|_{\mathrm{TV}}\\
&\le \frac{1}{2}\Delta_p^{\mathrm{lin}}(W_p,W_p')+\varepsilon_{p}^{\mathrm{NN}}(m,\,\delta_{\mathrm{NN}}).
\end{align*}
Using Eq.~\ref{eq:aux-delta-param-tv} proves the claim.

It remains to prove the column-wise transition bound. For any \(W_p\in\mathcal D_p\), the initialization condition and tangent-feature bound imply
\[
\left|F_{p,W_p}^{\mathrm{lin}}(s,a,s')\right|
=
\left|\big\langle \phi_p^{\mathrm{NN}}(s,a,s'),\,W_p-W_{p,0}\big\rangle_F\right|
\le B^{\mathrm{NN}}\delta^{\mathrm{NN}}(W_p)
\le B^{\mathrm{NN}}r_p^{\mathrm{NN}}.
\]
Hence, for every \(s,s'\in\bS\),
\[
P_{W_p}^{\mathrm{lin}}(s'\mid s,a)
\le
\frac{1}{1+(|\bS|-1)\exp\!\bigl(-2B^{\mathrm{NN}}r_p^{\mathrm{NN}}\bigr)}
\le
\frac{1}{1+(|\bS|-1)\exp\!\bigl(-2(B^{\mathrm{NN}}r_p^{\mathrm{NN}}+\varepsilon_{p}^{\mathrm{NN}}(m,\,\delta_{\mathrm{NN}}))\bigr)}.
\]
Therefore, \(
\sum_{s\in\bS}P_{W_p}^{\mathrm{lin}}(s'\mid s,a)\le c_p.
\) Repeating the column-wise Jacobian argument in Lemma~\ref{lem:model-lip}, now applied to the linearized model and using score mismatch \(\Delta_p^{\mathrm{lin}}(W_p,W_p')\), gives
\[
\bigl\|P_{W_p}^{\mathrm{lin}}(s'\mid\cdot,a)-P_{W_p'}^{\mathrm{lin}}(s'\mid\cdot,a)\bigr\|_1
\le 2c_p\,\Delta_p^{\mathrm{lin}}(W_p,W_p').
\]
Moreover, the score-linearization error implies
\[
\sup_{s,a,s'}|F_p(s,a,s';W_p)|\le B^{\mathrm{NN}}\delta^{\mathrm{NN}}(W_p)+\varepsilon_{p}^{\mathrm{NN}}(m,\,\delta_{\mathrm{NN}})\le B^{\mathrm{NN}}r_p^{\mathrm{NN}}+\varepsilon_{p}^{\mathrm{NN}}(m,\,\delta_{\mathrm{NN}}).
\]
Thus, for the same coordinate-wise softmax-Jacobian estimate,
\[
\bigl\|P_{W_p}(s'\mid\cdot,a)-P_{W_p}^{\mathrm{lin}}(s'\mid\cdot,a)\bigr\|_1
\le 2c_p\,\varepsilon_{p}^{\mathrm{NN}}(m,\,\delta_{\mathrm{NN}}),
\]
and the same bound holds with \(W_p'\) in place of \(W_p\). By the triangle inequality,
\begin{align*}
\bigl\|P_{W_p}(s'\mid\cdot,a)-P_{W_p'}(s'\mid\cdot,a)\bigr\|_1
&\le \bigl\|P_{W_p}(s'\mid\cdot,a)-P_{W_p}^{\mathrm{lin}}(s'\mid\cdot,a)\bigr\|_1\\
&\quad+\bigl\|P_{W_p}^{\mathrm{lin}}(s'\mid\cdot,a)-P_{W_p'}^{\mathrm{lin}}(s'\mid\cdot,a)\bigr\|_1\\
&\quad+\bigl\|P_{W_p'}^{\mathrm{lin}}(s'\mid\cdot,a)-P_{W_p'}(s'\mid\cdot,a)\bigr\|_1\\
&\le 2c_p\,\Delta_p^{\mathrm{lin}}(W_p,W_p')+4c_p\,\varepsilon_{p}^{\mathrm{NN}}(m,\,\delta_{\mathrm{NN}}).
\end{align*}
Using Eq.~\ref{eq:aux-delta-param-tv} gives the parameter-distance version.

The proof for the observation kernel is the same, replacing the transition score family \(F_p\) by \(F_\Phi\), the parameter \(W_p\) by \(W_\Phi\), the tangent feature map \(\phi_p^{\mathrm{NN}}\) by \(\phi_\Phi^{\mathrm{NN}}\), and the softmax output space \(\bS\) by \(\hat\bS\). The resulting fixed-\(\hat s\) column is summed over \(s\in\bS\), which gives the constant \(c_\Phi\). This proves Eq.~\ref{eq:nn-row-Phi-param}--Eq.~\ref{eq:nn-col-Phi-param}.
\end{proof}

\begin{lemma}[Neural-softmax column sums and Dobrushin bounds]\label{lem:nn-Dobrushin-row-col}
Fix \(\delta_{\mathrm{NN}}\in(0,1)\) and work on the event
\(\mathcal E_{\mathrm{lin}}^{\mathrm{NN}}(\delta_{\mathrm{NN}})\).
Assume Assumption~\ref{as:nn-model-linearization}. Then, the following properties hold.
\begin{equation}\label{eq:nn-column-sum-upperbound}
\sup_{W_p\in\mathcal D_p}\sup_{a\in\bA}\sup_{s'\in\bS}\sum_{s\in\bS}P_{W_p}(s'\mid s,a)
\le
\frac{|\bS|}{1+(|\bS|-1)\exp\!\bigl(-2(B^{\mathrm{NN}}r_p^{\mathrm{NN}}+\varepsilon_{p}^{\mathrm{NN}}(m,\,\delta_{\mathrm{NN}}))\bigr)} .
\end{equation}
\begin{equation}\label{eq:nn-column-sum-lowerbound}
\inf_{W_p\in\mathcal D_p}\inf_{a\in\bA}\inf_{s'\in\bS}\sum_{s\in\bS}P_{W_p}(s'\mid s,a)
\ge
\frac{|\bS|}{1+(|\bS|-1)\exp\!\bigl(2(B^{\mathrm{NN}}r_p^{\mathrm{NN}}+\varepsilon_{p}^{\mathrm{NN}}(m,\,\delta_{\mathrm{NN}}))\bigr)} .
\end{equation}
\begin{equation}\label{eq:nn-Phi-column-sum-upperbound}
\sup_{W_\Phi\in\mathcal D_\Phi}\sup_{\hat s\in\hat\bS}\sum_{s\in\bS}\Phi_{W_\Phi}(\hat s\mid s)
\le
\frac{|\bS|}{1+(|\hat\bS|-1)\exp\!\bigl(-2(B^{\mathrm{NN}}r_\Phi^{\mathrm{NN}}+\varepsilon_\Phi^{\mathrm{NN}}(m,\,\delta_{\mathrm{NN}}))\bigr)} .
\end{equation}
\begin{equation}\label{eq:nn-Phi-column-sum-lowerbound}
\inf_{W_\Phi\in\mathcal D_\Phi}\inf_{\hat s\in\hat\bS}\sum_{s\in\bS}\Phi_{W_\Phi}(\hat s\mid s)
\ge
\frac{|\bS|}{1+(|\hat\bS|-1)\exp\!\bigl(2(B^{\mathrm{NN}}r_\Phi^{\mathrm{NN}}+\varepsilon_\Phi^{\mathrm{NN}}(m,\,\delta_{\mathrm{NN}}))\bigr)} .
\end{equation}
\medskip\noindent
Furthermore, the uniform neural Dobrushin coefficients defined as
\[
\kappa_P^{\mathrm{NN}}
:=
\inf_{W_p\in\mathcal D_p}\inf_{a\in\bA}\inf_{s,s''\in\bS}
\sum_{x\in\bS}\min\{P_{W_p}(x\mid s,a),P_{W_p}(x\mid s'',a)\},
\]
\[
\kappa_\Phi^{\mathrm{NN}}
:=
\inf_{W_\Phi\in\mathcal D_\Phi}\inf_{s,s''\in\bS}
\sum_{\hat s\in\hat\bS}\min\{\Phi_{W_\Phi}(\hat s\mid s),\Phi_{W_\Phi}(\hat s\mid s'')\},
\]

satisfy
\begin{equation}\label{eq:nn-dobrushin-row}
\kappa_P^{\mathrm{NN}}
\ge
\frac{|\bS|}{1+(|\bS|-1)\exp\!\bigl(2(B^{\mathrm{NN}}r_p^{\mathrm{NN}}+\varepsilon_{p}^{\mathrm{NN}}(m,\,\delta_{\mathrm{NN}}))\bigr)} ,
\end{equation}
\begin{equation}\label{eq:nn-dobrushin-Phi}
\kappa_\Phi^{\mathrm{NN}}
\ge
\frac{|\hat\bS|}{1+(|\hat\bS|-1)\exp\!\bigl(2(B^{\mathrm{NN}}r_\Phi^{\mathrm{NN}}+\varepsilon_\Phi^{\mathrm{NN}}(m,\,\delta_{\mathrm{NN}}))\bigr)} .
\end{equation}

\end{lemma}

\begin{proof}[Proof of Lemma~\ref{lem:nn-Dobrushin-row-col}]
We first prove the claims for the transition kernel. Fix \((s,a,s')\in\bS\times\bA\times\bS\) and \(W_p\in\mathcal D_p\). By the symmetric initialization in Assumption~\ref{as:nn-model-linearization},
\[
F_{p,W_p}^{\mathrm{lin}}(s,a,s')=\big\langle \phi_p^{\mathrm{NN}}(s,a,s'),W_p-W_{p,0}\big\rangle_F .
\]
Hence, by the tangent-feature bound and the definition of \(r_p^{\mathrm{NN}}\),
\[
\big|F_{p,W_p}^{\mathrm{lin}}(s,a,s')\big|
\le \|\phi_p^{\mathrm{NN}}(s,a,s')\|_F\,\|W_p-W_{p,0}\|_F
\le B^{\mathrm{NN}}r_p^{\mathrm{NN}}.
\]
Using the score-linearization error in Definition~\ref{def:nn-linearized-scores}, we obtain the full nonlinear score bound
\[
|F_p(s,a,s';W_p)|
\le \big|F_{p,W_p}^{\mathrm{lin}}(s,a,s')\big|
+\big|F_p(s,a,s';W_p)-F_{p,W_p}^{\mathrm{lin}}(s,a,s')\big|
\le B^{\mathrm{NN}}r_p^{\mathrm{NN}}+\varepsilon_{p}^{\mathrm{NN}}(m,\,\delta_{\mathrm{NN}}).
\]
Thus the neural-softmax transition kernel satisfies the same uniform score bound as the log-linear kernel in Lemma~\ref{lem:Dobrushin-row-col}, with \(B(\|\theta^\star\|_2+r_\Theta)\) replaced by \(B^{\mathrm{NN}}r_p^{\mathrm{NN}}+\varepsilon_{p}^{\mathrm{NN}}(m,\,\delta_{\mathrm{NN}})\). Repeating the pointwise softmax upper-bound argument in Eq.~\ref{eq:kernel-uni-upperbound} gives, for all \((s,a,s')\),
\[
P_{W_p}(s'\mid s,a)
\le
\frac{1}{1+(|\bS|-1)\exp\!\bigl(-2(B^{\mathrm{NN}}r_p^{\mathrm{NN}}+\varepsilon_{p}^{\mathrm{NN}}(m,\,\delta_{\mathrm{NN}}))\bigr)}.
\]
Summing over \(s\in\bS\) and taking the supremum over \((W_p,a,s')\) proves Eq.~\ref{eq:nn-column-sum-upperbound}. Similarly, repeating the lower-bound argument in Eq.~\ref{eq:kernel-uni-lowerbound} gives
\[
P_{W_p}(s'\mid s,a)
\ge
\frac{1}{1+(|\bS|-1)\exp\!\bigl(2(B^{\mathrm{NN}}r_p^{\mathrm{NN}}+\varepsilon_{p}^{\mathrm{NN}}(m,\,\delta_{\mathrm{NN}}))\bigr)}
=:P_{\min}^{\mathrm{NN}}.
\]
Summing \(P_{\min}^{\mathrm{NN}}\) over \(s\in\bS\) and taking the infimum over \((W_p,a,s')\) proves Eq.~\ref{eq:nn-column-sum-lowerbound}. Moreover, since every entry of \(P_{W_p}(\cdot\mid s,a)\) is bounded from below by \(P_{\min}^{\mathrm{NN}}\), the same Dobrushin argument as in Eq.~\ref{eq:dobcoef-p-lowerbound} yields
\[
\sum_{x\in\bS}\min\{P_{W_p}(x\mid s,a),P_{W_p}(x\mid s'',a)\}
\ge
|\bS|P_{\min}^{\mathrm{NN}}
=
\frac{|\bS|}{1+(|\bS|-1)\exp\!\bigl(2(B^{\mathrm{NN}}r_p^{\mathrm{NN}}+\varepsilon_{p}^{\mathrm{NN}}(m,\,\delta_{\mathrm{NN}}))\bigr)}.
\]
Taking the infimum over \(W_p\in\mathcal D_p\), \(a\in\bA\), and \(s,s''\in\bS\) proves Eq.~\ref{eq:nn-dobrushin-row}.

\medskip\noindent
We now prove the observation bounds. Fix \((s,\hat s)\in\bS\times\hat\bS\) and \(W_\Phi\in\mathcal D_\Phi\). By the symmetric initialization,
\[
F_{\Phi,W_\Phi}^{\mathrm{lin}}(s,\hat s)
=
\big\langle \phi_\Phi^{\mathrm{NN}}(s,\hat s),W_\Phi-W_{\Phi,0}\big\rangle_F .
\]
The tangent-feature bound and the definition of \(r_\Phi^{\mathrm{NN}}\) imply
\[
\big|F_{\Phi,W_\Phi}^{\mathrm{lin}}(s,\hat s)\big|
\le \|\phi_\Phi^{\mathrm{NN}}(s,\hat s)\|_F\,\|W_\Phi-W_{\Phi,0}\|_F
\le B^{\mathrm{NN}}r_\Phi^{\mathrm{NN}}.
\]
Together with the score-linearization error, this gives
\(
|F_\Phi(s,\hat s;W_\Phi)|
\le B^{\mathrm{NN}}r_\Phi^{\mathrm{NN}}+\varepsilon_\Phi^{\mathrm{NN}}(m,\,\delta_{\mathrm{NN}}).
\)
Thus the neural-softmax observation kernel satisfies the same uniform score bound as the log-linear observation kernel in Lemma~\ref{lem:Dobrushin-row-col}, with \(B(\|w^\star\|_2+r_\Theta)\) replaced by \(B^{\mathrm{NN}}r_\Phi^{\mathrm{NN}}+\varepsilon_\Phi^{\mathrm{NN}}(m,\,\delta_{\mathrm{NN}})\). Repeating the pointwise upper- and lower-bound arguments in Eq.~\ref{eq:Phi-pointwise-upper}--Eq.~\ref{eq:Phi-pointwise-lower} gives, for all \((s,\hat s)\),
\[
\Phi_{W_\Phi}(\hat s\mid s)
\le
\frac{1}{1+(|\hat\bS|-1)\exp\!\bigl(-2(B^{\mathrm{NN}}r_\Phi^{\mathrm{NN}}+\varepsilon_\Phi^{\mathrm{NN}}(m,\,\delta_{\mathrm{NN}}))\bigr)}
=:\Phi_{\max}^{\mathrm{NN}},
\]
and
\[
\Phi_{W_\Phi}(\hat s\mid s)
\ge
\frac{1}{1+(|\hat\bS|-1)\exp\!\bigl(2(B^{\mathrm{NN}}r_\Phi^{\mathrm{NN}}+\varepsilon_\Phi^{\mathrm{NN}}(m,\,\delta_{\mathrm{NN}}))\bigr)}
=:\Phi_{\min}^{\mathrm{NN}}.
\]
Summing these pointwise bounds over \(s\in\bS\) and taking the supremum and infimum over \((W_\Phi,\hat s)\) proves Eq.~\ref{eq:nn-Phi-column-sum-upperbound} and Eq.~\ref{eq:nn-Phi-column-sum-lowerbound}, respectively. Finally, since every entry of \(\Phi_{W_\Phi}(\cdot\mid s)\) is bounded from below by \(\Phi_{\min}^{\mathrm{NN}}\), the same Dobrushin argument as in Eq.~\ref{eq:dobcoef-Phi-lowerbound} gives
\[
\sum_{\hat s\in\hat\bS}\min\{\Phi_{W_\Phi}(\hat s\mid s),\Phi_{W_\Phi}(\hat s\mid s'')\}
\ge
|\hat\bS|\Phi_{\min}^{\mathrm{NN}}
=
\frac{|\hat\bS|}{1+(|\hat\bS|-1)\exp\!\bigl(2(B^{\mathrm{NN}}r_\Phi^{\mathrm{NN}}+\varepsilon_\Phi^{\mathrm{NN}}(m,\,\delta_{\mathrm{NN}}))\bigr)}.
\]
Taking the infimum over \(W_\Phi\in\mathcal D_\Phi\) and \(s,s''\in\bS\) proves Eq.~\ref{eq:nn-dobrushin-Phi}.
\end{proof}

\section{Proof of Corollary \ref{cor:nn-param-belief-lip}}
\begin{proof}\label{proof:belief-lip-nn}
For \(a\in\bA\), write
\[
(P_{W_p}^{a}q)(s'):=\sum_{s\in\bS}P_{W_p}(s'\mid s,a)q(s),\qquad
(Q_{W_\Phi}q)(\hat s):=\sum_{s'\in\bS}\Phi_{W_\Phi}(\hat s\mid s')q(s').
\]
The belief recursions are
\[
b_{k+1}^{W}=\psi_{\Phi_{W_\Phi}}\big(P_{W_p}^{a_k}b_k^W,\hat s_{k+1}\big),\qquad
b_{k+1}^{W^\star}=\psi_{\Phi_{W_\Phi^\star}}\big(P_{W_p^\star}^{a_k}b_k^{W^\star},\hat s_{k+1}\big),
\qquad b_0^W=b_0^{W^\star}=\nu_0 .
\]
We work under \(\bP_{W^\star}(\cdot\mid a_{0:t-1})\), with the same filtrations \(\mathcal F_k^-\) and \(\mathcal F_k\) as in the proof of Theorem~\ref{thm:param-belief-lip}. Thus
\[
\hat s_{k+1}\mid\mathcal F_k\sim Q_{W_\Phi^\star}\!\left(P_{W_p^\star}^{a_k}b_k^{W^\star}\right).
\]
Define
\begin{align}
\ell_p^{\mathrm{NN}}
&:=\frac{B^{\mathrm{NN}}}{2}\|W_p-W_p^\star\|_F+\varepsilon_{p}^{\mathrm{NN}}(m,\,\delta_{\mathrm{NN}}),\label{eq:ell-p-nn}\\
\ell_\Phi^{\mathrm{NN}}
&:=B^{\mathrm{NN}}\|W_\Phi-W_\Phi^\star\|_F+2\varepsilon_\Phi^{\mathrm{NN}}(m,\,\delta_{\mathrm{NN}}).\label{eq:ell-Phi-nn}
\end{align}
By Lemma~\ref{lem:nn-kernel-lip}, for every \(b\in\Delta_{\bS}\) and \(a\in\bA\),
\begin{align}
\bigl\|P_{W_p}^{a}b-P_{W_p^\star}^{a}b\bigr\|_{\mathrm{TV}}
&\le \ell_p^{\mathrm{NN}},\label{eq:nn-transition-mixture-bound}\\
\sup_{s\in\bS}\bigl\|\Phi_{W_\Phi}(\cdot\mid s)-\Phi_{W_\Phi^\star}(\cdot\mid s)\bigr\|_1
&\le \ell_\Phi^{\mathrm{NN}}.\label{eq:nn-observation-row-l1-bound}
\end{align}

Fix \(k\in\{0,\dots,t-1\}\). As in the proof of Theorem~\ref{thm:param-belief-lip}, add and subtract
\[
\psi_{\Phi_{W_\Phi^\star}}(P_{W_p}^{a_k}b_k^W,\hat s_{k+1})
\quad\text{and}\quad
\psi_{\Phi_{W_\Phi^\star}}(P_{W_p^\star}^{a_k}b_k^W,\hat s_{k+1}),
\]
and denote the resulting three terms by \(\mathrm{(I)},\mathrm{(II)},\mathrm{(III)}\), corresponding respectively to observation-model perturbation, transition-model perturbation, and propagation of the previous belief error. The algebraic decomposition is identical to the decomposition in the proof of Theorem~\ref{thm:param-belief-lip}.

For Term \(\mathrm{(I)}\), the Bayes-operator perturbation argument in Eq.~\ref{eq:bayes-op-param-perturb}--Eq.~\ref{eq:exp-denominator-eliminate}, with \(\Phi_w,\Phi_{w^\star}\) replaced by \(\Phi_{W_\Phi},\Phi_{W_\Phi^\star}\), gives
\[
\bE_{\hat s\sim Q_{W_\Phi^\star}b}
\Big[\|\psi_{\Phi_{W_\Phi}}(b,\hat s)-\psi_{\Phi_{W_\Phi^\star}}(b,\hat s)\|_{\mathrm{TV}}\Big]
\le \ell_\Phi^{\mathrm{NN}}.
\]
The same change-of-measure step as in Eq.~\ref{eq:termI-main-belief}, together with the Dobrushin contraction of \(Q_{W_\Phi^\star}\), the contraction of \(P_{W_p^\star}^{a_k}\), and Eq.~\ref{eq:nn-transition-mixture-bound}, yields
\begin{align}
\bE[\mathrm{(I)}\mid\mathcal F_k]
&\le \ell_\Phi^{\mathrm{NN}}+(1-\kappa_\Phi^{\mathrm{NN}})\ell_p^{\mathrm{NN}}
+(1-\kappa_P^{\mathrm{NN}})(1-\kappa_\Phi^{\mathrm{NN}})
\|b_k^{W^\star}-b_k^W\|_{\mathrm{TV}}. \label{eq:nn-I-final-bound}
\end{align}

For Term \(\mathrm{(II)}\), applying Lemma 3.5 of \citet{McYuk} as in Eq.~\ref{eq:termII-lemma32-part}, and using Eq.~\ref{eq:nn-transition-mixture-bound}, gives
\[
\bE_{\hat s\sim Q_{W_\Phi^\star}(P_{W_p^\star}^{a_k}b_k^W)}
\Big[\|\psi_{\Phi_{W_\Phi^\star}}(P_{W_p}^{a_k}b_k^W,\hat s)-\psi_{\Phi_{W_\Phi^\star}}(P_{W_p^\star}^{a_k}b_k^W,\hat s)\|_{\mathrm{TV}}\Big]
\le (2-\kappa_\Phi^{\mathrm{NN}})\ell_p^{\mathrm{NN}}.
\]
The change-of-measure term is the same as in Eq.~\ref{eq:termII-measure-mismatch-final}, with the neural Dobrushin coefficients:
\[
\Big\|Q_{W_\Phi^\star}(P_{W_p^\star}^{a_k}b_k^{W^\star})-Q_{W_\Phi^\star}(P_{W_p^\star}^{a_k}b_k^W)\Big\|_{\mathrm{TV}}
\le (1-\kappa_\Phi^{\mathrm{NN}})(1-\kappa_P^{\mathrm{NN}})\|b_k^{W^\star}-b_k^W\|_{\mathrm{TV}}.
\]
Therefore,
\begin{align}
\bE[\mathrm{(II)}\mid\mathcal F_k]
&\le (2-\kappa_\Phi^{\mathrm{NN}})\ell_p^{\mathrm{NN}}
+(1-\kappa_\Phi^{\mathrm{NN}})(1-\kappa_P^{\mathrm{NN}})
\|b_k^{W^\star}-b_k^W\|_{\mathrm{TV}}. \label{eq:nn-II-final-bound}
\end{align}

For Term \(\mathrm{(III)}\), Theorems 3.6 and 4.1 of \citet{McYuk}, applied with the neural Dobrushin coefficients from Lemma~\ref{lem:nn-Dobrushin-row-col}, give
\begin{align}
\bE[\mathrm{(III)}\mid\mathcal F_k]
&\le (1-\kappa_P^{\mathrm{NN}})(2-\kappa_\Phi^{\mathrm{NN}})
\|b_k^W-b_k^{W^\star}\|_{\mathrm{TV}}. \label{eq:nn-III-final-bound}
\end{align}
The required dominance condition clearly holds, since the neural-softmax transition and observation kernels are strictly positive on finite spaces.

Combining Eq.~\ref{eq:nn-I-final-bound}--Eq.~\ref{eq:nn-III-final-bound} and taking outer expectation yields
\begin{align}
\bE\Big[\|b_{k+1}^W-b_{k+1}^{W^\star}\|_{\mathrm{TV}}\Big]
&\le \alpha_{\mathrm{NN}}\,
\bE\Big[\|b_k^W-b_k^{W^\star}\|_{\mathrm{TV}}\Big]
+\ell_\Phi^{\mathrm{NN}}+(3-2\kappa_\Phi^{\mathrm{NN}})\ell_p^{\mathrm{NN}}, \label{eq:nn-belief-recursion}
\end{align}
where \(\alpha_{\mathrm{NN}}=(1-\kappa_P^{\mathrm{NN}})(4-3\kappa_\Phi^{\mathrm{NN}})\). Since \(b_0^W=b_0^{W^\star}\), unrolling Eq.~\ref{eq:nn-belief-recursion} gives
\[
\bE\Big[\|b_t^W-b_t^{W^\star}\|_{\mathrm{TV}}\Big]
\le \Big(\ell_\Phi^{\mathrm{NN}}+(3-2\kappa_\Phi^{\mathrm{NN}})\ell_p^{\mathrm{NN}}\Big)
\frac{1-\alpha_{\mathrm{NN}}^t}{1-\alpha_{\mathrm{NN}}}.
\]
Substituting Eq.~\ref{eq:ell-p-nn} and Eq.~\ref{eq:ell-Phi-nn} proves Eq.~\ref{eq:nn-belief-lipschitz}.
\end{proof}

\medskip

\begin{remark}[Finite-width neural-network approximation]
\label{rem:nn-linearization-error-rate}
The quantities \(\varepsilon_p^{\mathrm{NN}}(m, \delta_{\mathrm{NN}})\) and \(\varepsilon_\Phi^{\mathrm{NN}}(m, \delta_{\mathrm{NN}})\) measure the finite-width error incurred by replacing the nonlinear ReLU scores with their first-order NTK linearizations around initialization. In the NTK regime, sufficiently over-parameterized networks remain close to initialization and their outputs are well approximated by this linearized model \citep{JacotNTK2018,ChizatBach2019,JiLiTelgarsky2021}. For two-layer ReLU networks with symmetric random initialization, high-probability local linearization bounds imply that, on a neighborhood of radius \(O(m^{-1/2})\) around initialization, the approximation error decays as
\[
\varepsilon_p^{\mathrm{NN}}(m,\delta_{\mathrm{NN}})=\tilde O(m^{-1/2}),
\qquad \varepsilon_\Phi^{\mathrm{NN}}(m,\,\delta_{\mathrm{NN}})=\tilde O(m^{-1/2}),
\]
Consequently, under the stated local linearization event, the additive neural approximation terms in Corollary~\ref{cor:nn-param-belief-lip} vanish as \(m\to\infty\), recovering the same belief-stability mechanism as in the log-linear model, with the NTK tangent features playing the role of the finite-dimensional feature maps.
\end{remark}

%====================================================================================================================================================================

\section{Proof of Corollary \ref{cor:belief-good-condition-cor}}
\begin{proof}\label{proof:col-belief}
Consider $X_t:=\|b_t^\Theta-b_t^{\Theta^\star}\|_{\mathrm{TV}}\in[0,1]$ and use the filtration
$\{\mathcal F_t\}_{t\ge -1}$ defined by $\mathcal F_{-1}:=\sigma(\varnothing,\Omega)$ and
$\mathcal F_t:=\sigma(\hat s_{1:t},a_{0:t})$ for $t\ge 0$. As $\mathcal F_t^- \subseteq \mathcal F_t$, the random variable $X_t$ is also $\mathcal F_t$-measurable, and in particular, under any non-anticipative policy, $a_t$ is $\mathcal F_t^-$-measurable and
$\mathcal F_t=\sigma(\mathcal F_t^-,a_t)$. From one-step conditional bound established in the proof of Theorem \ref{thm:param-belief-lip}
(the bounds for terms (I)--(III) up to \ref{eq:them-belief-lip-recursion} without taking the outer expectation), for all $t\ge 1$, we proved

\begin{equation}\label{eq:belief-freedman-drift-bound}
\bE[X_t|\cF_{t-1}]\le\alpha X_{t-1} +c_b\,(\delta(\theta),\delta(w)).
\end{equation}
 
Define the martingale difference sequence $D_t:=\; X_t\,-\,\bE\,[X_t\mid\mathcal F_{t-1}]$ and the partial sums $S_n:=\;\sum_{t=0}^{n-1}D_t$ for $n=1, \dots, T$. Note that $D_0=X_0-\bE[X_0|\cF_{-1}]=0$ since $X_0=0$ as $b_0^\Theta=b_0^{\Theta^\star}=\nu_0$. Then $\big\{S_n,\,\cF_{n-1}\big\}_{n=1}^{T}$ is a martingale. Summing \ref{eq:belief-freedman-drift-bound} over $t=1,\dots,T-1$ yields
\begin{align}
\sum_{t=1}^{T-1}\bE \left[X_t\,\middle|\,\mathcal F_{t-1}\right]
&\le \alpha\sum_{t=1}^{T-1}X_{t-1}+(T-1)c_b\,(\delta(\theta),\delta(w))\notag\\
&\le\; \alpha\sum_{t=0}^{T-1}X_t+(T-1)c_b\,(\delta(\theta),\delta(w)).\label{eq:sum-cond-drifts}
\end{align}
where the last inequality follows from adjusting the indices and knowing that $X_t\in[0,1]$. Moreover, observe that
\begin{align*}
\sum_{t=0}^{T-1}X_t&=\sum_{t=0}^{T-1}D_t + \sum_{t=0}^{T-1}\bE[X_t\mid\mathcal F_{t-1}]\\
&=S_{T}+\sum_{t=0}^{T-1}\bE[X_t\mid\mathcal F_{t-1}]
\end{align*}
together with the fact that $X_0=0$ and upper-bounding the last summation by \ref{eq:sum-cond-drifts}
\begin{align}
\label{eq:sumX-drift}
&\sum_{t=0}^{T-1}X_t
\le\; \alpha\sum_{t=0}^{T-1}X_t+(T-1)c_b\,(\delta(\theta),\delta(w)) \;+\; S_{T}\\
&\Rightarrow (1-\alpha)\sum_{t=0}^{T-1}X_t
\le\;(T-1)c_b\,(\delta(\theta),\delta(w)) \;+\; S_{T}
\end{align}
Now, for $n=1,\dots, T$ define \(
V_n \;:=\;\sum_{t=0}^{n-1}\bE\,[D_t^2\mid\mathcal F_{t-1}]
\). Since $X_t\in[0,1]$, we have $\bE[D_t^2\mid\mathcal F_{t-1}] = \text{Var}(X_t\mid\mathcal F_{t-1}) \le \bE[X_t^2\mid\mathcal F_{t-1}]\le \bE[X_t\mid\mathcal F_{t-1}].$

To apply Freedman's inequality with a deterministic bound on the predictable quadratic variation, define
\[
Y_n:=\sum_{t=0}^{n-1}X_t,\qquad n=1,\dots,T.
\]
Then, for every $n\le T$, using \ref{eq:belief-freedman-drift-bound},
\begin{align}
V_n
&\le \sum_{t=0}^{n-1}\bE\left[X_t\,\middle|\,\mathcal F_{t-1}\right]\notag\\
&\le \alpha\sum_{t=0}^{n-2}X_t+(n-1)c_b\,(\delta(\theta),\delta(w))\notag\\
&\le \alpha Y_{n-1}+T c_b\,(\delta(\theta),\delta(w)).
\label{eq:V-tight-bound}
\end{align}
Let
\[
R:=\frac{2T c_b\,(\delta(\theta),\delta(w))}{1-\alpha}
+\frac{2\sqrt{2T c_b\,(\delta(\theta),\delta(w))\log(1/\delta_b)}}{1-\alpha}
+\frac{\log(1/\delta_b)}{1-\alpha}\left(\frac{4}{3}+\frac{2\alpha}{1-\alpha}\right).
\]
If $R\ge T$, then $Y_T\le T\le R$ deterministically. Otherwise, define the first-crossing time
\[
\tau:=\inf\{n\in\{1,\dots,T\}:Y_n\ge R\},
\]
with $\tau=T+1$ if the set is empty. On the event $\{\tau\le T\}$, we have $Y_{\tau-1}<R$, and hence by \ref{eq:V-tight-bound},
\(
V_\tau\le \alpha R+T c_b\,(\delta(\theta),\delta(w)).
\) Moreover, using $Y_\tau=S_\tau+\sum_{t=0}^{\tau-1}\bE[X_t\mid\mathcal F_{t-1}]$ and the same drift bound,
\(
S_\tau>(1-\alpha)R-T c_b\,(\delta(\theta),\delta(w)). \)
Therefore, if $Y_T\ge R$, then for some $n\le T$,
\[
S_n\ge (1-\alpha)R-T c_b\,(\delta(\theta),\delta(w)),\qquad
V_n\le \alpha R+T c_b\,(\delta(\theta),\delta(w)).
\]
Since $X_t\in[0,1]$ yields $|D_t|\le1$, the maximal form of Freedman's inequality \cite{Freedman1975} gives, for deterministic $s,v>0$,
\begin{equation}\label{eq:freedman-applied}
\bP\left(\exists n\le T:\ S_n\ge s,\ V_n\le v\right)
\le \exp\left(-\frac{s^2}{2(v+s/3)}\right).
\end{equation}
Using $\sqrt{x+y}\le\sqrt{x}+\sqrt{y}$ and the arithmetic mean-geometric mean inequality,
\begin{align*}
\sqrt{2\log(1/\delta_b)\left(\alpha R+T c_b\,(\delta(\theta),\delta(w))\right)}
&\le \sqrt{2\alpha R\log(1/\delta_b)}
+\sqrt{2T c_b\,(\delta(\theta),\delta(w))\log(1/\delta_b)}\\
&\le \frac{(1-\alpha)R}{2}
+\frac{\alpha\log(1/\delta_b)}{1-\alpha}
+\sqrt{2T c_b\,(\delta(\theta),\delta(w))\log(1/\delta_b)}.
\end{align*}
By the definition of $R$, this gives
\[
\sqrt{2\log(1/\delta_b)\left(\alpha R+T c_b\,(\delta(\theta),\delta(w))\right)}
+\frac{2}{3}\log(1/\delta_b)
\le (1-\alpha)R-T c_b\,(\delta(\theta),\delta(w)).
\]
Thus, applying \ref{eq:freedman-applied} with
\(
s=(1-\alpha)R-T c_b\,(\delta(\theta),\delta(w)),\text{ and }
v=\alpha R+T c_b\,(\delta(\theta),\delta(w)),
\)
yields
\(
\bP(Y_T\ge R)\le\delta_b.
\)
Hence, with probability at least $1-\delta_b$, \(
\sum_{t=0}^{T-1}X_t=Y_T\le R,
\) and dividing by $T$ gives
\[
\frac{1}{T}\sum_{t=0}^{T-1}X_t \le \frac{2c_b\,(\delta(\theta),\delta(w))}{1-\alpha} + \frac{2}{1-\alpha}\sqrt{\frac{2 c_b\,(\delta(\theta),\delta(w)) \log(1/\delta_b)}{T}} + \frac{\log(1/\delta_b)}{T(1-\alpha)}\left(\frac{4}{3} + \frac{2\alpha}{1-\alpha}\right).
\]

The neural-softmax case follows by the same argument. One replaces
\(X_t=\|b_t^\Theta-b_t^{\Theta^\star}\|_{\mathrm{TV}}\) by
\(X_t^{\mathrm{NN}}=\|b_t^W-b_t^{W^\star}\|_{\mathrm{TV}}\), and uses the
one-step recursion from Corollary~\ref{cor:nn-param-belief-lip},
\[
\bE[X_t^{\mathrm{NN}}\mid\mathcal F_{t-1}]\le\alpha_{\mathrm{NN}}X_{t-1}^{\mathrm{NN}}
+c_b^{\mathrm{NN}}(W,W^\star;\delta_{\mathrm{NN}}).
\]
Repeating the preceding Freedman argument with
\((\alpha,c_b)\) replaced by
\((\alpha_{\mathrm{NN}},c_b^{\mathrm{NN}})\) gives the identical
high-probability time-average bound for the neural-softmax model.
\end{proof}

%====================================================================================================================================================================

\section{POMDP Model Learning via Covariance-Preconditioned Gradient EM}
In this appendix we extend the population-to-sample first-order EM framework of \cite{balakrishnan2017statistical} to latent-state estimation from POMDP trajectories, using the trajectory-level complete-data covariance to capture the intrinsic geometry and precondition the gradient update accordingly \cite[Section 10.1.3]{HMMCappe}. Observed-information identifiability yields the local gradient-stability behavior underlying population contraction, while covariance-relative filtering and smoothing stability control the pilot-estimated preconditioner used in the finite-sample analysis.
\label{sec:covariance-gradient-em}
\medskip\noindent
Fix an estimation horizon $K\in\bN$ and adopt the trajectory convention of Eq.~\ref{eq:traj-joint} with $T=K$: \(
Z=(s_{0:K}),\, Y=(\hat s_{1:K},a_{0:K-1}).
\)
Throughout this section, the action sequence is assumed to be fixed, and all likelihoods and expectations are conditional on $a_{0:K-1}$. The complete and observed likelihoods are therefore
\begin{align}
\log P_{\Theta}(Z,Y)
&=\log \nu_0(s_0)+\sum_{k=0}^{K-1}\log P_\theta(s_{k+1}\mid s_k,a_k)+\sum_{k=1}^{K}\log \Phi_w(\hat s_k\mid s_k),\label{eq:EM-elemenst-full-likelihood}\\
\log P_{\Theta}(Y)
&=\log\sum_{s_{0:K}\in\bS^{K+1}}P_{\Theta}(Z,Y).\label{eq:EM-marginal-likelihood}
\end{align}
The one-slice and two slice smoothing marginals are defined as
\begin{equation}\label{eq:gamma-one-slice}
\gamma_k^\Theta(s):=P_\Theta(s_k=s\mid Y),\,k=0,\dots,K ,
\end{equation}
\begin{equation}\label{eq:xi-two-slice}
\xi_k^\Theta(s,s'):=P_\Theta(s_k=s,s_{k+1}=s'\mid Y), k=0,\dots,K-1,
\end{equation}

respectively. These quantities can be evaluated by the standard forward--backward recursion associated with the Bayesian filter in Eq.~\ref{eq:belief-formula}; in particular, the normalized forward distribution coincides with the filtering belief $b_k^\Theta$, with $b_0^\Theta=\nu_0$. For a fixed observed trajectory $Y$, let $\Theta_i=(\theta_i,w_i)$, $i\in\{1,2\}$, and let the smoothing marginals be evaluated under $\Theta_2$. Considering the shorthands $\gamma_k^{(2)}:=\gamma_k^{\Theta_2}$ and $\xi_k^{(2)}:=\xi_k^{\Theta_2}$, the \emph{trajectory-level} EM auxiliary function is
\begin{align}
Q_Y(\Theta_1,\Theta_2)
&:=\bE_{Z\mid Y,\Theta_2}\!\left[\log P_{\Theta_1}(Y,Z)\right]\notag\\
&=\sum_{k=0}^{K-1}\sum_{s,s'\in\bS}\xi_k^{(2)}(s,s')\left[\theta_1^\top\phi_p(s,a_k,s')-\log\sum_{y\in\bS}\exp\{\theta_1^\top\phi_p(s,a_k,y)\}\right]\label{eq:Q-term-2}\\
&\quad+\sum_{k=1}^{K}\sum_{s\in\bS}\gamma_k^{(2)}(s)\left[w_1^\top\phi_\Phi(s,\hat s_k)-\log\sum_{y\in\hat\bS}\exp\{w_1^\top\phi_\Phi(s,y)\}\right].\label{eq:Q-term-3}
\end{align}
Define feature map expectations as
\[
\bar\phi_p(s,a;\theta):=\sum_{y\in\bS}P_\theta(y\mid s,a)\phi_p(s,a,y),\qquad
\bar\phi_\Phi(s;w):=\sum_{y\in\hat\bS}\Phi_w(y\mid s)\phi_\Phi(s,y),
\]
and the centered feature maps
\[
\widetilde\phi_p(s,a,s';\theta):=\phi_p(s,a,s')-\bar\phi_p(s,a;\theta),\qquad
\widetilde\phi_\Phi(s,\hat s;w):=\phi_\Phi(s,\hat s)-\bar\phi_\Phi(s;w).
\]
Using $\sum_{s'\in\bS}\xi_k^{(2)}(s,s')=\gamma_k^{(2)}(s)$, the gradients of the two parameter blocks are
\begin{align}
\nabla_{\theta_1}Q_Y(\Theta_1,\Theta_2)
&=\sum_{k=0}^{K-1}\sum_{s,s'\in\bS}\xi_k^{(2)}(s,s')\widetilde\phi_p(s,a_k,s';\theta_1),\label{eq:Em-grad-theta}\\
\nabla_{w_1}Q_Y(\Theta_1,\Theta_2)
&=\sum_{k=1}^{K}\sum_{s\in\bS}\gamma_k^{(2)}(s)\widetilde\phi_\Phi(s,\hat s_k;w_1).\label{eq:Em-grad-w}
\end{align}

Moreover, define the \emph{population} Q-function and its oracle function ($q(\cdot)$) and their gradients with respect to their first argument by
\begin{align*}
&Q(\Theta_1,\Theta_2):=\bE_{Y\sim P_{\Theta^\star}}\!\left[Q_Y(\Theta_1,\Theta_2)\right]
\qquad \qquad \qquad \quad q(\Theta_1):=Q(\Theta_1,\Theta^\star)\\
&\nabla_1Q(\Theta_1,\Theta_2)=\bE_{Y\sim P_{\Theta^\star}}\!\left[\nabla_{\Theta_1}Q_Y(\Theta_1,\Theta_2)\right]\qquad\qquad\,\nabla q(\Theta_1)=\bE_{Y\sim P_{\Theta^\star}}\!\left[\nabla_{\Theta_1}Q_Y(\Theta_1,\Theta^\star)\right]
\end{align*}
where the expectation is conditional on the fixed action sequence.

For $(s,a)\in\bS\times\bA$, define the transition and observation feature covariance matrices by
\begin{align}
\Sigma_p(s,a;\theta)
&:=\sum_{s'\in\bS}P_\theta(s'\mid s,a)\widetilde\phi_p(s,a,s';\theta)\widetilde\phi_p(s,a,s';\theta)^\top,\label{eq:transition-feature-covariance}\\
\Sigma_\Phi(s;w)
&:=\sum_{\hat s\in\hat\bS}\Phi_w(\hat s\mid s)\widetilde\phi_\Phi(s,\hat s;w)\widetilde\phi_\Phi(s,\hat s;w)^\top.\label{eq:observation-feature-covariance}
\end{align}
The population complete-data covariance matrix at the true POMDP parameter $\Theta^\star$ is
\begin{equation}\label{eq:complete-data-covariance}
\Sigma_\star:=\begin{pmatrix}\Sigma_{\theta,\star}&0\\0&\Sigma_{w,\star}\end{pmatrix},
\end{equation}
with
\begin{align}
\Sigma_{\theta,\star}
&:=\bE_{Y\sim P_{\Theta^\star}}\!\left[\sum_{k=0}^{K-1}\sum_{s\in\bS}\gamma_k^{\Theta^\star}(s)\Sigma_p(s,a_k;\theta^\star)\right],\label{eq:complete-data-covariance-theta}\\
\Sigma_{w,\star}
&:=\bE_{Y\sim P_{\Theta^\star}}\!\left[\sum_{k=1}^{K}\sum_{s\in\bS}\gamma_k^{\Theta^\star}(s)\Sigma_\Phi(s;w^\star)\right].\label{eq:complete-data-covariance-w}
\end{align}

\begin{lemma}[Oracle stationarity and curvature]\label{lem:complete-data-covariance-curvature}
The oracle auxiliary function satisfies
\begin{equation}\label{eq:oracle-stationarity-curvature}
\nabla q(\Theta^\star)=0,\qquad -\nabla^2q(\Theta^\star)=\Sigma_\star.
\end{equation}
\end{lemma}

\begin{proof}
By the tower property,
\(
q(\Theta)=\bE_{Y\sim P_{\Theta^\star}}\!\left[\bE_{Z\mid Y,\Theta^\star}\!\left[\log P_\Theta(Y,Z)\right]\right]
=\bE_{(Y,Z)\sim P_{\Theta^\star}}\!\left[\log P_\Theta(Y,Z)\right].
\)
At $\Theta=\Theta^\star$, the centered transition and observation feature maps have zero conditional expectation under the true model. Summing over $k$ and taking expectation gives $\nabla q(\Theta^\star)=0$. For the log-linear transition model in Assumption~\ref{as:Model-LogLinear}, \(
\nabla_\theta\bar\phi_p(s,a;\theta)=\Sigma_p(s,a;\theta),\text{ and }
\nabla_\theta\widetilde\phi_p(s,a,s';\theta)=-\Sigma_p(s,a;\theta).
\) Hence, using $\sum_{s'\in\bS}\xi_k^{(2)}(s,s')=\gamma_k^{(2)}(s)$, we have
\[
-\nabla_{\theta_1\theta_1}^2Q_Y(\Theta_1,\Theta_2)
=\sum_{k=0}^{K-1}\sum_{s\in\bS}\gamma_k^{(2)}(s)\Sigma_p(s,a_k;\theta_1).
\]
The same argument applies to the $w$-block. The mixed Hessian blocks vanish as the complete-data log-likelihood separates into transition terms depending on $\theta$ and observation terms depending on $w$. Evaluating at $\Theta_1=\Theta_2=\Theta^\star$, taking expectation over $Y\sim P_{\Theta^\star}$, and using the definition of $\Sigma_\star$ proves \eqref{eq:oracle-stationarity-curvature}.
\end{proof}

Throughout this appendix, we assume Assumption~\ref{as:complete-data-covariance-identifiability} holds.

For the feasible parameter set, consider
$\mathbb B_2(\Theta^\star,r_\Theta)$ for some $r_\Theta>0$,
as in Assumption~\ref{as:realizability-bounded-error}. and let
\[
\mathbb B_{\Sigma_\star}(\Theta^\star,r):=
\left\{\Theta\in\mathbb B_2(\Theta^\star,r_\Theta):
\|\Theta-\Theta^\star\|_{\Sigma_\star}\le r\right\}.
\]

\begin{lemma}[Local covariance curvature]\label{lem:local-covariance-curvature}
For every $\varepsilon\in(0,1)$, there exists $r_\varepsilon>0$ such that, on $\mathbb B_{\Sigma_\star}(\Theta^\star,r_\varepsilon)$, the oracle function $q$ is $(1-\varepsilon)$-strongly concave and $(1+\varepsilon)$-smooth with respect to the norm $\|\cdot\|_{\Sigma_\star}$, i.e., for all $\Theta\in\mathbb B_{\Sigma_\star}(\Theta^\star,r_\varepsilon)$
\begin{equation}\label{eq:local-covariance-curvature}
(1-\varepsilon)\Sigma_\star\preceq-\nabla^2q(\Theta)\preceq (1+\varepsilon)\Sigma_\star
\end{equation}
% and
% \begin{align}
% \left\langle\nabla q(\Theta)-\nabla q(\Theta'),\Theta'-\Theta\right\rangle
% &\ge(1-\varepsilon)\|\Theta-\Theta'\|_{\Sigma_\star}^2,\label{eq:covariance-strong-concavity}\\
% \|\nabla q(\Theta)-\nabla q(\Theta')\|_{\Sigma_\star^{-1}}
% &\le (1+\varepsilon)\|\Theta-\Theta'\|_{\Sigma_\star}.\label{eq:covariance-gradient-smoothness}
% \end{align}
\end{lemma}

\begin{proof}
The model in Assumption \ref{as:Model-LogLinear} implies that $q$ is twice continuously differentiable. Define the preconditioned negative Hessian 
\[
H(\Theta):=\Sigma_\star^{-1/2}\bigl(-\nabla^2q(\Theta)\bigr)\Sigma_\star^{-1/2},
\]
then, by Lemma~\ref{lem:complete-data-covariance-curvature}, $H(\Theta^\star)=I$. Hence, for every $\varepsilon\in(0,1)$, continuity of $H$ implies that there exists $r_\varepsilon>0$ such that
\[
\|H(\Theta)-I\|_{\mathrm{op}}\le\varepsilon,\qquad
\forall\,\Theta\in\mathbb B_{\Sigma_\star}(\Theta^\star,r_\varepsilon),
\]
that implies \ref{eq:local-covariance-curvature}. 

% By convexity of $\mathbb B_{\Sigma_\star}(\Theta^\star,r_\varepsilon)$ we have
% \[
% \nabla q(\Theta)-\nabla q(\Theta')
% =\int_0^1\nabla^2q\bigl(\Theta'+u(\Theta-\Theta')\bigr)(\Theta-\Theta')\,du.
% \]
% The lower bound in \ref{eq:local-covariance-curvature} gives \ref{eq:covariance-strong-concavity}, and the upper bound gives \ref{eq:covariance-gradient-smoothness}.
\end{proof}

\begin{definition}[Local covariance domination]\label{def:local-covariance-domination}
For a fixed radius $r>0$, define
\[
\chi_p:=\sup_{\Theta\in\mathbb B_{\Sigma_\star}(\Theta^\star,r)}
\max_{s\in\bS,\,a\in\bA}
\lambda_{\max}\!\left(\Sigma_{\theta,\star}^{-1/2}\Sigma_p(s,a;\theta)\Sigma_{\theta,\star}^{-1/2}\right),
\]
and
\[
\chi_\Phi:=\sup_{\Theta\in\mathbb B_{\Sigma_\star}(\Theta^\star,r)}
\max_{s\in\bS}
\lambda_{\max}\!\left(\Sigma_{w,\star}^{-1/2}\Sigma_\Phi(s;w)\Sigma_{w,\star}^{-1/2}\right).
\]

\end{definition}

It is straightforward to verify that $\chi_p,\chi_\Phi<\infty$ and, uniformly over $\Theta\in\mathbb B_{\Sigma_\star}(\Theta^\star,r)$, we have
\(
\Sigma_p(s,a;\theta)\preceq\chi_p\Sigma_{\theta,\star},\text{ and }
\Sigma_\Phi(s;w)\preceq\chi_\Phi\Sigma_{w,\star},
\)
for all $(s,a)\in\bS\times\bA$ and $s\in\bS$, respectively.

\medskip
\medskip

% \begin{remark}[Interpretation of the covariance-domination coefficients]
% The quantities $\chi_p$ and $\chi_\Phi$ measure the worst-case local leverage of the transition and observation score covariances relative to the aggregate complete-data information. Small values correspond to well-covered feature directions, whereas large values indicate state-action or observation directions that are weakly represented in $\Sigma_\star$.
% \end{remark}

\begin{lemma}[Covariance-relative perturbation of the log-linear kernels]\label{lem:covariance-relative-model-lip}
For every $\Theta=(\theta,w)\in\mathbb B_{\Sigma_\star}(\Theta^\star,r)$,
\begin{align}
\sup_{s\in\bS,\,a\in\bA}
\left\|P_\theta(\cdot\mid s,a)-P_{\theta^\star}(\cdot\mid s,a)\right\|_1
&\le\sqrt{\chi_p}\,\|\theta-\theta^\star\|_{\Sigma_{\theta,\star}},\label{eq:transition-covariance-lip}\\
\sup_{s\in\bS}
\left\|\Phi_w(\cdot\mid s)-\Phi_{w^\star}(\cdot\mid s)\right\|_1
&\le\sqrt{\chi_\Phi}\,\|w-w^\star\|_{\Sigma_{w,\star}}.\label{eq:observation-covariance-lip}
\end{align}
\end{lemma}

\begin{proof}
The argument follows the line-segment differentiation used in the proof of Lemma~\ref{lem:model-lip}, with the Euclidean feature bound replaced by covariance-relative control. Let $\Theta_u:=\Theta^\star+u(\Theta-\Theta^\star)=(\theta_u,w_u)$ for $u\in[0,1]$. By convexity of $\mathbb B_{\Sigma_\star}(\Theta^\star,r)$, the segment lies in this ball. For fixed $(s,a)\in\bS\times\bA$,
\[
\frac{d}{du}P_{\theta_u}(s'\mid s,a)
=P_{\theta_u}(s'\mid s,a)(\theta-\theta^\star)^\top
\widetilde\phi_p(s,a,s';\theta_u).
\]
Integrating along the segment and applying Cauchy--Schwarz under $P_{\theta_u}(\cdot\mid s,a)$ gives
\begin{align*}
\left\|P_\theta(\cdot\mid s,a)-P_{\theta^\star}(\cdot\mid s,a)\right\|_1
&\le\int_0^1\left((\theta-\theta^\star)^\top
\Sigma_p(s,a;\theta_u)(\theta-\theta^\star)\right)^{1/2}du\\
&\le\sqrt{\chi_p}\,\|\theta-\theta^\star\|_{\Sigma_{\theta,\star}},
\end{align*}
where the last inequality follows from Definition~\ref{def:local-covariance-domination}. Taking the supremum over $(s,a)$ proves \eqref{eq:transition-covariance-lip}. Repeating the same argument for
\[
\frac{d}{du}\Phi_{w_u}(\hat s\mid s)
=\Phi_{w_u}(\hat s\mid s)(w-w^\star)^\top
\widetilde\phi_\Phi(s,\hat s;w_u),
\]
proves \eqref{eq:observation-covariance-lip}.
\end{proof}

\begin{definition}[Covariance-weighted gradient stability]\label{def:covariance-gradient-stability}
For $r>0$, define the local covariance-weighted gradient-stability coefficient by
\[
L_Q^\Sigma(r):=\sup_{\substack{\Theta\in\mathbb B_{\Sigma_\star}(\Theta^\star,r)\\\Theta\neq\Theta^\star}}
\frac{\|\nabla q(\Theta)-\nabla_1Q(\Theta,\Theta)\|_{\Sigma_\star^{-1}}}
{\|\Theta-\Theta^\star\|_{\Sigma_\star}}.
\]
Thus,
\begin{equation}\label{eq:q-grad-smoothness-cov} 
\left\|\nabla q(\Theta)-\nabla_1Q(\Theta,\Theta)\right\|_{\Sigma_\star^{-1}}
\le L_Q^\Sigma(r)\|\Theta-\Theta^\star\|_{\Sigma_\star},\qquad
\forall\,\Theta\in\mathbb B_{\Sigma_\star}(\Theta^\star,r).
\end{equation}
\end{definition}

\begin{lemma}[Local gradient stability of the population EM auxiliary function]\label{lem-3condition-1}
Under Assumptions~\ref{as:complete-data-covariance-identifiability} and~\ref{as:observed-information-identifiability},
\[
\limsup_{r\downarrow0}L_Q^\Sigma(r)\le1-\nu_{\mathrm{obs}}<1.
\]
Therefore, for every $\varepsilon\in(0,\nu_{\mathrm{obs}})$,
\(
L_Q^\Sigma(r)<(1-\varepsilon)
\) for all sufficiently small $r>0$.
\end{lemma}

\begin{proof}
Define
\(
G(\Theta):=\nabla q(\Theta)-\nabla_1Q(\Theta,\Theta).
\) Fisher's identity gives
\(
\nabla_1Q(\Theta,\Theta)=\nabla\ell_{\mathrm{obs}}(\Theta),
\)
so $G(\Theta^\star)=0$. By Lemma~\ref{lem:complete-data-covariance-curvature},
\[
DG(\Theta^\star)=\nabla^2q(\Theta^\star)-\nabla^2\ell_{\mathrm{obs}}(\Theta^\star)
=-\Sigma_\star+\mathcal I_{\mathrm{obs}}.
\]
Define
\[
S_c:=\nabla_\Theta\log P_{\Theta^\star}(Y,Z),\qquad
S_o:=\nabla_\Theta\log P_{\Theta^\star}(Y).
\]
Fisher's identity gives $S_o=\bE[S_c\mid Y]$. Since the model has finite support and strictly positive log-linear probabilities, the complete and observed data information identities apply, and the law of total covariance yields \(
\Sigma_\star-\mathcal I_{\mathrm{obs}}
=\bE\!\left[\operatorname{Cov}(S_c\mid Y)\right]\succeq0.
\) Hence \(
\nu_{\mathrm{obs}}I\preceq\Sigma_\star^{-1/2}\mathcal I_{\mathrm{obs}}\Sigma_\star^{-1/2}\preceq I,
\)
and therefore
\[
\left\|\Sigma_\star^{-1/2}DG(\Theta^\star)\Sigma_\star^{-1/2}\right\|_{\mathrm{op}}
=1-\nu_{\mathrm{obs}}.
\]
As $G$ is continuously differentiable, for $\Theta_u:=\Theta^\star+u(\Theta-\Theta^\star)$,
\[
G(\Theta)=\int_0^1DG(\Theta_u)(\Theta-\Theta^\star)\,du,
\]
and continuity of $DG$ yields
\(
\limsup_{r\downarrow0}L_Q^\Sigma(r)\le1-\nu_{\mathrm{obs}},
\) then the claim follows since $\varepsilon<\nu_{\mathrm{obs}}$.
\end{proof}

The following covariance-relative filtering and smoothing bounds are used below to control the pilot complete-data covariance.

\begin{corollary}[Covariance-relative filtering stability]\label{cor:covariance-filter-stability}
Consider the setting of Theorem~\ref{thm:param-belief-lip} and suppose Assumption~\ref{as:complete-data-covariance-identifiability} and the local covariance domination condition in Definition~\ref{def:local-covariance-domination} hold on $\mathbb B_{\Sigma_\star}(\Theta^\star,r)$. Then, for every $\Theta=(\theta,w)\in\mathbb B_{\Sigma_\star}(\Theta^\star,r)$ and every $t=0,\dots,K$,
\begin{align}
\bE_{Y\sim P_{\Theta^\star}}\!\left[\|b_t^\Theta-b_t^{\Theta^\star}\|_1\right]
&\le\frac{1-\alpha^t}{1-\alpha}\left((3-2\kappa_\Phi)\sqrt{\chi_p}\|\theta-\theta^\star\|_{\Sigma_{\theta,\star}}+2\sqrt{\chi_\Phi}\|w-w^\star\|_{\Sigma_{w,\star}}\right)\notag\\
&\le L_{\alpha,t}^\Sigma\|\Theta-\Theta^\star\|_{\Sigma_\star},\label{eq:alpha-exp-lipschitz}
\end{align}
where
\(
L_{\alpha,t}^\Sigma:=\frac{1-\alpha^t}{1-\alpha}\sqrt{(3-2\kappa_\Phi)^2\chi_p+4\chi_\Phi}.
\)
\end{corollary}

\begin{proof}
The one-step prediction--update decomposition established in the proof of Theorem~\ref{thm:param-belief-lip}, before substituting the Euclidean row-wise bounds from Lemma~\ref{lem:model-lip}, gives
\begin{align*}
e_{t+1}
&\le\alpha e_t+(3-2\kappa_\Phi)\sup_{s\in\bS,\,a\in\bA}\|P_\theta(\cdot\mid s,a)-P_{\theta^\star}(\cdot\mid s,a)\|_1+2\sup_{s\in\bS}\|\Phi_w(\cdot\mid s)-\Phi_{w^\star}(\cdot\mid s)\|_1,
\end{align*}
where \(
e_t:=\bE_{Y\sim P_{\Theta^\star}}\!\left[\|b_t^\Theta-b_t^{\Theta^\star}\|_1\right].
\) Replacing the Euclidean kernel perturbation bounds by Lemma~\ref{lem:covariance-relative-model-lip} yields
\[
e_{t+1}\le\alpha e_t+(3-2\kappa_\Phi)\sqrt{\chi_p}\|\theta-\theta^\star\|_{\Sigma_{\theta,\star}}+2\sqrt{\chi_\Phi}\|w-w^\star\|_{\Sigma_{w,\star}}.
\]
Since $b_0^\Theta=b_0^{\Theta^\star}=\nu_0$, we have $e_0=0$, and unrolling the recursion proves the claim.

\end{proof}
For simplicity, we set \(L_\alpha^\Sigma:=\max_{0\le t\le K} L^\Sigma_{\alpha,t}=L_{\alpha, K}^\Sigma\).

\medskip
\begin{definition}[Backward smoothing kernel]\label{def:backward-smoothing-kernel}
For $k=0,\dots,K-1$, define
\[
p_{k+1\mid k}^\Theta(s'):=\sum_{s\in\bS}b_k^\Theta(s)P_\theta(s'\mid s,a_k),\qquad s'\in\bS.
\]
The backward smoothing kernel is
\begin{align}\label{eq:backward-smoothing-kernel}
R_{k,\Theta}(s\mid s')
&:=\frac{b_k^\Theta(s)P_\theta(s'\mid s,a_k)}{p_{k+1\mid k}^\Theta(s')}\notag\\
&=P_\Theta\!\left(s_k=s\mid s_{k+1}=s',\hat s_{1:k},a_{0:k}\right),
\end{align}
for $(s,s')\in\bS\times\bS$, where $\hat s_{1:0}$ denotes the empty observation history. By the Markov property and Bayes' rule, \ref{eq:gamma-one-slice} and \ref{eq:gamma-backward-recursion} yield
\begin{align}
\gamma_{K}^\Theta(s)&=b_{K}^\Theta(s),\label{eq:gamma-terminal-belief}\\
\gamma_k^\Theta(s)&=\sum_{s'\in\bS}R_{k,\Theta}(s\mid s')\gamma_{k+1}^\Theta(s'),\qquad k=K-1,\dots,0,\label{eq:gamma-backward-recursion}
% \xi_k^\Theta(s,s')&=R_{k,\Theta}(s\mid s')\gamma_{k+1}^\Theta(s'),\qquad k=0,\dots,K-1.\label{eq:xi-backward-representation}
\end{align}
\end{definition}

\begin{lemma}[Expected stability of the backward smoothing kernel]\label{lem:backward-kernel-stability}
Under the conditions of Corollary~\ref{cor:covariance-filter-stability}, for every $\Theta\in\mathbb B_{\Sigma_\star}(\Theta^\star,r)$ and every $k=0,\dots,K-1$,
\begin{align}
&\bE_{Y\sim P_{\Theta^\star}}\!\left[\sum_{s'\in\bS}\gamma_{k+1}^{\Theta^\star}(s')\left\|R_{k,\Theta}(\cdot\mid s')-R_{k,\Theta^\star}(\cdot\mid s')\right\|_1\right]\le L_R^\Sigma\|\Theta-\Theta^\star\|_{\Sigma_\star},\label{eq:backward-kernel-stability}
\end{align}
where
\[
L_R^\Sigma:=2\max_{0\le j\le K-1}\sqrt{\left(\frac{(3-2\kappa_\Phi)\sqrt{\chi_p}(1-\alpha^j)}{1-\alpha}+\sqrt{\chi_p}\right)^2+\left(\frac{2\sqrt{\chi_\Phi}(1-\alpha^j)}{1-\alpha}\right)^2}.
\]
\end{lemma}
\begin{proof}
Let \(\mathcal F_k:=\sigma(\hat s_{1:k},a_{0:k})\), \(\mathcal G:=\sigma(Y)=\sigma(\hat s_{1:K},a_{0:K-1})\),  and define
\(
\widetilde R_{k,\Theta}(s,s'):=b_k^\Theta(s)P_\theta(s'\mid s,a_k).
\)
For fixed $s'\in\bS$, the difference $\|R_{k,\Theta}(\cdot\mid s')-R_{k,\Theta^\star}(\cdot\mid s')\|_1$ is $\mathcal F_k$-measurable. Moreover using tower property
\begin{align}
\bE_{\Theta^\star}\!\left[\gamma_{k+1}^{\Theta^\star}(s')\mid\mathcal F_k\right]
&=\bE_{\Theta^\star}\!\left[
\bE_{\Theta^\star}\!\left[\mathbbm 1_{\{s_{k+1}=s'\}}\mid\mathcal G\right]
\middle|\mathcal F_k\right]\notag\\
&=\bE_{\Theta^\star}\!\left[\mathbbm 1_{\{s_{k+1}=s'\}}\mid\mathcal F_k\right]\notag\\
&=P_{\Theta^\star}\!\left(s_{k+1}=s'\mid\hat s_{1:k},a_{0:k}\right)\notag\\
&=p_{k+1\mid k}^{\Theta^\star}(s').
\label{eq:smoothing-to-predictive-tower}
\end{align}
Then,

\begin{align}
&\bE_{\Theta^\star}\!\left[\sum_{s'\in\bS}\gamma_{k+1}^{\Theta^\star}(s')\left\|R_{k,\Theta}(\cdot\mid s')-R_{k,\Theta^\star}(\cdot\mid s')\right\|_1\right]\notag\\
&\qquad\quad=\bE_{\Theta^\star}\!\left[\sum_{s'\in\bS}p_{k+1\mid k}^{\Theta^\star}(s')\left\|R_{k,\Theta}(\cdot\mid s')-R_{k,\Theta^\star}(\cdot\mid s')\right\|_1\right]\notag\\
&\qquad\quad=\bE_{\Theta^\star}\!\left[\sum_{s,s'}\,p_{k+1\mid k}^{\Theta^\star}(s')\left|R_{k,\Theta}(s\mid s')-R_{k,\Theta^\star}(s\mid s')\right|\right]\notag\\
&\qquad\quad\le\bE_{\Theta^\star}\!\Bigg[\sum_{s,s'}\,\left|p_{k+1\mid k}^{\Theta^\star}(s')-p_{k+1\mid k}^\Theta(s')\right|R_{k,\Theta}(s\mid s')\notag\\
&\qquad\qquad\qquad+\left|\widetilde R_{k,\Theta}(s,s')-\widetilde R_{k,\Theta^\star}(s,s')\right|\Bigg]\notag\\
&\qquad\quad =\bE_{\Theta^\star}\!\left[ \|p_{k+1\mid k}^{\Theta}-p_{k+1\mid k}^{\Theta^\star}\|_1+\|\widetilde R_{k,\Theta}-\widetilde R_{k,\Theta^\star}\|_1\right]\notag\\
&\qquad\quad \le\bE_{\Theta^\star}\!\left[2\|\widetilde R_{k,\Theta}-\widetilde R_{k,\Theta^\star}\|_1\right]\notag\\
&\qquad\quad \le2\bE_{\Theta^\star}\!\Bigg[\sum_{s,s'\in\bS}|b_k^\Theta(s)-b_k^{\Theta^\star}(s)|P_\theta(s'\mid s,a_k)\notag\\
&\qquad\qquad\qquad+\sum_{s,s'\in\bS}b_k^{\Theta^\star}(s)|P_\theta(s'\mid s,a_k)-P_{\theta^\star}(s'\mid s,a_k)|\Bigg]\notag\\
&\qquad\quad\le2\bE_{\Theta^\star}\!\left[\|b_k^\Theta-b_k^{\Theta^\star}\|_1+\sqrt{\chi_p}\|\theta-\theta^\star\|_{\Sigma_{\theta,\star}}\right]
\end{align}

where the last inequality follows from Lemma~\ref{lem:covariance-relative-model-lip}. Corollary~\ref{cor:covariance-filter-stability} then implies
\begin{align*}
&2\,\bE\!\left[\|b_k^\Theta-b_k^{\Theta^\star}\|_1\right]+2\sqrt{\chi_p}\|\theta-\theta^\star\|_{\Sigma_{\theta,\star}}\\
&\qquad\le2\left(\frac{(3-2\kappa_\Phi)\sqrt{\chi_p}(1-\alpha^k)}{1-\alpha}+\sqrt{\chi_p}\right)\|\theta-\theta^\star\|_{\Sigma_{\theta,\star}}+\frac{4\sqrt{\chi_\Phi}(1-\alpha^k)}{1-\alpha}\|w-w^\star\|_{\Sigma_{w,\star}}\\
&\qquad\le L_R^\Sigma\|\Theta-\Theta^\star\|_{\Sigma_\star},
\end{align*}

\end{proof}

\begin{lemma}[Dobrushin bound for the backward smoothing kernel]\label{lem:backward-smoothing-contraction}
Suppose Assumption~\ref{as:Model-LogLinear} holds, and let $\Theta\in\mathbb B_{\Sigma_\star}(\Theta^\star,r)$. Then, for every $k=0,\dots,K-1$ and every realization of the observation history,

\begin{equation}\label{eq:backward-kernel-dobrushin}
\sup_{s',v\in\bS}\frac12\left\|R_{k,\Theta}(\cdot\mid s')-R_{k,\Theta}(\cdot\mid v)\right\|_1\le1-\kappa_R,
\end{equation}
\begin{equation}\label{eq:backward-smoothing-contraction}
\|R_{k,\Theta}\nu-R_{k,\Theta}\nu'\|_1\le(1-\kappa_R)\|\nu-\nu'\|_1,\quad  \forall\nu,\nu'\in\Delta_{\bS},
\end{equation}
where
\begin{equation}\label{eq:backward-dobrushin}
\kappa_R:=\frac{1+(|\bS|-1)\exp\!\left(-2B(\|\theta^\star\|_2+r_\Theta)\right)}
{1+(|\bS|-1)\exp\!\left(2B(\|\theta^\star\|_2+r_\Theta)\right)}\in(0,1].
\end{equation}

\end{lemma}

\begin{proof}
Fix $k\in\{0,\dots,K-1\}$. Since $\Theta\in\mathbb B_{\Sigma_\star}(\Theta^\star,r)\subseteq\mathbb B_2(\Theta^\star,r_\Theta)$, the pointwise transition bounds in Eqs.~\ref{eq:kernel-uni-upperbound} and \ref{eq:kernel-uni-lowerbound}, together with $b_k^\Theta\in\Delta_{\bS}$, give
\[
p_{k+1\mid k}^\Theta(s')=\sum_{u\in\bS}b_k^\Theta(u)P_\theta(s'\mid u,a_k)
\le\frac{1}{1+(|\bS|-1)\exp\!\left(-2B(\|\theta^\star\|_2+r_\Theta)\right)}.
\]
Hence, using the lower bound in Eq.~\ref{eq:kernel-uni-lowerbound},
\[
R_{k,\Theta}(s\mid s')=\frac{b_k^\Theta(s)P_\theta(s'\mid s,a_k)}{p_{k+1\mid k}^\Theta(s')}
\ge\kappa_R b_k^\Theta(s).
\]
Therefore, for every $s',v\in\bS$,
\[
\sum_{s\in\bS}\min\!\left\{R_{k,\Theta}(s\mid s'),R_{k,\Theta}(s\mid v)\right\}
\ge\kappa_R\sum_{s\in\bS}b_k^\Theta(s)=\kappa_R.
\]
Using $\frac12\|\nu-\nu'\|_1=1-\sum_s\min\{\nu(s),\nu'(s)\}$ for probability distributions $\nu,\nu'\in\Delta_{\bS}$ proves \eqref{eq:backward-kernel-dobrushin}. The Dobrushin contraction inequality then gives \eqref{eq:backward-smoothing-contraction}. 

\end{proof}

\begin{lemma}[Stability of the one-slice smoothing marginals]\label{lem:gamma-stability}
Assume the conditions of Corollary~\ref{cor:covariance-filter-stability}, and let $(1-\kappa_R)\in[0,1)$ be defined in Lemma~\ref{lem:backward-smoothing-contraction}. Then, for every $\Theta\in\mathbb B_{\Sigma_\star}(\Theta^\star,r)$ and every $k=0,\dots,K$,
\begin{equation}\label{eq:gamma-pointwise-stability}
\bE_{Y\sim P_{\Theta^\star}}\!\left[\|\gamma_k^\Theta-\gamma_k^{\Theta^\star}\|_1\right]
\le\left[(1-\kappa_R)^{\,K-k}L_\alpha^\Sigma+\frac{1-(1-\kappa_R)^{\,K-k}}{\kappa_R}L_R^\Sigma\right]\|\Theta-\Theta^\star\|_{\Sigma_\star}.
\end{equation}
Consequently,
\[
\bE_{Y\sim P_{\Theta^\star}}\!\left[\sum_{k=0}^{K}\|\gamma_k^\Theta-\gamma_k^{\Theta^\star}\|_1\right]
\le L_\gamma^\Sigma\|\Theta-\Theta^\star\|_{\Sigma_\star},
\]
where
\[
L_\gamma^\Sigma:=\frac{1-(1-\kappa_R)^{K+1}}{\kappa_R}L_\alpha^\Sigma
+\frac{L_R^\Sigma}{\kappa_R}\left(K+1-\frac{1-(1-\kappa_R)^{K+1}}{\kappa_R}\right).
\]
\end{lemma}

\begin{proof}
By \ref{eq:gamma-backward-recursion}, for $k=0,\dots,K-1$,
\begin{align*}
\|\gamma_k^\Theta-\gamma_k^{\Theta^\star}\|_1&=
\|R_{k,\Theta}\bigl(\gamma_{k+1}^\Theta-\gamma_{k+1}^{\Theta^\star}\bigr)
+\bigl(R_{k,\Theta}-R_{k,\Theta^\star}\bigr)\gamma_{k+1}^{\Theta^\star}\|_1\\
&\le\left\|R_{k,\Theta}\bigl(\gamma_{k+1}^\Theta-\gamma_{k+1}^{\Theta^\star}\bigr)\right\|_1\\
&\qquad\qquad+\sum_{s'\in\bS}\gamma_{k+1}^{\Theta^\star}(s')
\left\|R_{k,\Theta}(\cdot\mid s')-R_{k,\Theta^\star}(\cdot\mid s')\right\|_1.
\end{align*}
Taking expectation, applying \eqref{eq:backward-smoothing-contraction} pathwise to the first term and Lemma~\ref{lem:backward-kernel-stability} to the second, gives
\begin{equation}\label{eq:gamma-stability-recursion}
\bE\!\left[\|\gamma_k^\Theta-\gamma_k^{\Theta^\star}\|_1\right]
\le(1-\kappa_R)\bE\!\left[\|\gamma_{k+1}^\Theta-\gamma_{k+1}^{\Theta^\star}\|_1\right]
+L_R^\Sigma\|\Theta-\Theta^\star\|_{\Sigma_\star}.
\end{equation}
Moreover, by \ref{eq:gamma-terminal-belief} and Corollary~\ref{cor:covariance-filter-stability}, \(
\bE\!\left[\|\gamma_{K}^\Theta-\gamma_{K}^{\Theta^\star}\|_1\right]
\le L_\alpha^\Sigma\|\Theta-\Theta^\star\|_{\Sigma_\star}.
\)
Unrolling \ref{eq:gamma-stability-recursion} proves \ref{eq:gamma-pointwise-stability}.
\end{proof}

\begin{remark}[Covariance feature-diameter bound]\label{lem:centered-feature-covariance-leverage}
Define
\[
B_{p,\Sigma}:=\max_{\substack{s\in\bS,\,a\in\bA\\s',y\in\bS}}
\|\phi_p(s,a,s')-\phi_p(s,a,y)\|_{\Sigma_{\theta,\star}^{-1}},\qquad
B_{\Phi,\Sigma}:=\max_{\substack{s\in\bS\\\hat s,y\in\hat\bS}}
\|\phi_\Phi(s,\hat s)-\phi_\Phi(s,y)\|_{\Sigma_{w,\star}^{-1}}.
\]
Since we have
\[
\widetilde\phi_p(s,a,s';\theta)
=\sum_{y\in\bS}P_\theta(y\mid s,a)
\bigl(\phi_p(s,a,s')-\phi_p(s,a,y)\bigr),
\]
and analogously for $\widetilde\phi_\Phi$, the triangle inequality gives, uniformly over $\Theta\in\mathbb B_{\Sigma_\star}(\Theta^\star,r)$,
\[
\max_{s,s',a}\|\widetilde\phi_p(s,a,s';\theta)\|_{\Sigma_{\theta,\star}^{-1}}
\le B_{p,\Sigma},\qquad
\max_{s,\hat s}\|\widetilde\phi_\Phi(s,\hat s;w)\|_{\Sigma_{w,\star}^{-1}}
\le B_{\Phi,\Sigma}.
\]
\end{remark}

\begin{proof}[Proof of Proposition~\ref{prop:covariance-gradient-em-convergence}]\label{proof:prop-covgradEM}
Fix $\Theta\in\mathbb B_{\Sigma_\star}(\Theta^\star,r)$ and write $\Delta:=\Theta-\Theta^\star$. By adding and subtracting the oracle gradient,
\begin{align*}
&\left\|\Theta+\eta\Sigma_\star^{-1}\nabla_1Q(\Theta,\Theta)-\Theta^\star\right\|_{\Sigma_\star}\\
&\qquad\le\left\|\Delta+\eta\Sigma_\star^{-1}\nabla q(\Theta)
+\eta\Sigma_\star^{-1}\bigl(\nabla_1Q(\Theta,\Theta)-\nabla q(\Theta)\bigr)\right\|_{\Sigma_\star}\\
&\qquad\le\left\|\Delta+\eta\Sigma_\star^{-1}\nabla q(\Theta)\right\|_{\Sigma_\star}
+\eta\|\nabla_1Q(\Theta,\Theta)-\nabla q(\Theta)\|_{\Sigma_\star^{-1}}.
\end{align*}
By Lemma~\ref{lem:complete-data-covariance-curvature},
\[
\nabla q(\Theta)
=\int_0^1\nabla^2q(\Theta^\star+u\Delta)\Delta\,du.
\]
Define
\[
\overline H_\Theta:=\int_0^1
\Sigma_\star^{-1/2}\bigl(-\nabla^2q(\Theta^\star+u\Delta)\bigr)
\Sigma_\star^{-1/2}\,du,
\]
Then Lemma~\ref{lem:local-covariance-curvature} gives
\(
(1-\varepsilon)I\preceq\overline H_\Theta\preceq(1+\varepsilon)I.
\)
Therefore,
\[
\left\|\Delta+\eta\Sigma_\star^{-1}\nabla q(\Theta)\right\|_{\Sigma_\star}
=\left\|\bigl(I-\eta\overline H_\Theta\bigr)\Sigma_\star^{1/2}\Delta\right\|_2
\le\rho_\eta\|\Delta\|_{\Sigma_\star}.
\]
Combining the preceding inequalities yields
\[
\left\|\Theta+\eta\Sigma_\star^{-1}\nabla_1Q(\Theta,\Theta)-\Theta^\star\right\|_{\Sigma_\star}
\le\bigl(\rho_\eta+\eta L_Q^\Sigma(r)\bigr)\|\Theta-\Theta^\star\|_{\Sigma_\star}.
\]
Since $\rho_\eta+\eta L_Q^\Sigma(r)<1$, the update maps $\mathbb B_{\Sigma_\star}(\Theta^\star,r)$ into itself. Hence, by induction, all iterates remain in this ball and
\[
\|\Theta^{(t+1)}-\Theta^\star\|_{\Sigma_\star}
\le\bigl(\rho_\eta+\eta L_Q^\Sigma(r)\bigr)\|\Theta^{(t)}-\Theta^\star\|_{\Sigma_\star}.
\]
Iteration proves \ref{eq:covariance-gradient-em-contraction}.
\end{proof}

\medskip
\medskip
\begin{remark}
Under Assumption~\ref{as:observed-information-identifiability}, the contraction condition in Proposition \ref{prop:covariance-gradient-em-convergence} is non-vacuous. In particular, if $\eta=1$ and $\varepsilon\in(0,\nu_{\mathrm{obs}})$, then Lemma~\ref{lem-3condition-1} gives $\kappa_{\mathrm{EM}}<1$ for all sufficiently small $r\le r_\varepsilon$.
\end{remark}

\medskip
\medskip
\begin{remark}[Interface with reward learning]\label{rem:em-reward-interface}
Define
\[
c_b^\Sigma(\Theta,\Theta^\star)
:=\sqrt{\chi_\Phi}\|w-w^\star\|_{\Sigma_{w,\star}}
+\frac{3-2\kappa_\Phi}{2}\sqrt{\chi_p}\|\theta-\theta^\star\|_{\Sigma_{\theta,\star}},
\qquad
C_b^\Sigma:=\sqrt{\chi_\Phi+\frac{(3-2\kappa_\Phi)^2}{4}\chi_p}.
\]
By blockwise Cauchy--Schwarz,
\(
c_b^\Sigma(\Theta,\Theta^\star)\le C_b^\Sigma\|\Theta-\Theta^\star\|_{\Sigma_\star}.
\) Hence, under Proposition~\ref{prop:covariance-gradient-em-convergence}, with
$\kappa_{\mathrm{EM}}:=\rho_\eta+\eta L_Q^\Sigma(r)<1$,
\begin{equation}\label{eq:em-belief-interface}
c_b^\Sigma(\Theta^{(t)},\Theta^\star)
\le C_b^\Sigma\kappa_{\mathrm{EM}}^t\|\Theta^{(0)}-\Theta^\star\|_{\Sigma_\star}.
\end{equation}
Thus, the downstream belief and reward-learning guarantees apply with
$c_b(\delta(\theta),\delta(w))$ replaced by
$c_b^\Sigma(\Theta^{(t)},\Theta^\star)$.
\end{remark}

\subsection{Finite-sample covariance-preconditioned gradient EM}
\label{subsec:finite-sample-covariance-gradient-em}

Let $\{Y_0^{(i)}\}_{i=1}^{M_0}$ and $\{Y^{(i)}\}_{i=1}^{M}$ be two independent collections of trajectories generated under $P_{\Theta^\star}$, conditional on the same fixed action sequence $a_{0:K-1}$. Let $\widetilde\Theta=(\widetilde\theta,\widetilde w)\in\mathbb B_{\Sigma_\star}(\Theta^\star,r)$ be fixed independently of both collections. Since the latent state sequence is unobserved, the complete-data curvature associated with an observed trajectory $Y$ is replaced by its conditional expectation under the smoothing law at $\widetilde\Theta$. In particular,
\[
\widehat\Sigma_Y(\widetilde\Theta)
:=\bE_{Z\mid Y,\widetilde\Theta}\!\left[-\nabla_\Theta^2\log P_{\widetilde\Theta}(Y,Z)\right]
=-\left.\nabla_1^2Q_Y(\Theta_1,\widetilde\Theta)\right|_{\Theta_1=\widetilde\Theta}.
\]
Using the Hessian identities above and $\sum_{s'\in\bS}\xi_k^{\widetilde\Theta}(s,s')=\gamma_k^{\widetilde\Theta}(s)$, this matrix has the explicit form
\begin{equation}\label{eq:single-trajectory-complete-data-covariance}
\widehat\Sigma_Y(\widetilde\Theta)
=\begin{pmatrix}
\displaystyle\sum_{k=0}^{K-1}\sum_{s\in\bS}\gamma_k^{\widetilde\Theta}(s)\Sigma_p(s,a_k;\widetilde\theta)&0\\[3mm]
0&\displaystyle\sum_{k=1}^{K}\sum_{s\in\bS}\gamma_k^{\widetilde\Theta}(s)\Sigma_\Phi(s;\widetilde w)
\end{pmatrix}.
\end{equation}
Thus, $\gamma_k^{\widetilde\Theta}(s)=P_{\widetilde\Theta}(s_k=s\mid Y)$ acts as the posterior weight assigned to the unobserved state $s_k$ when evaluating the complete-data curvature. Define the pilot complete-data covariance estimator and its regularized version by
\begin{align}
\widehat\Sigma_{M_0}(\widetilde\Theta)&:=\frac{1}{M_0}\sum_{i=1}^{M_0}\widehat\Sigma_{Y_0^{(i)}}(\widetilde\Theta),\label{eq:pilot-complete-data-covariance}\\
\widehat\Sigma_\lambda&:=\widehat\Sigma_{M_0}(\widetilde\Theta)+\lambda I,\qquad \lambda>0.\label{eq:regularized-pilot-preconditioner}
\end{align}
Let $d_\Theta:=\dim(\mathcal U_\theta)+\dim(\mathcal U_w)$ and $\chi_\Sigma:=\max\{\chi_p,\chi_\Phi\}$.
\begin{lemma}[Lipschitz continuity of the local covariance maps]\label{lem:local-covariance-map-lipschitz}
Under Assumption~\ref{as:complete-data-covariance-identifiability} and Definition~\ref{def:local-covariance-domination}, for every $\Theta=(\theta,w),\Theta'=(\theta',w')\in\mathbb B_{\Sigma_\star}(\Theta^\star,r)$,
\begin{align}
\sup_{s\in\bS,\,a\in\bA}\left\|\Sigma_{\theta,\star}^{-1/2}\bigl(\Sigma_p(s,a;\theta)-\Sigma_p(s,a;\theta')\bigr)\Sigma_{\theta,\star}^{-1/2}\right\|_{\mathrm{op}}
&\le L_{\Sigma,p}\|\theta-\theta'\|_{\Sigma_{\theta,\star}},\label{eq:local-transition-covariance-lipschitz}\\
\sup_{s\in\bS}\left\|\Sigma_{w,\star}^{-1/2}\bigl(\Sigma_\Phi(s;w)-\Sigma_\Phi(s;w')\bigr)\Sigma_{w,\star}^{-1/2}\right\|_{\mathrm{op}}
&\le L_{\Sigma,\Phi}\|w-w'\|_{\Sigma_{w,\star}},\label{eq:local-observation-covariance-lipschitz}
\end{align}
where \(
L_{\Sigma,p}:=\chi_pB_{p,\Sigma} \text{, and }L_{\Sigma,\Phi}:=\chi_\Phi B_{\Phi,\Sigma}.
\)
\end{lemma}

\begin{proof}
Fix $(s,a)\in\bS\times\bA$, let $\theta_u:=\theta'+u(\theta-\theta')$, and $h:=\theta-\theta'$. The directional derivative of the log-linear covariance map is
\[
D\Sigma_p(s,a;\theta_u)[h]=\sum_{s'\in\bS}P_{\theta_u}(s'\mid s,a)\bigl(h^\top\widetilde\phi_p(s,a,s';\theta_u)\bigr)\widetilde\phi_p(s,a,s';\theta_u)\widetilde\phi_p(s,a,s';\theta_u)^\top.
\]
For any $v\in\bR^{d_\theta}$ with $\|v\|_2=1$, Remark~\ref{lem:centered-feature-covariance-leverage}, and Definition~\ref{def:local-covariance-domination} give
\begin{align*}
&\left|v^\top\Sigma_{\theta,\star}^{-1/2}D\Sigma_p(s,a;\theta_u)[h]\Sigma_{\theta,\star}^{-1/2}v\right|\\
&\qquad\le B_{p,\Sigma}\|h\|_{\Sigma_{\theta,\star}}\sum_{s'\in\bS}P_{\theta_u}(s'\mid s,a)\left(v^\top\Sigma_{\theta,\star}^{-1/2}\widetilde\phi_p(s,a,s';\theta_u)\right)^2\\
&\qquad=B_{p,\Sigma}\|h\|_{\Sigma_{\theta,\star}}v^\top\Sigma_{\theta,\star}^{-1/2}\Sigma_p(s,a;\theta_u)\Sigma_{\theta,\star}^{-1/2}v
\le\chi_pB_{p,\Sigma}\|h\|_{\Sigma_{\theta,\star}}.
\end{align*}
Integrating over $u\in[0,1]$ proves \ref{eq:local-transition-covariance-lipschitz}. The observation bound follows similarly from
\[
D\Sigma_\Phi(s;w_u)[h_w]=\sum_{\hat s\in\hat\bS}\Phi_{w_u}(\hat s\mid s)\bigl(h_w^\top\widetilde\phi_\Phi(s,\hat s;w_u)\bigr)\widetilde\phi_\Phi(s,\hat s;w_u)\widetilde\phi_\Phi(s,\hat s;w_u)^\top.
\]
\end{proof}

\begin{lemma}[Bias of the pilot covariance]\label{lem:pilot-covariance-bias}
Let $\widetilde\Theta\in\mathbb B_{\Sigma_\star}(\Theta^\star,r)$. Then
\begin{equation}\label{eq:pilot-covariance-bias}
\left\|\Sigma_\star^{-1/2}\left(\bE_{Y\sim P_{\Theta^\star}}\!\left[\widehat\Sigma_Y(\widetilde\Theta)\right]-\Sigma_\star\right)\Sigma_\star^{-1/2}\right\|_{\mathrm{op}}
\le L_{\mathrm{pilot}}^\Sigma\|\widetilde\Theta-\Theta^\star\|_{\Sigma_\star},
\end{equation}
where
\(
L_{\mathrm{pilot}}^\Sigma:=\max\left\{\chi_pL_\gamma^\Sigma+KL_{\Sigma,p},\ \chi_\Phi L_\gamma^\Sigma+KL_{\Sigma,\Phi}\right\}.
\)
\end{lemma}

\begin{proof}
Consider the transition block. Adding and subtracting the covariance weighted by $\gamma_k^{\Theta^\star}$ at $\widetilde\theta$ yields
\begin{align*}
\bE\!\left[\sum_{k=0}^{K-1}\sum_{s\in\bS}\gamma_k^{\widetilde\Theta}(s)\Sigma_p(s,a_k;\widetilde\theta)\right]-\Sigma_{\theta,\star}
&=\bE\!\left[\sum_{k=0}^{K-1}\sum_{s\in\bS}\bigl(\gamma_k^{\widetilde\Theta}(s)-\gamma_k^{\Theta^\star}(s)\bigr)\Sigma_p(s,a_k;\widetilde\theta)\right]\\
&\qquad\qquad+\bE\!\left[\sum_{k=0}^{K-1}\sum_{s\in\bS}\gamma_k^{\Theta^\star}(s)\bigl(\Sigma_p(s,a_k;\widetilde\theta)-\Sigma_p(s,a_k;\theta^\star)\bigr)\right].
\end{align*}
After conjugation by $\Sigma_{\theta,\star}^{-1/2}$, by Lemma~\ref{lem:gamma-stability} the first term is bounded by
\[
\chi_p\,\bE\!\left[\sum_{k=0}^{K-1}\|\gamma_k^{\widetilde\Theta}-\gamma_k^{\Theta^\star}\|_1\right]
\le\chi_pL_\gamma^\Sigma\|\widetilde\Theta-\Theta^\star\|_{\Sigma_\star}.
\]
By Lemma~\ref{lem:local-covariance-map-lipschitz}, the second term is bounded by
\[
KL_{\Sigma,p}\|\widetilde\theta-\theta^\star\|_{\Sigma_{\theta,\star}}
\le KL_{\Sigma,p}\|\widetilde\Theta-\Theta^\star\|_{\Sigma_\star}.
\]
Thus, the transition block is bounded by $(\chi_pL_\gamma^\Sigma+KL_{\Sigma,p})\|\widetilde\Theta-\Theta^\star\|_{\Sigma_\star}$. The observation block is treated identically and is bounded by $(\chi_\Phi L_\gamma^\Sigma+KL_{\Sigma,\Phi})\|\widetilde\Theta-\Theta^\star\|_{\Sigma_\star}$. Since the complete-data covariance is block diagonal, the operator norm of the full normalized difference is the maximum of the two block norms.
\end{proof}

\begin{lemma}[Spectral approximation of the complete-data covariance]\label{lem:finite-sample-covariance-approximation}
Under the conditions of Lemma~\ref{lem:pilot-covariance-bias}, for every $\delta_{\mathrm{cov}}\in(0,1)$, with probability at least $1-\delta_{\mathrm{cov}}$,
\begin{equation}\label{eq:finite-sample-covariance-approximation}
\left\|\Sigma_\star^{-1/2}\bigl(\widehat\Sigma_\lambda-\Sigma_\star\bigr)\Sigma_\star^{-1/2}\right\|_{\mathrm{op}}
\le\varepsilon_{\mathrm{cov}},
\end{equation}
where
\begin{align}
\varepsilon_{\mathrm{cov}}
&:=L_{\mathrm{pilot}}^\Sigma\|\widetilde\Theta-\Theta^\star\|_{\Sigma_\star}
+\lambda\|\Sigma_\star^{-1}\|_{\mathrm{op}}\notag\\
&\qquad+K\,\chi_\Sigma\,\left(\sqrt{\frac{2\log(2d_\Theta/\delta_{\mathrm{cov}})}{M_0}}+\frac{2\log(2d_\Theta/\delta_{\mathrm{cov}})}{3M_0}\right).
\label{eq:finite-sample-covariance-error}
\end{align}
Consequently, if $\varepsilon_{\mathrm{cov}}<1$, then
\begin{equation}\label{eq:finite-sample-covariance-sandwich}
(1-\varepsilon_{\mathrm{cov}})\Sigma_\star\preceq\widehat\Sigma_\lambda\preceq(1+\varepsilon_{\mathrm{cov}})\Sigma_\star.
\end{equation}
\end{lemma}

\begin{proof}
Decompose
\[
\widehat\Sigma_\lambda-\Sigma_\star
=\Big(\widehat\Sigma_{M_0}(\widetilde\Theta)-\bE[\widehat\Sigma_Y(\widetilde\Theta)]\Big)
+\Big(\bE[\widehat\Sigma_Y(\widetilde\Theta)]-\Sigma_\star\Big)+\lambda I.
\]
Lemma~\ref{lem:pilot-covariance-bias} bounds the second term. Moreover, Definition~\ref{def:local-covariance-domination} gives
\[
0\preceq\Sigma_\star^{-1/2}\widehat\Sigma_Y(\widetilde\Theta)\Sigma_\star^{-1/2}\preceq K\chi_\Sigma I.
\]
Applying matrix Bernstein to the first term proves the claim.
\end{proof}

Now, by the second collection of trajectories, we form the empirical auxiliary function
\begin{equation}\label{eq:finite-sample-Q-function}
Q_M(\Theta_1,\Theta_2):=\frac{1}{M}\sum_{i=1}^{M}Q_{Y^{(i)}}(\Theta_1,\Theta_2).
\end{equation}

By Eqs.~\ref{eq:Em-grad-theta}--\ref{eq:Em-grad-w} and Remark~\ref{lem:centered-feature-covariance-leverage}, the single-trajectory gradient satisfies
\[
\sup_Y\sup_{\Theta\in\mathbb B_{\Sigma_\star}(\Theta^\star,r)}
\|\nabla_1Q_Y(\Theta,\Theta)\|_{\Sigma_\star^{-1}}\le K\sqrt{B_{p,\Sigma}^2+B_{\Phi,\Sigma}^2}.
\]

\begin{lemma}[Local Lipschitz continuity of the trajectory auxiliary gradient]\label{lem:trajectory-Q-gradient-lipschitz}
Let
\[
G_Y(\Theta):=\nabla_1Q_Y(\Theta,\Theta),
\]
and define
\[
L_G^\Sigma(r):=\sup_Y\sup_{\Theta\in\mathbb B_{\Sigma_\star}(\Theta^\star,r)}
\left\|\Sigma_\star^{-1/2}DG_Y(\Theta)\Sigma_\star^{-1/2}\right\|_{\mathrm{op}},
\]
where $DG_Y(\Theta)$ denotes the derivative of the diagonal map $\Theta\mapsto\nabla_1Q_Y(\Theta,\Theta)$. Then $L_G^\Sigma(r)<\infty$ and, for all $\Theta,\Theta'\in\mathbb B_{\Sigma_\star}(\Theta^\star,r)$,
\begin{equation}\label{eq:trajectory-Q-gradient-lipschitz}
\|G_Y(\Theta)-G_Y(\Theta')\|_{\Sigma_\star^{-1}}
\le L_G^\Sigma(r)\|\Theta-\Theta'\|_{\Sigma_\star}.
\end{equation}
\end{lemma}

\begin{proof}
The observation and state spaces are finite, the action sequence is fixed, and all log-linear transition and observation probabilities are strictly positive and smooth in the parameters. Hence, for every trajectory $Y$, the forward--backward marginals and $G_Y$ are continuously differentiable. Since the trajectory space is finite and $\mathbb B_{\Sigma_\star}(\Theta^\star,r)$ is compact, the displayed supremum is finite. For $\Theta_u:=\Theta'+u(\Theta-\Theta')$, the fundamental theorem of calculus gives
\[
G_Y(\Theta)-G_Y(\Theta')=\int_0^1DG_Y(\Theta_u)(\Theta-\Theta')\,du.
\]
Conjugating by the primal and dual covariance factors and using the definition of $L_G^\Sigma(r)$ proves \ref{eq:trajectory-Q-gradient-lipschitz}.
\end{proof}

\begin{lemma}[Uniform empirical auxiliary-gradient bound]\label{lem:explicit-uniform-empirical-Q-gradient}
For every $\delta_Q\in(0,1)$ and every $\tau\in(0,r]$, with probability at least $1-\delta_Q$,
\begin{equation}\label{eq:explicit-uniform-empirical-Q-gradient}
\sup_{\Theta\in\mathbb B_{\Sigma_\star}(\Theta^\star,r)}
\left\|\nabla_1Q_M(\Theta,\Theta)-\nabla_1Q(\Theta,\Theta)\right\|_{\Sigma_\star^{-1}}
\le\epsilon_{Q,\tau}^{\mathrm{unif},\Sigma}(M,\delta_Q),
\end{equation}
where
\begin{equation}\label{eq:explicit-uniform-Q-rate}
\epsilon_{Q,\tau}^{\mathrm{unif},\Sigma}(M,\delta_Q)
:=2L_G^\Sigma(r)\tau+2K\sqrt{B_{p,\Sigma}^2+B_{\Phi,\Sigma}^2}\sqrt{\frac{2\Lambda_\tau}{M}}+\frac{4K\sqrt{B_{p,\Sigma}^2+B_{\Phi,\Sigma}^2}\Lambda_\tau}{3M},
\end{equation}
and
\[
\Lambda_\tau:=d_\Theta\log\left(1+\frac{2r}{\tau}\right)+\log\frac{2}{\delta_Q}.
\]
Consequently, Definition~\ref{def:finite-sample-em-deviations} holds with
\[
\epsilon_Q^{\mathrm{unif},\Sigma}(M,\delta_Q)
:=\inf_{0<\tau\le r}\epsilon_{Q,\tau}^{\mathrm{unif},\Sigma}(M,\delta_Q).
\]
In particular, for $\tau=r/M$ we have
\begin{align*}
\epsilon_Q^{\mathrm{unif},\Sigma}(M,\delta_Q)
&\le\frac{2L_G^\Sigma(r)r}{M}
+2K\sqrt{B_{p,\Sigma}^2+B_{\Phi,\Sigma}^2}\sqrt{\frac{2\left[d_\Theta\log(1+2M)+\log(2/\delta_Q)\right]}{M}}\\
&\qquad\qquad+\frac{4K\sqrt{B_{p,\Sigma}^2+B_{\Phi,\Sigma}^2}\left[d_\Theta\log(1+2M)+\log(2/\delta_Q)\right]}{3M}.
\end{align*}

\end{lemma}

\begin{proof}
Let $\mathcal N_\tau$ be a $\tau$-net of $\mathbb B_{\Sigma_\star}(\Theta^\star,r)$ in the $\Sigma_\star$-norm. Since this ball has dimension $d_\Theta$, it can be chosen so that
\[
|\mathcal N_\tau|\le\left(1+\frac{2r}{\tau}\right)^{d_\Theta}.
\]
For fixed $\Theta\in\mathcal N_\tau$, define
\[
X_i(\Theta):=\Sigma_\star^{-1/2}\left(G_{Y^{(i)}}(\Theta)-\bE[G_Y(\Theta)]\right).
\]
The single-trajectory envelope gives $\|X_i(\Theta)\|_2\le2K\sqrt{B_{p,\Sigma}^2+B_{\Phi,\Sigma}^2}$. Vector Bernstein and a union bound over $\mathcal N_\tau$ therefore imply
\[
\max_{\Theta\in\mathcal N_\tau}\left\|\frac1M\sum_{i=1}^MX_i(\Theta)\right\|_2
\le2K\sqrt{B_{p,\Sigma}^2+B_{\Phi,\Sigma}^2}\sqrt{\frac{2\Lambda_\tau}{M}}+\frac{4K\sqrt{B_{p,\Sigma}^2+B_{\Phi,\Sigma}^2}\Lambda_\tau}{3M}
\]
with probability at least $1-\delta_Q$. For arbitrary $\Theta$ in the covariance ball, choose $\Theta^{\mathrm{net}}\in\mathcal N_\tau$ with $\|\Theta-\Theta^{\mathrm{net}}\|_{\Sigma_\star}\le\tau$. Lemma~\ref{lem:trajectory-Q-gradient-lipschitz} gives
\begin{align*}
&\left\|\nabla_1Q_M(\Theta,\Theta)-\nabla_1Q(\Theta,\Theta)\right\|_{\Sigma_\star^{-1}}\\
&\quad\le\left\|\nabla_1Q_M(\Theta^{\mathrm{net}},\Theta^{\mathrm{net}})-\nabla_1Q(\Theta^{\mathrm{net}},\Theta^{\mathrm{net}})\right\|_{\Sigma_\star^{-1}}+2L_G^\Sigma(r)\tau.
\end{align*}
Taking the supremum proves \ref{eq:explicit-uniform-empirical-Q-gradient}.
\end{proof}

\begin{lemma}\label{lem:empirical-covariance-geometry-transfer}
Suppose \ref{eq:finite-sample-covariance-sandwich} holds with $\varepsilon_{\mathrm{cov}}<1$. Then, for every vector $u$ and every vector $g$ of the corresponding dimensions,
\begin{equation}\label{eq:empirical-covariance-norm-comparison}
\sqrt{1-\varepsilon_{\mathrm{cov}}}\|u\|_{\Sigma_\star}\le\|u\|_{\widehat\Sigma_\lambda}\le\sqrt{1+\varepsilon_{\mathrm{cov}}}\|u\|_{\Sigma_\star},\qquad
\|g\|_{\widehat\Sigma_\lambda^{-1}}\le\frac{1}{\sqrt{1-\varepsilon_{\mathrm{cov}}}}\|g\|_{\Sigma_\star^{-1}}.
\end{equation}
Moreover, for every $\Theta\in\mathbb B_{\Sigma_\star}(\Theta^\star,r)$ with $r\le r_\varepsilon$,
\begin{align}
\frac{(1-\varepsilon)}{1+\varepsilon_{\mathrm{cov}}}\widehat\Sigma_\lambda
&\preceq-\nabla^2q(\Theta)\preceq\frac{1+\varepsilon}{1-\varepsilon_{\mathrm{cov}}}\widehat\Sigma_\lambda,\label{eq:empirical-covariance-curvature-transfer}\\
\left\|\nabla q(\Theta)-\nabla_1Q(\Theta,\Theta)\right\|_{\widehat\Sigma_\lambda^{-1}}
&\le\frac{L_Q^\Sigma(r)}{1-\varepsilon_{\mathrm{cov}}}\|\Theta-\Theta^\star\|_{\widehat\Sigma_\lambda}.\label{eq:empirical-covariance-gradient-stability}
\end{align}
\end{lemma}

\begin{proof}
The norm comparisons follow from \ref{eq:finite-sample-covariance-sandwich} and the inverse inequalities
\[
\frac{1}{1+\varepsilon_{\mathrm{cov}}}\Sigma_\star^{-1}\preceq\widehat\Sigma_\lambda^{-1}\preceq\frac{1}{1-\varepsilon_{\mathrm{cov}}}\Sigma_\star^{-1}.
\]
Combining \ref{eq:finite-sample-covariance-sandwich} with Lemma~\ref{lem:local-covariance-curvature} gives \ref{eq:empirical-covariance-curvature-transfer}. Finally, by Definition~\ref{def:covariance-gradient-stability} and \ref{eq:empirical-covariance-norm-comparison} 
\begin{align*}
\left\|\nabla q(\Theta)-\nabla_1Q(\Theta,\Theta)\right\|_{\widehat\Sigma_\lambda^{-1}}
&\le\frac{L_Q^\Sigma(r)}{\sqrt{1-\varepsilon_{\mathrm{cov}}}}\|\Theta-\Theta^\star\|_{\Sigma_\star}\\
&\le\frac{L_Q^\Sigma(r)}{1-\varepsilon_{\mathrm{cov}}}\|\Theta-\Theta^\star\|_{\widehat\Sigma_\lambda}.
\end{align*}
\end{proof}

\begin{algorithm}[H]
\caption{Covariance-preconditioned gradient EM}
\label{alg:covariance-preconditioned-gradient-em}
\begin{algorithmic}[1]
\STATE \textbf{Input:} Independent pilot trajectories $\{Y_0^{(i)}\}_{i=1}^{M_0}$, estimation trajectories $\{Y^{(i)}\}_{i=1}^{M}$, pilot parameter $\widetilde\Theta$, initialization $\widehat\Theta^{(0)}$, ridge parameter $\lambda>0$, step size $\eta>0$, and iteration count $T_{\mathrm{EM}}$.
\FOR{$i=1,\dots,M_0$}
\STATE Run the forward--backward recursion on $Y_0^{(i)}$ under $\widetilde\Theta$ and compute $\widehat\Sigma_{Y_0^{(i)}}(\widetilde\Theta)$ as in \ref{eq:single-trajectory-complete-data-covariance}.
\ENDFOR
\STATE Set $\widehat\Sigma_\lambda\gets M_0^{-1}\sum_{i=1}^{M_0}\widehat\Sigma_{Y_0^{(i)}}(\widetilde\Theta)+\lambda I$.
\FOR{$t=0,\dots,T_{\mathrm{EM}}-1$}
\FOR{$i=1,\dots,M$}
\STATE Run the forward--backward recursion on $Y^{(i)}$ under $\widehat\Theta^{(t)}$ to obtain $\gamma_{i,k}^{(t)}$ and $\xi_{i,k}^{(t)}$.
\ENDFOR
\STATE Compute $g^{(t)}\gets\nabla_1Q_M(\widehat\Theta^{(t)},\widehat\Theta^{(t)})$ using \ref{eq:Em-grad-theta}--\ref{eq:Em-grad-w}.
\STATE Update $\widehat\Theta^{(t+1)}\gets\widehat\Theta^{(t)}+\eta\widehat\Sigma_\lambda^{-1}g^{(t)}$.
\ENDFOR
\STATE \textbf{Output:} $\widehat\Theta^{(T_{\mathrm{EM}})}$.
\end{algorithmic}
\end{algorithm}

\begin{proof}[Proof of Theorem~\ref{thm:finite-sample-covariance-gradient-em}]
Work on the intersection of the events in Eqs.~\ref{eq:empirical-covariance-deviation-main} and \ref{eq:uniform-empirical-gradient-deviation-main}. On the covariance event, $\varepsilon_{\mathrm{cov}}<1$ implies
\[
(1-\varepsilon_{\mathrm{cov}})\Sigma_\star\preceq\widehat\Sigma_\lambda\preceq(1+\varepsilon_{\mathrm{cov}})\Sigma_\star,
\]
so Lemma~\ref{lem:empirical-covariance-geometry-transfer} applies. Fix $\Theta\in\mathbb B_{\Sigma_\star}(\Theta^\star,r)$ and write $\Delta:=\Theta-\Theta^\star$. By adding and subtracting the population and oracle gradients,
\begin{align*}
\left\|\Delta+\eta\widehat\Sigma_\lambda^{-1}\nabla_1Q_M(\Theta,\Theta)\right\|_{\widehat\Sigma_\lambda}
&\le\left\|\Delta+\eta\widehat\Sigma_\lambda^{-1}\nabla q(\Theta)\right\|_{\widehat\Sigma_\lambda}+\eta\left\|\nabla_1Q(\Theta,\Theta)-\nabla q(\Theta)\right\|_{\widehat\Sigma_\lambda^{-1}}
\\
&\qquad\qquad+\eta\left\|\nabla_1Q_M(\Theta,\Theta)-\nabla_1Q(\Theta,\Theta)\right\|_{\widehat\Sigma_\lambda^{-1}}.
\end{align*}
By Eq.~\ref{eq:empirical-covariance-curvature-transfer} and the same mean-value argument used in the proof of Proposition~\ref{prop:covariance-gradient-em-convergence}, the first term is at most $\widehat\rho_\eta\|\Delta\|_{\widehat\Sigma_\lambda}$. By Eq.~\ref{eq:empirical-covariance-gradient-stability}, the second term is at most
\(
\frac{\eta L_Q^\Sigma(r)}{1-\varepsilon_{\mathrm{cov}}}\|\Delta\|_{\widehat\Sigma_\lambda},
\)
while Eqs.~\ref{eq:uniform-empirical-gradient-deviation-main} and \ref{eq:empirical-covariance-norm-comparison} bound the final term by
\(
\frac{\eta}{\sqrt{1-\varepsilon_{\mathrm{cov}}}}\epsilon_Q^{\mathrm{unif},\Sigma}(M,\delta_Q).
\)
Hence,
\begin{equation}\label{eq:finite-sample-em-one-step-recursion}
\|\widehat\Theta^{(t+1)}-\Theta^\star\|_{\widehat\Sigma_\lambda}
\le\widehat\kappa_{\mathrm{EM}}\|\widehat\Theta^{(t)}-\Theta^\star\|_{\widehat\Sigma_\lambda}
+\frac{\eta}{\sqrt{1-\varepsilon_{\mathrm{cov}}}}\epsilon_Q^{\mathrm{unif},\Sigma}(M,\delta_Q).
\end{equation}
Let $R:=r\sqrt{1-\varepsilon_{\mathrm{cov}}}$. If $\|\widehat\Theta^{(t)}-\Theta^\star\|_{\widehat\Sigma_\lambda}\le R$, then Eq.~\ref{eq:empirical-covariance-norm-comparison} implies $\|\widehat\Theta^{(t)}-\Theta^\star\|_{\Sigma_\star}\le r$, and hence $\widehat\Theta^{(t)}\in\mathbb B_{\Sigma_\star}(\Theta^\star,r)$. Using Eq.~\ref{eq:finite-sample-em-deviation-condition} in Eq.~\ref{eq:finite-sample-em-one-step-recursion} gives
\begin{align*}
\|\widehat\Theta^{(t+1)}-\Theta^\star\|_{\widehat\Sigma_\lambda}
&\le\widehat\kappa_{\mathrm{EM}}R+(1-\widehat\kappa_{\mathrm{EM}})R\\
&=R.
\end{align*}
On the covariance event, Eq.~\ref{eq:empirical-covariance-norm-comparison} and Eq.~\ref{eq:finite-sample-em-initialization} imply

$$
\|\widehat\Theta^{(0)}-\Theta^\star\|_{\widehat\Sigma_\lambda}
\le \sqrt{1+\varepsilon_{\mathrm{cov}}}\,
\|\widehat\Theta^{(0)}-\Theta^\star\|_{\Sigma_\star}
\le r\sqrt{1-\varepsilon_{\mathrm{cov}}}=R.
$$

Thus, Eq.~\ref{eq:finite-sample-em-initialization} and induction show that all iterates remain in the local covariance ball. Unrolling Eq.~\ref{eq:finite-sample-em-one-step-recursion} yields
\begin{equation}\label{eq:finite-sample-em-contraction-empirical}
\|\widehat\Theta^{(t)}-\Theta^\star\|_{\widehat\Sigma_\lambda}
\le\widehat\kappa_{\mathrm{EM}}^t\|\widehat\Theta^{(0)}-\Theta^\star\|_{\widehat\Sigma_\lambda}
+\frac{\eta(1-\widehat\kappa_{\mathrm{EM}}^t)}
{(1-\widehat\kappa_{\mathrm{EM}})\sqrt{1-\varepsilon_{\mathrm{cov}}}}
\epsilon_Q^{\mathrm{unif},\Sigma}(M,\delta_Q).
\end{equation}
Applying Eq.~\ref{eq:empirical-covariance-norm-comparison} to both the left-hand side and the initialization term gives Eq.~\ref{eq:finite-sample-em-contraction-population}. The probability statement follows by a union bound.
\end{proof}

%%%%%%%%%%%%%%%%%%%%%%%%%%%%%%%%%%%%%%%%%%
%%%%%%%%%%%%%%%%%%%%%%%%%%%%%%%%%%%%%%%%%%
%%%%%%%%%%%%%%%%%%%%%%%%%%%%%%%%%%%%%%%%%%
%%%%%%%%%%%%%%%%%%%%%%%%%%%%%%%%%%%%%%%%%%
%%%%%%%%%%%%%%%%%%%%%%%%%%%%%%%%%%%%%%%%%%
%%%%%%%%%%%%%%%%%%%%%%%%%%%%%%%%%%%%%%%%%%
%%%%%%%%%%%%%%%%%%%%%%%%%%%%%%%%%%%%%%%%%%
%%%%%%%%%%%%%%%%%%%%%%%%%%%%%%%%%%%%%%%%%%

\begin{corollary}[End-to-end reward learning after finite-sample POMDP estimation]\label{cor:end-to-end-em-reward-learning}
Suppose the conditions of Theorems~\ref{thm:finite-sample-covariance-gradient-em} and \ref{thm:main-mu} hold, the POMDP-estimation trajectories are independent of the preference sample, and $T-1\le K$. Fix an EM iteration $\ell\ge0$ and confidence levels $\delta_{\mathrm{cov}},\delta_Q,\delta_b,\delta_c\in(0,1)$. Let $\mathcal R_{\mathrm{EM}}^{(\ell)}$ denote the right-hand side of Eq.~\ref{eq:finite-sample-em-contraction-population}, and set
\[
\epsilon_{b,\mathrm{EM}}^{(\ell)}
:=\left.\epsilon_b(\delta_b,2N_{\mathrm{HF}})\right|_{c_b(\delta(\theta),\delta(w))=C_b^\Sigma\mathcal R_{\mathrm{EM}}^{(\ell)}}.
\]
Construct the filtering beliefs, perturbed features $\{\widetilde\phi_i^{(\ell)}\}_{i=1}^{N_{\mathrm{HF}}}$, empirical covariance $\widetilde\Sigma^{(\ell)}$, and constrained Bradley--Terry estimator $\widetilde\mu^{(\ell)}$ as in Section~\ref{sec:rlhf}, using $\widehat\Theta^{(\ell)}$ in place of $\Theta$. For any $c_\zeta>1$, let
\(
\zeta_\ell:=16c_\zeta T^2B_r^2\epsilon_{b,\mathrm{EM}}^{(\ell)}\bigl(1+\epsilon_{b,\mathrm{EM}}^{(\ell)}\bigr).
\)
Then, with probability at least $1-\delta_{\mathrm{cov}}-\delta_Q-\delta_b-\delta_c$,
\begin{align}
\|\widetilde\mu^{(\ell)}-\mu^\star\|_{\widetilde\Sigma^{(\ell)}+\zeta_\ell I}
&\le\frac{2\sqrt{c_\zeta}}{\rho\sqrt{N_{\mathrm{HF}}(c_\zeta-1)}}
\sqrt{d\log\left(1+\frac{4T^2B_r^2}{\zeta_\ell d}\right)+2\log\left(\frac{1}{\delta_c}\right)}\notag\\
&\qquad\qquad+\frac{8TB_r\epsilon_{b,\mathrm{EM}}^{(\ell)}}{\rho\sqrt{\zeta_\ell}}
\left(1+\frac{1}{2}TB_rr_\mu\right)+2r_\mu\sqrt{\zeta_\ell}.
\label{eq:end-to-end-em-reward-bound}
\end{align}
\end{corollary}

\begin{proof}[Proof of Corollary \ref{cor:end-to-end-em-reward-learning}]
Let $\mathcal E_{\mathrm{EM}}$ be the event in Definition~\ref{def:finite-sample-em-deviations}. By Theorem~\ref{thm:finite-sample-covariance-gradient-em}, $\bP(\mathcal E_{\mathrm{EM}})\ge1-\delta_{\mathrm{cov}}-\delta_Q$ and, on this event,
\[
\|\widehat\Theta^{(\ell)}-\Theta^\star\|_{\Sigma_\star}\le\mathcal R_{\mathrm{EM}}^{(\ell)}.
\]
Remark~\ref{rem:em-reward-interface} then yields $c_b^\Sigma(\widehat\Theta^{(\ell)},\Theta^\star)\le C_b^\Sigma\mathcal R_{\mathrm{EM}}^{(\ell)}$. Conditioning on the independent POMDP-estimation data, Corollary~\ref{cor:belief-good-condition-cor} and the union bound in Eq.~\ref{eq:Eb-RLHF-uniform} give the belief-accuracy level $\epsilon_{b,\mathrm{EM}}^{(\ell)}$ with probability at least $1-\delta_b$. Applying Theorem~\ref{thm:main-mu} with $\epsilon_b(\delta_b,2N_{\mathrm{HF}})$ and $\zeta$ replaced by $\epsilon_{b,\mathrm{EM}}^{(\ell)}$ and $\zeta_\ell$ proves the stated bound.
\end{proof}

\begin{lemma}[Clean negative log-likelihood gradient bound]\label{lem:clean-like-grad-bound}
Consider the \emph{clean} (\ref{subsec:features-mle}) negative log-likelihood gradient
\[
\nabla L(\mu^\star)= \frac{1}{N_{\mathrm{HF}}}\sum_{i=1}^{N_{\mathrm{HF}}}\bigl(\sigma(\phi_i^\top\mu^\star)-y_i\bigr)\phi_i.
\]
where $\sigma$ is the sigmoid function and $\phi_i\in\{\phi_j\}_{j=1}^{N_{\mathrm{HF}}}\subset\bR^d$ is the given trajectory-level exact (clean) accumulated features differences for $i-th$ trajectory, defined in \ref{eq:accu-feature-exact}. Given $\zeta >0$ define the clean empirical covariance
\(
\Sigma \;:=\;\frac{1}{N_{\mathrm{HF}}}\sum_{i=1}^{N_{\mathrm{HF}}}\phi_i\phi_i^\top
\)
and regularized empirical covariance by $\Sigma+\zeta I$. Then, for any $\delta_c\in(0,1)$ with probability at least $1-\delta_c$ it holds
\begin{equation}\label{eq:score_bound_tildeSigma}
\bigl\|\nabla L(\mu^\star)\bigr\|_{(\Sigma+\zeta I)^{-1}}\;\le\;\frac{1}{\sqrt{N_{\mathrm{HF}}}}
\sqrt{d\log \Big(1+\frac{4T^2 B_r^2}{\zeta d}\Big)+2\log \Big(\frac{1}{\delta_c}\Big)}.
\end{equation}
\end{lemma}

\begin{proof}[Proof of lemma \ref{lem:clean-like-grad-bound}]
Define the shorthand $\xi_i := \sigma(\phi_i^\top\mu^\star)-y_i \in [-1,1],$ $S:=\sum_{i=1}^{N_{\mathrm{HF}}}\xi_i\phi_i$, and $V := N_{\mathrm{HF}}\,\zeta I + \sum_{i=1}^{N_{\mathrm{HF}}}\phi_i\phi_i^\top$. By Assumption \ref{assum:reward-rlhf}, item $3$ (Preference Realizability) and Remark \ref{remark:label-cond-indep}, for $0\le i\le N_{\mathrm{HF}}-1$ we have \(\bE[\xi_i \mid \phi_i]=\sigma(\phi_i^\top \mu^\star) - \bE[y_i \mid \phi_i]\notag=\sigma(\phi_i^\top \mu^\star) - \sigma(\phi_i^\top \mu^\star)\notag=
0.\) Also, since we have $\xi_i\in[-1,1]$, by Hoeffding's lemma $\xi_i$ is conditionally $1$-sub-Gaussian given $\phi_i$, i.e.,
\[
\forall t\in\bR,\qquad \bE\big[\exp(t\xi_i)\,|\,\phi_i\big]\le\exp(\frac{\,t^2}{\,2})
\]

Then, by Theorem 1 of \citet{NIPS2011_e1d5be1c}, for any $\delta_c\in(0,1)$, with probability at
least $1-\delta_c$
\begin{align}
\label{eq:self-norm}
\|S\|_{V^{-1}} &\le\; \sqrt{2\log \Bigg(\frac{\det(V)^{1/2}}{\det(N_{\mathrm{HF}}\zeta I)^{1/2}}\cdot\frac{1}{\delta_c}\Bigg)}\notag\\
&= \sqrt{\log\det\Bigg(\frac{1}{N_{\mathrm{HF}}\zeta }I \cdot \left( N_{\mathrm{HF}}\zeta I + \sum_{i=1}^{N_{\mathrm{HF}}}\phi_i\phi_i^\top \right)\Bigg)+2\log(\frac{1}{\delta_c})}\notag\\
&= \sqrt{\log\det\left( I + \frac{1}{N_{\mathrm{HF}}\zeta }\sum_{i=1}^{N_{\mathrm{HF}}}\phi_i\phi_i^\top \right)+2\log(\frac{1}{\delta_c})}\notag\\
&\le \sqrt{d\log \Big(1+\frac{(2T B_r)^2}{\zeta d}\Big)+2\log \Big(\frac{1}{\delta_c}\Big)}
\end{align}

Where the last inequality follows from Lemma 10 in \citet{NIPS2011_e1d5be1c} (determinant-trace inequality). Finally, note that $\nabla L(\mu^\star)=\frac{1}{N_{\mathrm{HF}}}S$ and \(
\Sigma+\zeta I
=\frac{1}{N_{\mathrm{HF}}}\sum_{i=1}^{N_{\mathrm{HF}}}\phi_i\phi_i^\top + \zeta I=\frac{1}{N_{\mathrm{HF}}}V
\), we have
\begin{align}
\bigl\|\nabla L(\mu^\star)\bigr\|_{(\Sigma+\zeta I)^{-1}}
&=\Bigl\|\frac{1}{N_{\mathrm{HF}}}S\Bigr\|_{(\frac{1}{N_{\mathrm{HF}}}V)^{-1}}=\sqrt{\Big(\frac{1}{N_{\mathrm{HF}}}S\Big)^\top \big(N_{\mathrm{HF}}V^{-1}\big)\Big(\frac{1}{N_{\mathrm{HF}}}S\Big)}\notag\\
&=\frac{1}{\sqrt{N_{\mathrm{HF}}}}\,\|S\|_{V^{-1}}.\label{eq:scaling_fix}
\end{align} This result combined with \ref{eq:self-norm} completes the proof.

\end{proof}

\begin{lemma}[Trajectory-level feature perturbation bound]\label{lem:Delta-Bound}
Fix $i\in\{1,\dots,N_{\mathrm{HF}}\}$ and define the clean and perturbed trajectory-level feature differences
$\phi_i$ and $\tilde\phi_i$ as in \ref{eq:accu-feature-exact}--\ref{eq:accu-feature-perturbed}. Let \(
\Delta_{\phi,i}\;:=\;\tilde\phi_i-\phi_i.
\) Suppose the event
$\mathcal E_{b}^{\mathrm{RLHF}}(\delta_b)$ in \ref{eq:Eb-RLHF-uniform} holds. Then, it holds
\begin{align}
\|\Delta_{\phi,i}\|_2
&\le\;4T B_r\,\epsilon_b(\delta_b, 2N_{\mathrm{HF}}) 
\label{eq:Delta-upper-bound-lemma}
\end{align}
where $\epsilon_b(\delta_b, 2N_{\mathrm{HF}})$ is defined in \ref{eq:eps-b-bar-def}.
\end{lemma}

\begin{proof}[Proof of lemma \ref{lem:Delta-Bound}]\label{proof:lem.1}
By the triangle inequality and \ref{eq:belief-feature-lip} we have
\begin{align}
\|\Delta_{\phi,i}\|_2 &= \left\| \sum_{h=0}^{T-1} \bigl[\phi_b(b_{i,h}^{(1),\Theta}, a_{i,h}^{(1)}) - \phi_b(b_{i,h}^{(1),\Theta^\star},a_{i,h}^{(1)})\bigr] 
- \sum_{h=0}^{T-1} \bigl[\phi_b(b_{i,h}^{(2),\Theta}, a_{i,h}^{(2)}) - \phi_b(b_{i,h}^{(2),\Theta^\star},a_{i,h}^{(2)})\bigr] \right\|_2 \notag\\
&\le \sum_{h=0}^{T-1} \bigl\|\phi_b(b_{i,h}^{(1),\Theta}, a_{i,h}^{(1)}) - \phi_b(b_{i,h}^{(1),\Theta^\star},a_{i,h}^{(1)})\bigr\|_2 
+ \sum_{h=0}^{T-1} \bigl\|\phi_b(b_{i,h}^{(2),\Theta}, a_{i,h}^{(2)}) - \phi_b(b_{i,h}^{(2),\Theta^\star},a_{i,h}^{(2)})\bigr\|_2 \notag\\
&\le \sum_{h=0}^{T-1} B_r\|b_{i,h}^{(1),\Theta}-b_{i,h}^{(1),\Theta^\star}\|_1
+\sum_{h=0}^{T-1} B_r\|b_{i,h}^{(2),\Theta}-b_{i,h}^{(2),\Theta^\star}\|_1 \notag\\
&= 2B_r\sum_{h=0}^{T-1}\|b_{i,h}^{(1),\Theta}-b_{i,h}^{(1),\Theta^\star}\|_{\mathrm{TV}}
+2B_r\sum_{h=0}^{T-1}\|b_{i,h}^{(2),\Theta}-b_{i,h}^{(2),\Theta^\star}\|_{\mathrm{TV}}.\label{eq:sum-action-diff}
\end{align}
On the event $\mathcal E_{b}^{\mathrm{RLHF}}(\delta_b)$, for each $j\in\{1,2\}$,
\[
\frac{1}{T}\sum_{h=0}^{T-1}\|b_{i,h}^{(j),\Theta}-b_{i,h}^{(j),\Theta^\star}\|_{\mathrm{TV}} \le \epsilon_b(\delta_b, 2N_{\mathrm{HF}}),
\]
hence \(\sum_{h=0}^{T-1}\|b_{i,h}^{(j),\Theta}-b_{i,h}^{(j),\Theta^\star}\|_{\mathrm{TV}}
\le T\epsilon_b(\delta_b, 2N_{\mathrm{HF}})\).
Substituting into \ref{eq:sum-action-diff} yields
\[
\|\Delta_{\phi,i}\|_2\le 2B_r\cdot T\epsilon_b(\delta_b, 2N_{\mathrm{HF}})+2B_r\cdot T\epsilon_b(\delta_b, 2N_{\mathrm{HF}}) =4T B_r\,\epsilon_b(\delta_b, 2N_{\mathrm{HF}}),
\]
which proves Eq.~\ref{eq:Delta-upper-bound-lemma}.

\end{proof}

\section{Proof of Theorem \ref{thm:main-mu}}
\begin{proof}\label{proof:thm-rlhf}
 Assume the true reward model parameter $\mu^\star$ exists and consider the set $\bB_2(r_\mu)$ (Assumption \ref{assum:reward-rlhf}). Define the maximum likelihood estimation objective for perturbed features (\ref{eq:accu-feature-perturbed})
\begin{align}
\tilde L(\mu) &:=-\frac{1}{N_{\mathrm{HF}}}\sum_{i=1}^{N_{\mathrm{HF}}} \log\left(\frac{\mathbf{1}\{y_i = 1\}}{1 + \exp\big(-(\,\,\tilde{\phi}^{\tau^{(1)}_i,\tau^{(2)}_i})^\top\mu\big)}+\frac{\mathbf{1}\{y_i = 0\}}{1 +\exp\big((\,\,\tilde{\phi}^{\tau^{(1)}_i, \tau^{(2)}_i})^\top\mu\big)} \right)\notag\\
& =\; -\frac{1}{N_{\mathrm{HF}}}\sum_{i=1}^{{N_{\mathrm{HF}}}} \Bigl[y_i\ln\sigma(\langle \tilde\phi_i,\mu\rangle)+(1-y_i)\ln\bigl(1-\sigma(\langle \tilde\phi_i,\mu\rangle)\bigr)\Bigr]\notag\\
&=\; -\frac{1}{N_{\mathrm{HF}}}\sum_{i=1}^{{N_{\mathrm{HF}}}} \Bigl[y_i\ln\sigma(\langle \phi_i+\Delta_{\phi,i},\mu\rangle)+(1-y_i)\ln\bigl(1-\sigma(\langle \phi_i+\Delta_{\phi,i},\mu\rangle)\bigr)\Bigr] \label{eq:likelihood-perturbed}
\end{align}
where $\sigma(\cdot)$ is the sigmoid function. The gradient and Hessian of $\tilde L(\mu)$ are

\begin{align}
\nabla \tilde L(\mu)&=\frac{1}{N_{\mathrm{HF}}}\sum_{i=1}^{N_{\mathrm{HF}}}
\bigl(\sigma(\tilde\phi_i^\top\mu)-y_i\bigr)\tilde\phi_i\notag\\
&=\frac{1}{N_{\mathrm{HF}}}\sum_{i=1}^{N_{\mathrm{HF}}}
\bigl(\sigma((\phi_i+ \Delta_{\phi,i})^\top\mu)-y_i\bigr)(\phi_i+ \Delta_{\phi,i})\label{mle-grad}
\end{align}

\begin{equation}\label{eq:mle-hessian}
\nabla^2 \tilde L(\mu)
=\frac{1}{N_{\mathrm{HF}}}\sum_{i=1}^{N_{\mathrm{HF}}}\sigma(\tilde\phi_i^\top\mu)\bigl(1-\sigma(\tilde\phi_i^\top\mu)\bigr)\,
\tilde\phi_i\tilde\phi_i^\top.
\end{equation}

The function $\sigma(z)(1-\sigma(z))$
is even and strictly decreasing on $[0,\infty)$. For any $\mu\in\bB_2(r_\mu)$, we have \(
|\tilde\phi_i^\top\mu| \le \|\tilde\phi_i\|_2\|\mu\|_2 \le 2T B_r r_\mu,
\) then for all $i$, 
$$\sigma(\tilde\phi^\top_i\mu)(1-\sigma(\tilde\phi^\top_i\mu))\ge\frac{1}{2+\exp(2T B_rr_\mu)+\exp(-2T B_rr_\mu)}=\rho $$
 substituting into \ref{eq:mle-hessian} yields
\begin{align}
\nabla^2 \tilde L(\mu) &=\frac{1}{N_{\mathrm{HF}}}\sum_{i=1}^{N_{\mathrm{HF}}}\sigma(\tilde\phi_i^\top\mu)\bigl(1-\sigma(\tilde\phi_i^\top\mu)\bigr)\,
\tilde\phi_i\tilde\phi_i^\top\\
&\succeq\rho\bigl( \frac{1}{N_{\mathrm{HF}}}\sum_{i=1}^{N_{\mathrm{HF}}}\tilde\phi_i\tilde\phi_i^\top\bigr)
=\rho\tilde\Sigma
\end{align} 
where $\tilde\Sigma$ is the feature covariance matrix defined in \ref{eq:sigma-cov-hf}. Now, for any $\mu_1$ and $\mu_2$ in $\bB_2(r_\mu)$ and line segment $\mu_\lambda=\lambda\mu_1+(1-\lambda)\mu_2$ with $\lambda\in[0,1]$, we apply the second order Taylor expansion

\begin{align}
&\tilde L(\mu_1)-\tilde L(\mu_2)-\langle\nabla \tilde L(\mu_2),\mu_1-\mu_2\rangle \notag\\
&\qquad=\int_0^1(1-\lambda)\,(\mu_1-\mu_2)^\top\nabla^2 \tilde L(\mu_\lambda)\,(\mu_1-\mu_2)\,d\lambda \notag\\
&\qquad\ge \frac{\rho}{2}\,(\mu_1-\mu_2)^\top(\tilde\Sigma)(\mu_1-\mu_2). \label{eq:taylor-sc}
\end{align}
Therefore,
\begin{align}
\tilde L(\mu_1)-\tilde L(\mu_2)-\langle\nabla \tilde L(\mu_2),\mu_1-\mu_2\rangle
&\ge \frac{\rho}{2} \|\mu_1-\mu_2\|_{\tilde\Sigma}^2. 
\end{align}
This proves the $\rho-$strong convexity of $\tilde L(\mu)$ with $\mu\in\bB_2(r_\mu)$ with respect to the semi-norm $\|\cdot\|_{\tilde\Sigma}$. Next, writing \ref{eq:taylor-sc} for $\mu^\star$ and $\tilde\mu $ yields
\begin{align}
\frac{\rho}{2}\,\|\tilde\mu-\mu^\star\|_{\tilde\Sigma}^2
&\le\; \tilde L(\tilde\mu)-\tilde L(\mu^\star)-\langle\nabla\tilde L(\mu^\star),\tilde\mu-\mu^\star\rangle\notag\\
&\overset{(a)}{\le} -\langle\nabla\tilde L(\mu^\star),\tilde\mu-\mu^\star\rangle\notag\\
&\overset{(b)}{\le} \|\nabla\tilde L(\mu^\star)\|_{(\tilde\Sigma+\zeta I)^{-1}}\,
\|\tilde\mu-\mu^\star\|_{(\tilde\Sigma+\zeta I)}\label{eq:sc-mu-reg-norm}
\end{align}
where $(a)$ follows from the fact that $\tilde L(\tilde\mu)-\tilde L(\mu^\star)\le 0$ as $\tilde\mu\in\argmin_{\mu\in\bB_2(r_\mu)}\tilde L(\mu)$, and $(b)$ follows from applying Cauchy--Schwarz in the dual pair of norms induced by $\tilde\Sigma+\zeta I$. Using the identity $\|\tilde\mu-\mu^\star\|_{\tilde\Sigma}^2
=\|\tilde\mu-\mu^\star\|_{(\tilde\Sigma+\zeta I)}^2
-\zeta \,\|\tilde\mu-\mu^\star\|_2^2$, we have
\begin{align}
\frac{\rho}{2}\,\|\tilde\mu-\mu^\star\|_{(\tilde\Sigma+\zeta I)}^2
&\le \|\nabla\tilde L(\mu^\star)\|_{(\tilde\Sigma+\zeta I)^{-1}}\, \|\tilde\mu-\mu^\star\|_{(\tilde\Sigma+\zeta I)}
+2\rho\,\zeta \,r_\mu^2.
\end{align}
Solving this quadratic inequality and using $\sqrt{a+b}\le \sqrt{a}+\sqrt{b}$ yields
\begin{align}
\|\tilde\mu-\mu^\star\|_{(\tilde\Sigma+\zeta I)} &\le \; \frac{2}{\rho}\, \|\nabla\tilde L(\mu^\star)\|_{(\tilde\Sigma+\zeta I)^{-1}} \;+\;2r_\mu\sqrt{\zeta }.
\label{eq:mu-opt-bound-grad-final}
\end{align}
Now we proceed with bounding $\|\nabla\tilde L(\mu^\star)\|_{(\tilde\Sigma+\zeta I)^{-1}}$ with reference to \ref{mle-grad}. we define per-sample gradients at $\mu^\star$ as \(g_i^\star
:= \bigl(\sigma(\phi_i^\top\mu^\star)-y_i\bigr)\phi_i\) and \(\tilde g_i^\star
:= \bigl(\sigma(\tilde\phi_i^\top\mu^\star)-y_i\bigr)\tilde\phi_i\), with respect to exact and perturbed features. Hence, $\nabla\tilde L(\mu^\star)=\frac{1}{N_{\mathrm{HF}}}\sum_i\tilde g_i(\mu^\star)$, and we can write
\begin{align}\label{eq:likelihood-grad-decomp}
\nabla\tilde L(\mu^\star)
=\nabla L(\mu^\star)+\frac{1}{N_{\mathrm{HF}}}\sum_{i=1}^{N_{\mathrm{HF}}}\bigl(\tilde g_i^\star-g_i^\star\bigr).
\end{align}
Consider the following decomposition
\begin{align*}
\tilde g_i^\star-g_i^\star
&=\underbrace{\Bigl(\sigma(\phi_i^\top\mu^\star+\Delta_{\phi,i}^\top\mu^\star)-\sigma(\phi_i^\top\mu^\star)\Bigr)\phi_i}_{(I)}\;+\;\underbrace{\Bigl(\sigma(\phi_i^\top\mu^\star+\Delta_{\phi,i}^\top\mu^\star)-y_i\Bigr)\Delta_{\phi,i}}_{(II)}.
\end{align*}

Since \(\sigma(\cdot)\) is \(\tfrac14\)-Lipschitz, \(\sigma(\cdot)\in(0,1)\) and \(y_i\in\{0,1\}\), we have
\[
\|(I)\|_2\le\frac14\,|\Delta_{\phi,i}^\top\mu^\star|\,\|\phi_i\|_2\le\frac{2T B_r}{4}\|\Delta_{\phi,i}\|_2\|\mu^\star\|_2\le \frac{T B_r}{2}r_\mu\|\Delta_{\phi,i}\|_2
\]
\[\|(II)\|_2\le\|\Delta_{\phi,i}\|_2.
\]
where we used  \(\mu^\star\in\bB_2(r_\mu)\), \(|\Delta_{\phi,i}^\top\mu^\star|\le \|\Delta_{\phi,i}\|_2\|\mu^\star\|_2\le r_\mu\|\Delta_{\phi,i}\|_2\),
together with \(\|\phi_i\|_2\le 2T B_r\). Thus, for $i=1,\dots,N_{\mathrm{HF}}$, we obtain
\begin{equation}\label{eq:gi-perturb-bound}
\|\tilde g_i^\star-g_i^\star\|_2 \le \Bigl(1+\tfrac12\,T B_r\,r_\mu\Bigr)\,\|\Delta_{\phi,i}\|_2,
\end{equation}
which yields
\begin{align}
\bigl\|\nabla\tilde L(\mu^\star)-\nabla L(\mu^\star)\bigr\|_2
&=\Bigl\|\frac{1}{N_{\mathrm{HF}}}\sum_{i=1}^{N_{\mathrm{HF}}}(\tilde g_i^\star-g_i^\star)\Bigr\|_2 \;\le\; \frac{1}{N_{\mathrm{HF}}}\sum_{i=1}^{N_{\mathrm{HF}}}\|\tilde g_i^\star-g_i^\star\|_2 \notag\\
&\le \Bigl(1+\tfrac12\,T B_r\,r_\mu\Bigr)\cdot \frac{1}{N_{\mathrm{HF}}}\sum_{i=1}^{N_{\mathrm{HF}}}\|\Delta_{\phi,i}\|_2. \label{eq:score-shift-euclid}
\end{align}

Since \(\tilde\Sigma+\zeta I\succeq \zeta I\), the inequality

\begin{align}
\bigl\|\nabla\tilde L(\mu^\star)-\nabla L(\mu^\star)\bigr\|_{(\tilde\Sigma+\zeta I)^{-1}}
&\le \frac{1}{\sqrt{\zeta }}\bigl\|\nabla\tilde L(\mu^\star)-\nabla L(\mu^\star)\bigr\|_{2} \notag\\
&\le \frac{1}{\sqrt{\zeta }}\,
\Bigl(1+\tfrac12\,T B_r\,r_\mu\Bigr)\cdot
\frac{1}{N_{\mathrm{HF}}}\sum_{i=1}^{N_{\mathrm{HF}}}\|\Delta_{\phi,i}\|_2
\end{align}
on the event $\mathcal{E}_{b}^{\mathrm{RLHF}}(\delta_b)$ Eq.~\ref{eq:Eb-RLHF-uniform} where \(\|\Delta_{\phi,i}\|_2\le 4T B_r\,\epsilon_b(\delta_b, 2N_{\mathrm{HF}})\) for all \(i\) with the probability at least $1-\delta_b$, implies
\begin{align}\label{eq:grad-like-tilde-bound-half}
\bigl\|\nabla\tilde L(\mu^\star) \bigr\|_{(\tilde\Sigma+\zeta I)^{-1}}
&\le \bigl\|\nabla L(\mu^\star)\bigr\|_{(\tilde\Sigma+\zeta I)^{-1}}
+\frac{1}{\sqrt{\zeta }}\,
\Bigl(1+\tfrac12\,T B_r\,r_\mu\Bigr)\cdot
\frac{1}{N_{\mathrm{HF}}}\sum_{i=1}^{N_{\mathrm{HF}}}\|\Delta_{\phi,i}\|_2\notag\\
&\le \bigl\|\nabla L(\mu^\star)\bigr\|_{(\tilde\Sigma+\zeta I)^{-1}} + \frac{4T B_r\,\epsilon_b(\delta_b, 2N_{\mathrm{HF}})}{\sqrt{\zeta }}\,
\Bigl(1+\tfrac12\,T B_r\,r_\mu\Bigr).
\end{align}
By lemma \ref{lem:clean-like-grad-bound}, for any $\delta_c\in(0,1)$ with probability at least $1-\delta_c$ we have
 \begin{equation}\label{eq:grad-like-bound-half-way}
\bigl\|\nabla L(\mu^\star) \bigr\|_{(\Sigma+\zeta I)^{-1}}\le\frac{1}{\sqrt{N_{\mathrm{HF}}}}
\sqrt{d\log \Big(1+\frac{4T^2 B_r^2}{\zeta d}\Big)+2\log \Big(\frac{1}{\delta_c}\Big)}.
\end{equation}
Now, we study the covariance perturbation. Define
\[
E:= \tilde\Sigma-\Sigma=\frac{1}{N_{\mathrm{HF}}}\sum_{i=1}^{N_{\mathrm{HF}}}
\Big(\phi_i\Delta_{\phi,i}^\top+\Delta_{\phi,i}\phi_i^\top+\Delta_{\phi,i}\Delta_{\phi,i}^\top\Big)
\]
which is a symmetric. On $\mathcal{E}_{b}^{\mathrm{RLHF}}(\delta_b)$, we have $\|\Delta_{\phi,i}\|_2\le 4T B_r\,\epsilon_b(\delta_b, 2N_{\mathrm{HF}})$, for all $i$. 
Moreover, we have $\|\phi_i\|_2\le 2T B_r $,
hence, using $\|uv^\top\|_{op}\le \|u\|_2\|v\|_2$ and the triangle inequality,
\begin{align}
\|E\|_{op}
&\le \frac{1}{N_{\mathrm{HF}}}\sum_{i=1}^{N_{\mathrm{HF}}}
\Big(\|\phi_i\Delta_{\phi,i}^\top\|_{op}+\|\Delta_{\phi,i}\phi_i^\top\|_{op}
+\|\Delta_{\phi,i}\Delta_{\phi,i}^\top\|_{op}\Big)\notag\\
&\le \frac{1}{N_{\mathrm{HF}}}\sum_{i=1}^{N_{\mathrm{HF}}} \Big(2\|\phi_i\|_2\|\Delta_{\phi,i}\|_2+\|\Delta_{\phi,i}\|_2^2\Big)\notag\\
&\le 16T^2 B_r^2\epsilon_b(\delta_b, 2N_{\mathrm{HF}}) + (4T B_r\,\epsilon_b(\delta_b, 2N_{\mathrm{HF}}))^2\\
&=16T^2B_r^2\epsilon_b(\delta_b, 2N_{\mathrm{HF}})(1+\epsilon_b(\delta_b, 2N_{\mathrm{HF}}))\notag\\
&:=\epsilon_\Sigma
\label{eq:Eop_bd_remainder}
\end{align}

Therefore, a sufficient level of regularization on $\mathcal{E}_{b}^{\mathrm{RLHF}}(\delta_b)$ is $\epsilon_\Sigma$. Then, for any $1<c_{\zeta}$, choosing $\zeta=c_{\zeta}\epsilon_\Sigma$ on $\mathcal{E}_{b}^{\mathrm{RLHF}}(\delta_b)$ yields
\begin{equation}\label{eq:cov-pert-upper}
\frac{\|E\|_{op}}{\zeta}\le\frac1{c_{\zeta}}<1
\qquad\Longrightarrow\qquad \frac{1}{\sqrt{1-\|E\|_{op}/\zeta}}
\le \sqrt{\frac{c_{\zeta}}{c_{\zeta}-1}} .
\end{equation}

Since $E$ is symmetric, $-\|E\|_{op}I\preceq E\preceq \|E\|_{op}I$, and hence

\begin{align} \label{eq:inverse_comparison_repl_noeta}
&\Sigma+\zeta I-\|E\|_{op}I \;\preceq\; \tilde\Sigma+\zeta I \;\preceq\; \Sigma+\zeta I+\|E\|_{op}I\notag\\
&\Longrightarrow \Sigma+\zeta I-\|E\|_{op}I
\succeq
\Bigl(1-\frac{\|E\|_{op}}{\zeta}\Bigr)\,(\Sigma+\zeta I)\notag\\
&\Longrightarrow (\tilde\Sigma+\zeta I)^{-1}
\;\preceq\;
\frac{1}{1-\|E\|_{op}/\zeta}\,(\Sigma+\zeta I)^{-1}
\end{align}
as $\|E\|_{op}/\zeta<1$. This implies that for any $v\in\bR^d$,
\begin{equation}\label{eq:norm_transfer_noeta}
\|v\|_{(\tilde\Sigma+\zeta I)^{-1}}
\;\le\;
\frac{1}{\sqrt{1-\|E\|_{op}/\zeta}}\;
\|v\|_{(\Sigma+\zeta I)^{-1}}\le \sqrt{\frac{c_\zeta}{c_\zeta-1}}\|v\|_{(\Sigma+\zeta I)^{-1}}.
\end{equation}
\noindent Finally, returning to Eq.~\ref{eq:grad-like-tilde-bound-half} and combining
Eq.~\ref{eq:norm_transfer_noeta} with Lemma~\ref{lem:clean-like-grad-bound} (Eq.~\ref{eq:grad-like-bound-half-way}), on the event
$\mathcal E_{b}^{\mathrm{RLHF}}(\delta_b)$ we obtain
\begin{align*}
\bigl\|\nabla\tilde L(\mu^\star) \bigr\|_{(\tilde\Sigma+\zeta I)^{-1}}
&\le \bigl\|\nabla L(\mu^\star)\bigr\|_{(\tilde\Sigma+\zeta I)^{-1}}
+ \frac{4T B_r\,\epsilon_b(\delta_b, 2N_{\mathrm{HF}})}{\sqrt{\zeta }}\,
\Bigl(1+\tfrac12\,T B_r\,r_\mu\Bigr)\\
&\le \frac{1}{\sqrt{1-\|E\|_{op}/\zeta}}\;
\bigl\|\nabla L(\mu^\star)\bigr\|_{(\Sigma+\zeta I)^{-1}}
+ \frac{4T B_r\,\epsilon_b(\delta_b, 2N_{\mathrm{HF}})}{\sqrt{\zeta }}\,
\Bigl(1+\tfrac12\,T B_r\,r_\mu\Bigr)\\
&\le \sqrt{\frac{c_{\zeta}}{c_{\zeta}-1}}\;
\bigl\|\nabla L(\mu^\star)\bigr\|_{(\Sigma+\zeta I)^{-1}}
+ \frac{4T B_r\,\epsilon_b(\delta_b, 2N_{\mathrm{HF}})}{\sqrt{\zeta }}\,
\Bigl(1+\tfrac12\,T B_r\,r_\mu\Bigr)\\
&\le \sqrt{\frac{c_{\zeta}}{N_{\mathrm{HF}}(c_{\zeta}-1)}}\;
\sqrt{d\log \Bigl(1+\frac{4T^2 B_r^2}{\zeta d}\Bigr)+2\log \Bigl(\frac{1}{\delta_c}\Bigr)}
+ \frac{4T B_r\,\epsilon_b(\delta_b, 2N_{\mathrm{HF}})}{\sqrt{\zeta }}\,
\Bigl(1+\tfrac12\,T B_r\,r_\mu\Bigr).
\end{align*}

 Substituting the above bound into Eq.~\ref{eq:mu-opt-bound-grad-final}, on the event
$\mathcal E_{b}^{\mathrm{RLHF}}(\delta_b)$ and the event of Lemma \ref{lem:clean-like-grad-bound}
(with probability at least $1-\delta_c$), we obtain
\begin{align}
\|\tilde\mu-\mu^\star\|_{(\tilde\Sigma+\zeta I)}
&\le \frac{2}{\rho}\,\bigl\|\nabla\tilde L(\mu^\star)\bigr\|_{(\tilde\Sigma+\zeta I)^{-1}}
\;+\;2r_\mu\sqrt{\zeta }\notag\\[3pt]
&\le \frac{2\sqrt{c_\zeta}}{\rho\sqrt{N_{\mathrm{HF}}(c_\zeta-1)}}\;
\sqrt{d\log \Big(1+\frac{4T^2 B_r^2}{\zeta d}\Big)+2\log \Big(\frac{1}{\delta_c}\Big)} \notag\\
&\qquad\qquad + \frac{8T B_r\,\epsilon_b(\delta_b, 2N_{\mathrm{HF}})}{\rho\sqrt{\zeta }}\,
\Bigl(1+\tfrac12\,T B_r\,r_\mu\Bigr)
\;+\;2r_\mu\sqrt{\zeta }.
\label{eq:final-mu-bound}
\end{align}
Therefore, taking a union bound over $\mathcal E_{b}^{\mathrm{RLHF}}(\delta_b)$ and the event of
Lemma~\ref{lem:clean-like-grad-bound}, the claim of Theorem~\ref{thm:main-mu} holds with
probability at least $1-\delta_b-\delta_c$.

\medskip
The proof for the neural-softmax model is identical. By the modular structure of the pipeline, the only POMDP model-specific input used above is the
belief-accuracy event $\mathcal E_b^{\mathrm{RLHF}}(\delta_b)$, through the bound
\(\|\Delta_{\phi,i}\|_2\le 4T B_r\,\epsilon_b(\delta_b,2N_{\mathrm{HF}})\).
Under Corollary~\ref{cor:nn-param-belief-lip} and its corresponding high-probability event established in the second part of Corollary \ref{cor:belief-good-condition-cor}, similar event holds with
\(b_{i,h}^{(j),\Theta},b_{i,h}^{(j),\Theta^\star}\) replaced by
\(b_{i,h}^{(j),W},b_{i,h}^{(j),W^\star}\), and with
\(\epsilon_b(\delta_b,2N_{\mathrm{HF}})\) replaced by
\(\epsilon_b^{\mathrm{NN}}(\delta_b,2N_{\mathrm{HF}},\delta_{\mathrm{NN}})\).
All subsequent steps, e.g., the feature perturbation bound, covariance perturbation bound, gradient decomposition, and strong-convexity argument are downstream and therefore remain unchanged after this replacement.
\end{proof}

\begin{remark}\label{remark:why-not-exact-cov}
A common approach to bound $\|\nabla \tilde L(\mu^\star)\|_{(\tilde\Sigma+\zeta I)^{-1}}$ would be to apply an elliptical potential inequality to the summands \(
\tilde g_i^\star
=\bigl(\sigma(\tilde\phi_i^\top\mu^\star)-y_i\bigr)\tilde\phi_i, i=1,\dots,N_{\mathrm{HF}}.
\) as in \ref{lem:clean-like-grad-bound}, but this is not valid here because the required conditional mean-zero property fails as
$\bE[y_i\mid \phi_i]=\sigma(\phi_i^\top\mu^\star)$, hence in general \(
\bE\!\left[\sigma(\tilde\phi_i^\top\mu^\star)-y_i \,\middle|\, \tilde\phi_i\right]\neq 0.
\)
Therefore, $\{\tilde g_i^\star\}$ is not a martingale-difference sequence w.r.t.\ $\tilde\phi_i$, and we instead perform the decompose in \ref{eq:grad-like-tilde-bound-half}
and apply the self-normalized bound only to the clean gradient $\nabla L(\mu^\star)$, for which
$\bE[\sigma(\phi_i^\top\mu^\star)-y_i\mid \phi_i]=0$ holds.
\end{remark}

\medskip

\begin{remark}\label{rem:norm-change-justification}
The inequality in Eq.~\ref{eq:mu-opt-bound-grad-final} is a standard device that makes the analysis well-posed and compatible with self-normalized concentration arguments. 
Indeed, the strong convexity step yields control only in the semi-norm induced by $\tilde\Sigma$,
\[
\tilde L(\mu_1)-\tilde L(\mu_2)-\langle \nabla \tilde L(\mu_2),\mu_1-\mu_2\rangle \;\ge\; \frac{\rho}{2}\,\|\mu_1-\mu_2\|_{\tilde\Sigma}^2,
\]
where $\tilde\Sigma$ may be singular (see also \citet{Agarwal2020,Du2024Exploration-DrivenUtilization,ZhuPrincipledComparisons,NIPS2011_e1d5be1c}). 
Introducing the term $\zeta I$ ensures invertibility and allows us to apply Cauchy--Schwarz in the dual pair
$\|\cdot\|_{\tilde\Sigma+\zeta I}$ and $\|\cdot\|_{(\tilde\Sigma+\zeta I)^{-1}}$. 
This modification contributes an additive regularization bias of order $\sqrt{\zeta}\|\tilde\mu-\mu^\star\|_2\le 2r_\mu\sqrt{\zeta}$ (e.g., Theorem 5.2 in \citet{ZhuPrincipledComparisons}).
\end{remark}

\begin{remark}[Making $\zeta$ fully explicit]\label{rem:zeta-explicit}
Although we write
\[
\zeta \;:=\; c_{\zeta}\,16T^2B_r^2\,\epsilon_b(\delta_b,2N_{\mathrm{HF}})
\Bigl(1+\epsilon_b(\delta_b,2N_{\mathrm{HF}})\Bigr),
\]
the quantity $\epsilon_b(\delta_b,2N_{\mathrm{HF}})$ is \emph{explicit} in our analysis: it is given in
Eq.~\ref{eq:eps-b-bar-def} (see also Corollary~\ref{cor:belief-good-condition-cor}).
In particular, $\epsilon_b$ depends only on $(T,N_{\mathrm{HF}},\delta_b)$ and the problem parameters
$\alpha=(1-\kappa_P)(4-3\kappa_\Phi)$ and $c_b(\delta(\theta),\delta(w))$.
To obtain a deterministic bound that does not depend on the unknown perturbations,
$\delta(\theta),\delta(w)$ (or on $\kappa_P,\kappa_\Phi$), one may replace them by worst-case bounds
over the feasible set, e.g. $\delta(\theta),\delta(w)\le r_\Theta$ and
lower bounds on $\kappa_P,\kappa_\Phi$ from Lemma~\ref{lem:Dobrushin-row-col},
thereby obtaining a closed-form upper bound $\bar\epsilon_b$ and the explicit choice
\[
\zeta \;:=\; c_{\zeta}\,16T^2B_r^2\,\bar\epsilon_b\,(1+\bar\epsilon_b).
\]
This substitution only affects constants, while preserving the stated guarantees. Moreover, when \(\epsilon_b=0\), the same argument may be run with any \(\zeta>0\).
\end{remark}

\begin{remark}[Scaling under the prescribed choice of $\zeta$]\label{rem:zeta-choice}
The bound in Theorem~\ref{thm:main-mu} depends on $\zeta$ in all three terms; the statistical term depends on $\zeta$ only through the logarithmic factor, while the explicit belief-mismatch and regularization bias terms have the form
\(C_1\epsilon_b/\sqrt{\zeta}\) and \(C_2\sqrt{\zeta}\). Under the prescribed tuning
\(\zeta = \bar c\,\epsilon_b(1+\epsilon_b)\) (Theorem~\ref{thm:main-mu} and Remark~\ref{rem:zeta-explicit}),
the two explicit bias contributions satisfy
\[
C_1\frac{\epsilon_b}{\sqrt{\zeta}}+ C_2 \sqrt{\zeta}=\frac{C_1}{\sqrt{\bar c}}\sqrt{\frac{\epsilon_b}{1+\epsilon_b}}+C_2 \sqrt{\bar c}\sqrt{\epsilon_b(1+\epsilon_b)}=\mathcal O(\sqrt{\epsilon_b}),
\]
where we used that \(\epsilon_b\in(0,1]\) implies \(1+\epsilon_b\in[1,2]\). Thus, as the time-averaged belief error \(\epsilon_b\) decreases, the explicit belief-mismatch/regularization bias does not blow up under this tuning and decays at rate \(\sqrt{\epsilon_b}\), up to the logarithmic \(\zeta\)-dependence in the statistical term.
\end{remark}

\begin{corollary}[Value perturbation from reward parameter error]\label{cor:value-perturb}
Working under the setting of Theorem~\ref{thm:main-mu}, fix the (possibly randomized) policy
\(\pi:\Delta_{\bS}\to\Delta_{\bA}\) used in data generation and any \(\gamma\in(0,1)\).
For \(\mu\in\bR^d\), define the discounted value
\[
V_\mu^\pi(b)\;:=\;\bE^\pi\!\left[\sum_{t=0}^\infty \gamma^t\,r_\mu(b_t,a_t)\ \bigg|\ b_0=b\right].
\]
Then with probability at least \(1-(\delta_b+\delta_c)\),
\[
\sup_{b\in\Delta_{\bS}} \big|V_{\tilde\mu}^\pi(b)-V_{\mu^\star}^\pi(b)\big|
\;\le\;
\frac{B_r}{(1-\gamma)\sqrt{\zeta}}\ \epsilon(\delta_b,\delta_c,\zeta),
\]
where \(\epsilon(\delta_b,\delta_c,\zeta)\) is as in~Eq.~\ref{eq:main_mu_bound}.
\end{corollary}

\begin{proof}[Proof of Corollary \ref{cor:value-perturb}]\label{proof:corollary-value}
Fix any \((b,a)\in\Delta_{\bS}\times\bA\). By boundedness of the belief feature map,
\(\|\phi_b(b,a)\|_2\le B_r\), hence for all \(\mu,\mu'\in\bR^d\),
\begin{align*}
|r_\mu(b,a)-r_{\mu'}(b,a)|
&=\big|\phi_b(b,a)^\top(\mu-\mu')\big|\\
&\le B_r\,\|\mu-\mu'\|_2.
\end{align*}

Therefore, by linearity of expectation and \(\sum_{t\ge0}\gamma^t=(1-\gamma)^{-1}\),
\[
\big|V_\mu^\pi(b)-V_{\mu'}^\pi(b)\big|
\le \frac{B_r}{1-\gamma}\,\|\mu-\mu'\|_2.
\]
Taking \(\mu=\tilde\mu\), \(\mu'=\mu^\star\), and using \(\tilde\Sigma+\zeta I\succeq \zeta I\) (so
\(\|\nu\|_2\le \zeta^{-1/2}\|\nu\|_{\tilde\Sigma+\zeta I}\)),
we obtain
\[
\big|V_{\tilde\mu}^\pi(b)-V_{\mu^\star}^\pi(b)\big|
\le \frac{B_r}{(1-\gamma)\sqrt{\zeta}}\ \|\tilde\mu-\mu^\star\|_{\tilde\Sigma+\zeta I}.
\]
On the event of Theorem~\ref{thm:main-mu}, \(\|\tilde\mu-\mu^\star\|_{\tilde\Sigma+\zeta I}\le \epsilon(\delta_b,\delta_c,\zeta)\).
Taking the supremum over \(b\) completes the proof.

\end{proof}

\newpage

\section{Experiments}\label{sec:Experiments}

The experiments follow the modular structure of the analysis. We first examine filtering stability under prescribed model perturbations, then study covariance-preconditioned gradient EM as a mechanism for obtaining the learned POMDP parameter, and finally evaluate the effect of belief-induced feature perturbations on reward estimation.

 \subsection{Belief-stability bound}
\begin{figure}[h]
    \centering
    \begin{subfigure}[t]{0.78\textwidth}
        \centering
        \includegraphics[width=\linewidth]{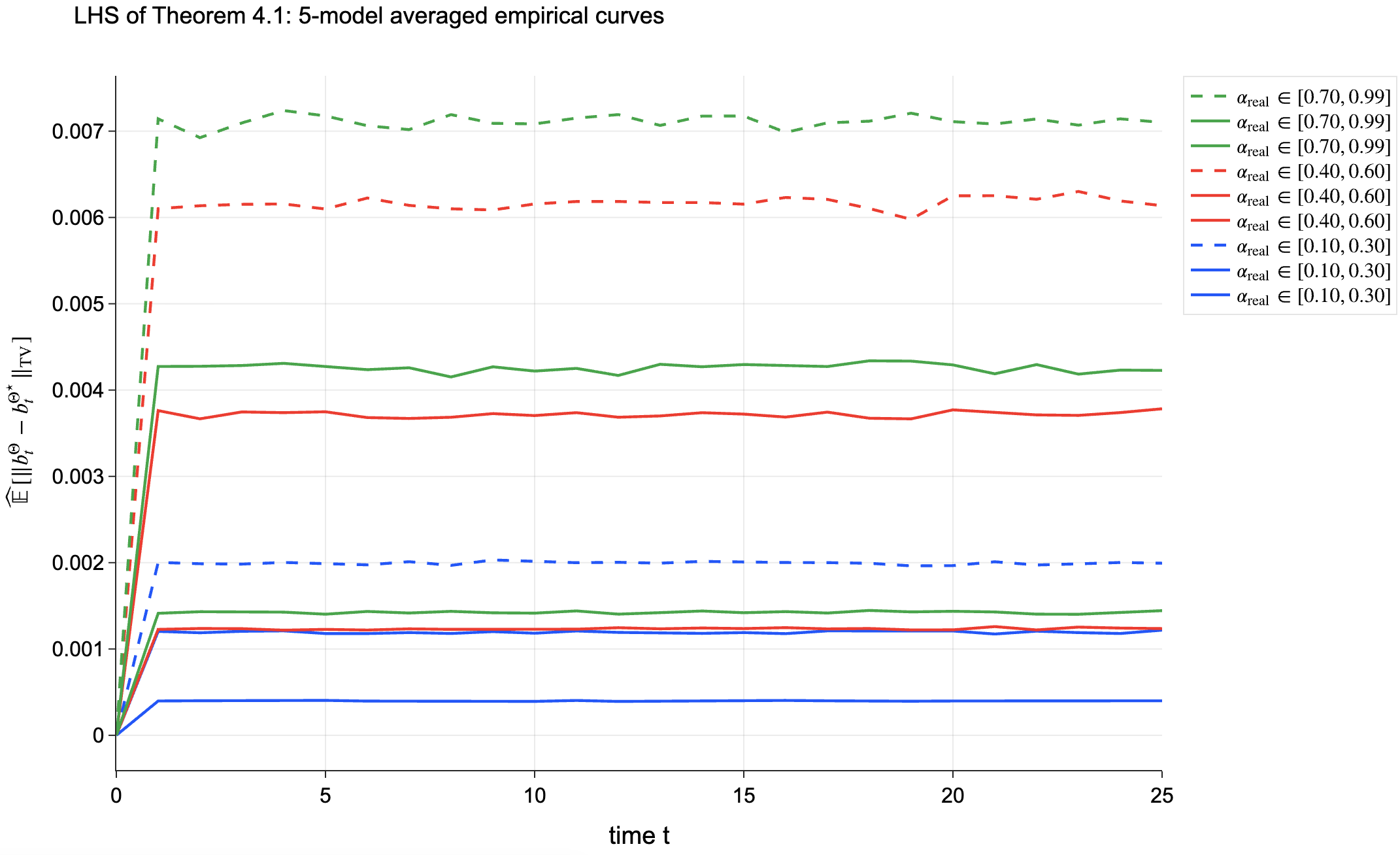}
        \caption{Empirical estimates of the LHS,
        \(\widehat{\bE}[\|b_t^\Theta-b_t^{\Theta^\star}\|_{\mathrm{TV}}]\),
        over time for different realized stability bands and perturbation levels.}
        \label{fig:belief-lhs-curves}
    \end{subfigure}
    
    \vspace{0.5em}
    
    \begin{subfigure}[t]{0.45\textwidth}
        \centering
        \includegraphics[width=\linewidth]{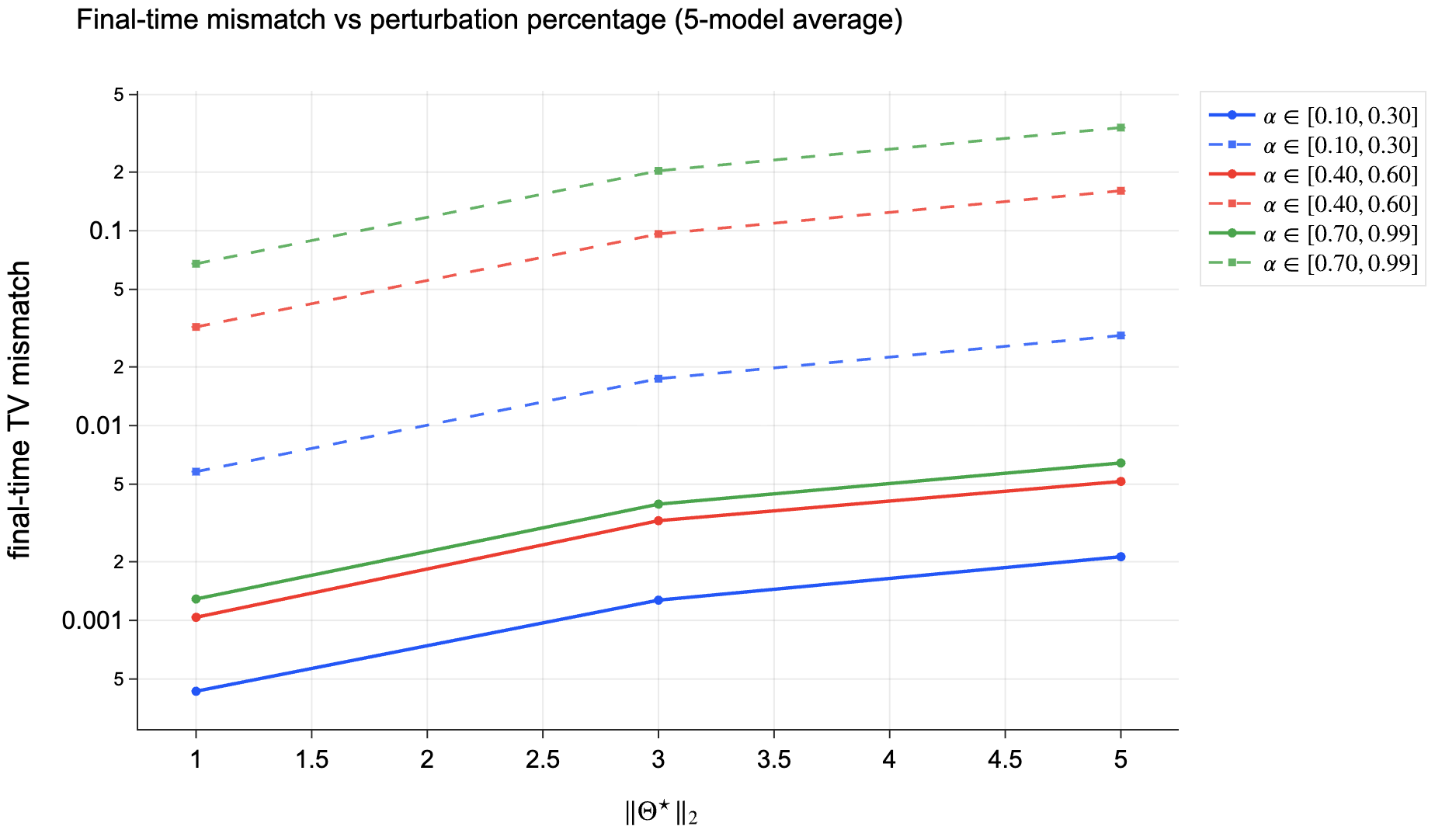}
        \caption{Final-time empirical mismatch and realized theorem RHS as a function of the perturbation magnitude, grouped by realized \(\alpha\)-bands.}
        \label{fig:belief-final-perturbation}
    \end{subfigure}\hfill
    \begin{subfigure}[t]{0.45\textwidth}
        \centering
        \includegraphics[width=0.80\linewidth]{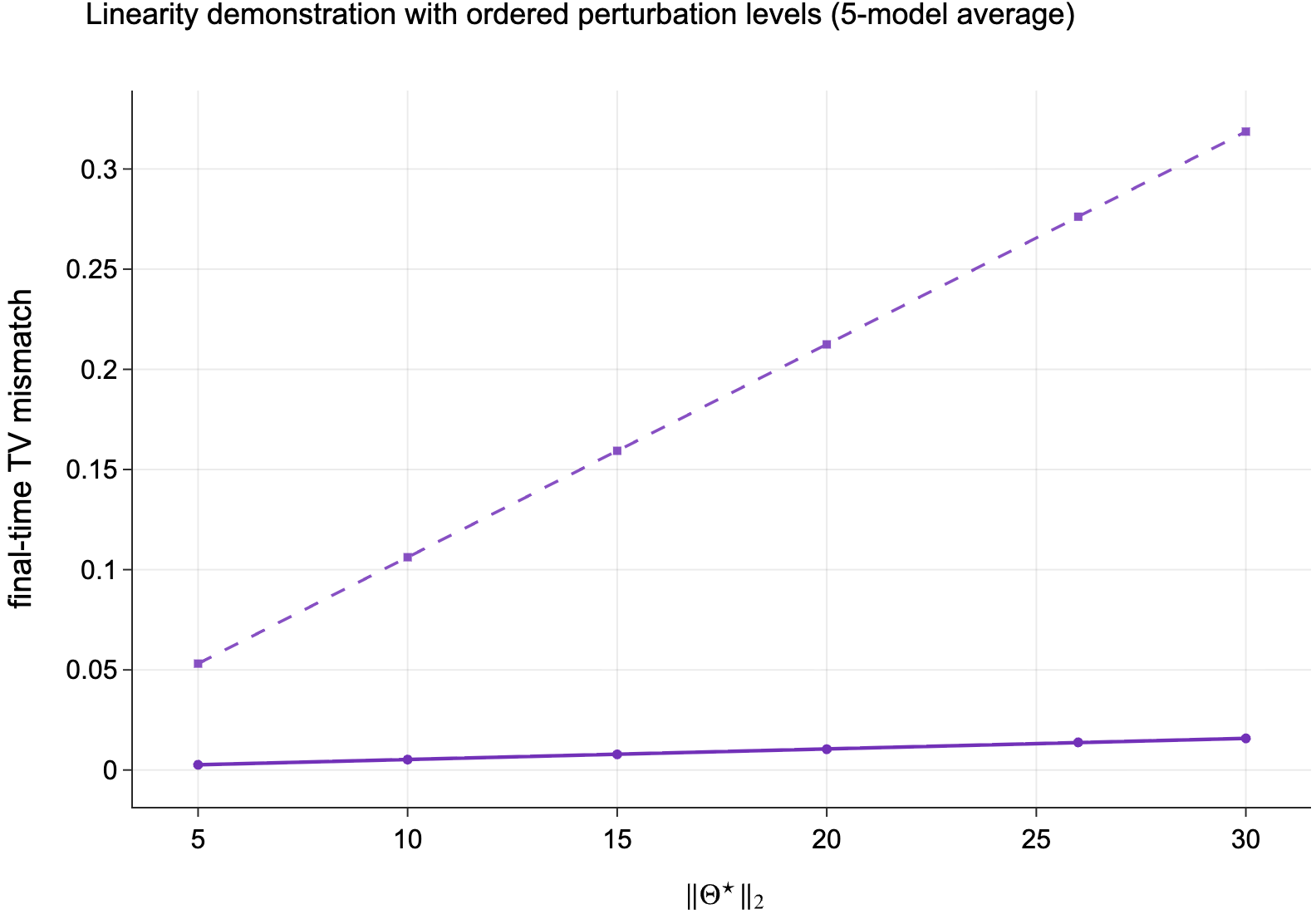}
        \caption{Fixed-\(\alpha\) linearity check: both the final-time empirical mismatch and the realized theorem RHS scale approximately linearly with the perturbation magnitude.}
        \label{fig:belief-linearity}
    \end{subfigure}
    
    \caption{Synthetic illustration of the belief-stability mechanism in Theorem~\ref{thm:param-belief-lip}. The experiment isolates the filtering component by comparing Bayesian filters computed under the true log-linear POMDP model and a perturbed model on common histories generated from the true model. The plots are intended to test the qualitative predictions of the theorem: bounded-in-time propagation under \(\alpha<1\), larger mismatch for weaker stability, and approximately linear dependence on the parameter perturbation size.}
    \label{fig:belief-stability-synthetic}
\end{figure}

We first test the belief-stability component of Theorem~\ref{thm:param-belief-lip}. The purpose of this experiment is to verify the qualitative dependences predicted by the theorem once a controlled model perturbation is introduced. In particular, Theorem~\ref{thm:param-belief-lip} predicts that, for histories generated from the true model,
\[
\bE\!\left[\|b_t^\Theta-b_t^{\Theta^\star}\|_{\mathrm{TV}}\right]
\le \frac{1-\alpha^t}{1-\alpha}\,
c_b(\delta(\theta),\delta(w)), \qquad
c_b(\delta(\theta),\delta(w))=B\left(\delta(w)+\frac{3-2\kappa_\Phi}{2}\delta(\theta)\right).
\]
Thus, for \(\alpha<1\), the time-propagation factor saturates at order \((1-\alpha)^{-1}\), while the one-step perturbation term is linear in the parameter deviations. The experiment is designed to separately demonstrates these two effects.

\paragraph{Setting.}
We generate finite log-linear POMDPs satisfying the bounded-feature condition
\(
\|\phi_p(s,a,s')\|_2\le B,\text{ and }\|\phi_\Phi(s,\hat s)\|_2\le B,
\)
with \(|\bS|=20\), \(|\bA|=10\), \(|\hat\bS|=20\), \(d_\theta=d_w=10\), horizon \(T=25\), and \(B=0.1\). True parameters \(\Theta^\star=(\theta^\star,w^\star)\) are sampled and retained according to the realized stability coefficient \(
\alpha_{\mathrm{real}}:=(1-\kappa_P)(4-3\kappa_\Phi),
\) so that the accepted models fall into prescribed stability bands. We then form a perturbed model \(\Theta=(\theta,w)\) by moving \(\Theta^\star\) in a fixed perturbation direction with prescribed magnitude. For each accepted model, observation histories are generated from the true POMDP, and the two Bayesian filters \(b_t^{\Theta^\star}\) and \(b_t^\Theta\) are run on the same action--observation histories. This common-history comparison matches the distributional object in Theorem~\ref{thm:param-belief-lip} and removes policy-induced trajectory drift from the experiment. To stress the least stable case within the fixed-action-sequence setting of the theorem, the action sequence is fixed to the worst realized action for the sampled model.

The plotted empirical quantity is the Monte Carlo estimate \(
\widehat{\bE}\!\left[\|b_t^\Theta-b_t^{\Theta^\star}\|_{\mathrm{TV}}\right],
\) where the expectation is over observation histories generated under \(\Theta^\star\). In the reported plots, curves are averaged over \(5\) independently accepted models in the corresponding stability regime and over \(50\) simulated observation histories per model. The displayed RHS is obtained by evaluating the theorem expression with the realized \(\kappa_P,\kappa_\Phi\) and the corresponding perturbation term \(c_b(\delta(\theta),\delta(w))\). Since this uses realized stability constants and finite Monte Carlo estimates, the figure should be interpreted as a synthetic stability diagnostic rather than a claim of tightness of the worst-case theorem constants.

\paragraph{Results.}
Figure~\ref{fig:belief-lhs-curves} shows that the empirical belief mismatch rises from zero and then remains stable over the horizon. This behavior is consistent with the structure of Theorem~\ref{thm:param-belief-lip}; the two filters start from the same prior, model mismatch injects fresh error after filtering begins, and the Dobrushin contraction prevents this error from accumulating linearly in time. The ordering across stability regimes is also consistent with the theorem. Models with larger realized \(\alpha_{\mathrm{real}}\) exhibit larger belief mismatch, while more stable models have substantially smaller plateaus. Within each stability band, increasing the perturbation magnitude increases the empirical mismatch.

Figure~\ref{fig:belief-final-perturbation} compares the final-time empirical mismatch with the realized theorem RHS. The empirical curves remain below the corresponding diagnostic RHS curves and inherit the same monotone structure and demonstrate how larger parameter perturbations lead to larger final-time belief error, and weaker stability, i.e., larger \(\alpha_{\mathrm{real}}\), shifts both the empirical mismatch and the bound upward. The gap between the empirical values and the RHS is expected, because the theorem uses total-variation Lipschitz bounds, Dobrushin worst-case contractions, and a uniform propagation argument; these constants are designed for robustness rather than numerical tightness.

Finally, Figure~\ref{fig:belief-linearity} isolates the perturbation dependence by using a separate fixed-stability experiment with \(\alpha_{\mathrm{real}}\in[0.25,0.40]\). In this setting the propagation factor \((1-\alpha^t)/(1-\alpha)\) is approximately fixed across models, so the theorem predicts that the dominant dependence on the perturbation size should enter through the linear term \(c_b(\delta(\theta),\delta(w))\). The observed final-time empirical mismatch and the realized RHS both grow approximately linearly with the perturbation magnitude, supporting the Lipschitz interpretation of Theorem~\ref{thm:param-belief-lip}. Overall, under a stable filter, model mismatch induces a controlled belief-error level, but this error increases with both the learned-model perturbation and the weakness of the realized stability coefficient.

\subsection{Covariance-preconditioned gradient EM}
\label{subsec:covariance-gradient-em-experiment}

Theorem~\ref{thm:param-belief-lip} is estimator-agnostic and converts a local parameter error into a belief-filtering error. Here we provide one example route for obtaining such a parameter estimate from incomplete trajectories. We therefore examine whether covariance-preconditioned gradient EM exhibits the two behaviors predicted by Proposition~\ref{prop:covariance-gradient-em-convergence} and~\ref{thm:finite-sample-covariance-gradient-em}, which are contraction of the population iteration and stabilization of the finite-sample iteration in a statistical neighborhood of \(\Theta^\star\).
\paragraph{Setting.}
We generate independently seeded candidate finite log-linear POMDPs satisfying Assumption~\ref{as:Model-LogLinear} and retain five models after the numerical screening described below, with\(
|\bS|=25,\, |\bA|=15,\, |\hat\bS|=25,\, d_\theta=d_w=10,\, K=200,\, B=0.1.
\) For each model \(m\in\{1,\ldots,5\}\), the bounded transition and observation features are generated using an independent seed, and the action sequence is fixed in advance as a balanced permutation of the available actions. Candidate parameters are restricted to the realized mixing range
\(
\alpha_m=(1-\kappa_{P,m})(4-3\kappa_{\Phi,m})\in[0.40,0.60]
\)
and are screened using population-level diagnostics for complete-data covariance conditioning, local relative curvature, covariance-gradient stability, and contraction of the population update. The accepted models satisfy \(\alpha_m\in[0.498,0.500]\), with local covariance radii \(r_m\in[0.25,1.00]\). This selection isolates the local regime addressed by the EM theorems and is independent of the final pilot and estimation samples. For each accepted model, \(\Sigma_{\star,m}\) is computed from the exact latent-state marginals under \(\Theta_m^\star\). The population-reference iteration uses a high-precision Monte Carlo approximation of the population self-gradient, centered so that \(\Theta_m^\star\) is a fixed point of the numerical map. We then generate ten independent pilot--estimation dataset pairs. A pilot sample of size \(M_0=25{,}000\) is used to construct the fixed regularized preconditioner \(\widehat\Sigma_{\lambda,m,j}\), while an independent sample of size \(M=80{,}000\) is reused throughout the batch gradient-EM iterations. The experiment therefore contains five independently generated models and fifty independent finite-sample trials.

\paragraph{Normalization and numerical audit.}
Since the covariance geometry varies across models, we normalize each trajectory by its model-specific initialization error:
\[
e_{m,j}^{(t)}:=\frac{\|\widehat\Theta_{m,j}^{(t)}-\Theta_m^\star\|_{\Sigma_{\star,m}}}{\|\widehat\Theta_{m,j}^{(0)}-\Theta_m^\star\|_{\Sigma_{\star,m}}},\qquad
e_{m,\mathrm{pop}}^{(t)}:=\frac{\|\Theta_m^{(t)}-\Theta_m^\star\|_{\Sigma_{\star,m}}}{\|\Theta_m^{(0)}-\Theta_m^\star\|_{\Sigma_{\star,m}}}.
\]
We first compute the median of \(e_{m,j}^{(t)}\) over the ten trials of each model and then take the median across the five model-level curves. The shaded region is the min--max range of these five curves.

The assumptions of Theorem~\ref{thm:finite-sample-covariance-gradient-em} contain suprema over \(\mathbb B_{\Sigma_{\star,m}}(\Theta_m^\star,r_m)\), which cannot be evaluated exactly in a finite simulation. We therefore evaluate the relative curvature, covariance-gradient stability, and empirical-gradient deviation on a fixed, data-independent collection of covariance-normalized directions and radial fractions. The grid-audited RHS is obtained from Eq.~\ref{eq:finite-sample-em-contraction-population} by replacing its uniform quantities by the corresponding grid maxima and the realized pilot-covariance deviation. We call a dataset audit-passing when the transferred contraction, initialization, and noise conditions hold for these quantities and all iterates remain inside the accepted covariance ball. This audit is a numerical diagnostic and does not replace the uniform assumptions of the theorem.

\begin{figure}[ht]
\centering
\includegraphics[width=0.98\linewidth]{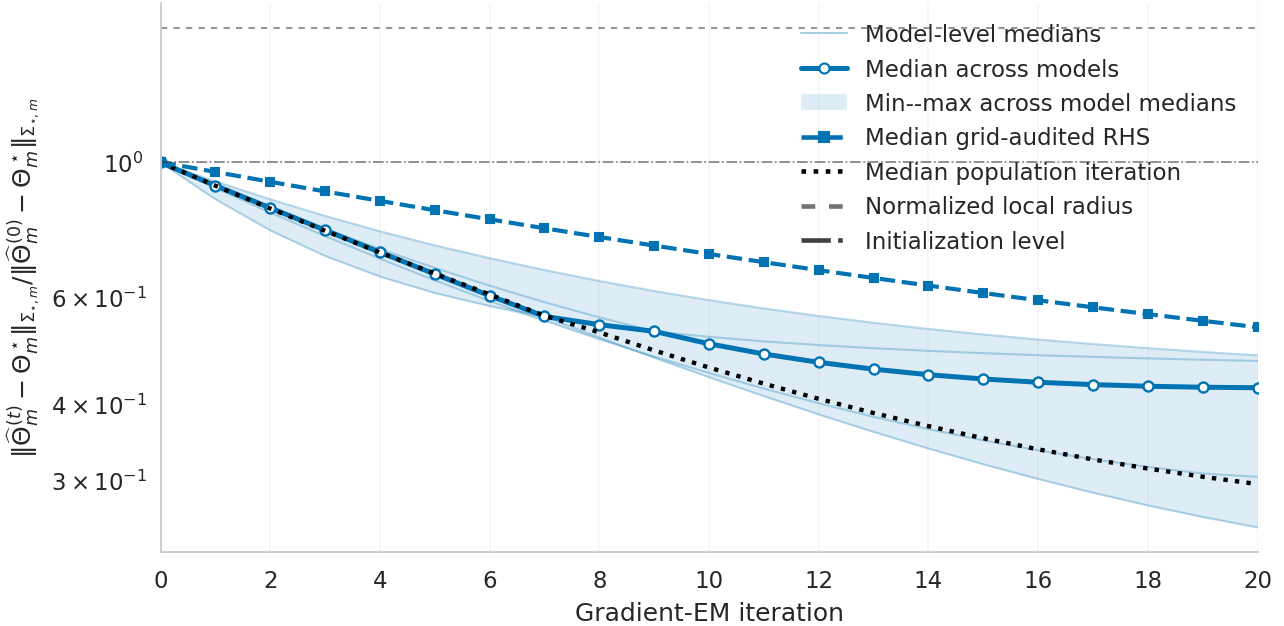}
\caption{Normalized covariance-norm error of covariance-preconditioned gradient EM over five audit-admissible POMDP models and ten independent pilot--estimation trials per model, with \(M=80{,}000\). Thin solid curves show model-wise trial medians, the thick solid curve their median across models, and the shaded region the corresponding min--max range. The dotted curve is the median population-reference iteration, while the dashed blue curve is the median normalized grid-audited RHS from Eq.~\ref{eq:finite-sample-em-contraction-population}. Horizontal reference lines indicate the initialization level and median normalized local radius. All \(50\) datasets pass the numerical audit, and all five models contribute to the displayed RHS.}
\label{fig:covariance-gradient-em-convergence}
\end{figure}

\paragraph{Results.}
Figure~\ref{fig:covariance-gradient-em-convergence} shows contraction across all five accepted models. After \(20\) iterations, the median normalized finite-sample error decreases from \(1\) to approximately \(0.42\), while the median population-reference error decreases to approximately \(0.29\). The two curves are initially close, but the finite-sample curve subsequently levels off above the population curve. This pattern is consistent with Theorem~\ref{thm:finite-sample-covariance-gradient-em}: the ideal population iteration contracts toward \(\Theta_m^\star\), whereas the finite-sample guarantee permits convergence to a sample-dependent neighborhood governed by the uniform gradient deviation and the pilot-covariance error.

The wider model-level range reflects differences in covariance geometry, contraction rates, and statistical error across the accepted models. Still, every model-level median remains below its initialization value, and all fifty trials stay within their local covariance balls. The grid-audited RHS decreases to about \(0.53\) and remains above the aggregate empirical curve.

% This experiment illustrates the model-estimation stage of the error-propagation pipeline in Figure~\ref{fig:error-propagation-pipeline}. The remaining POMDP parameter error propagates to belief error through Theorem~\ref{thm:param-belief-lip}, and then to reward-estimation error through Theorem~\ref{thm:main-mu}.

\begin{remark}[Scope of the numerical experiment]\label{rem:em-grid-audit-scope}
The candidate models are generated independently and retained only if they satisfy the fixed screening criteria. The experiment therefore illustrates a selected nonvacuous regime rather than how often this regime occurs for arbitrary POMDPs. Using five retained models and ten independent finite-sample trials per model also shows that the observed behavior is not specific to a single instance.
\end{remark}

\subsection{Reward-learning}
\label{subsec:exp-reward-learning}

We next isolate the downstream reward-estimation component of Theorem~\ref{thm:main-mu}. Owing to the modular structure of our error-propagation pipeline, this experiment separates reward estimation from the preceding belief-filtering analysis. In particular, it does not estimate a POMDP model and does not run Bayesian filtering. Instead, it generates clean belief-based trajectory features and perturbed belief-based trajectory features whose discrepancy is controlled through a prescribed belief-accuracy level \(\epsilon_b\). The goal is therefore to test the second part of the theory and show that once belief approximation induces trajectory-feature perturbations, how does this perturbation affect Bradley--Terry reward estimation.

The theorem controls the perturbed-feature estimator in the adaptive covariance geometry. In particular, ignoring logarithmic factors, Theorem~\ref{thm:main-mu} has the qualitative form of
\[
\|\tilde\mu-\mu^\star\|_{\tilde\Sigma+\zeta I}
\;\lesssim\;
\underbrace{N_{\mathrm{HF}}^{-1/2}}_{\text{statistical error}}+
\underbrace{\frac{T B_r}{\sqrt{\zeta}}\epsilon_b}_{\text{belief-induced bias}}+
\underbrace{r_\mu\sqrt{\zeta}}_{\text{regularization bias}} .
\]
Thus, for a fixed belief-error level \(\epsilon_b\) and fixed regularization \(\zeta\), increasing \(N_{\mathrm{HF}}\) should reduce the statistical component, while the belief-induced and regularization terms determine the remaining scale of the theorem RHS and are not directly reduced by additional preference comparisons.

\paragraph{Setting.}
We use a finite synthetic belief-MDP with \(|\bS|=20\), \(|\bA|=5\), reward-feature dimension \(d=5\), horizon \(T=10\), feature bound \(B_r=0.1\), and parameter constraint radius \(r_\mu=1\). The true reward parameter satisfies \(\|\mu^\star\|_2=0.8r_\mu\). For each
\[
N_{\mathrm{HF}}\in\{10,100,500,1000,5000,10000,20000,50000\}
\quad\text{and}\quad
\epsilon_b\in\{0.05,0.10,0.15,0.20,0.30\},
\]
we generate \(N_{\mathrm{HF}}\) independent pairwise comparisons. Clean belief-based trajectory features are denoted by \(\phi_i\), while perturbed features are denoted by \(\tilde\phi_i\). The perturbations are constructed so that each trajectory satisfies the prescribed average total-variation belief-error level \(\epsilon_b\), matching the type of belief-accuracy event used in the proof of Theorem~\ref{thm:main-mu}.

Preference labels are sampled from the clean Bradley--Terry model \(
\bP(y_i=1\mid \phi_i)=\sigma(\phi_i^\top\mu^\star),
\) whereas the estimator is fitted using the perturbed features \(\tilde\phi_i\):
\[
\tilde\mu\in
\arg\min_{\|\mu\|_2\le r_\mu}-\sum_{i=1}^{N_{\mathrm{HF}}}\Bigl[
y_i\log\sigma(\tilde\phi_i^\top\mu)+(1-y_i)\log\bigl(1-\sigma(\tilde\phi_i^\top\mu)\bigr)\Bigr].
\]
This creates the downstream misspecification studied by the theorem, that the preferences are generated according to the oracle belief features, but the learner observes only perturbed belief-induced features. For each \(\epsilon_b\), we set the adaptive-norm regularization level according to the deterministic theorem scale
\(\zeta=1.05\cdot 16T^2B_r^2\epsilon_b(1+\epsilon_b)\).

\begin{figure}[h]
    \centering
    \includegraphics[width=0.58\linewidth]{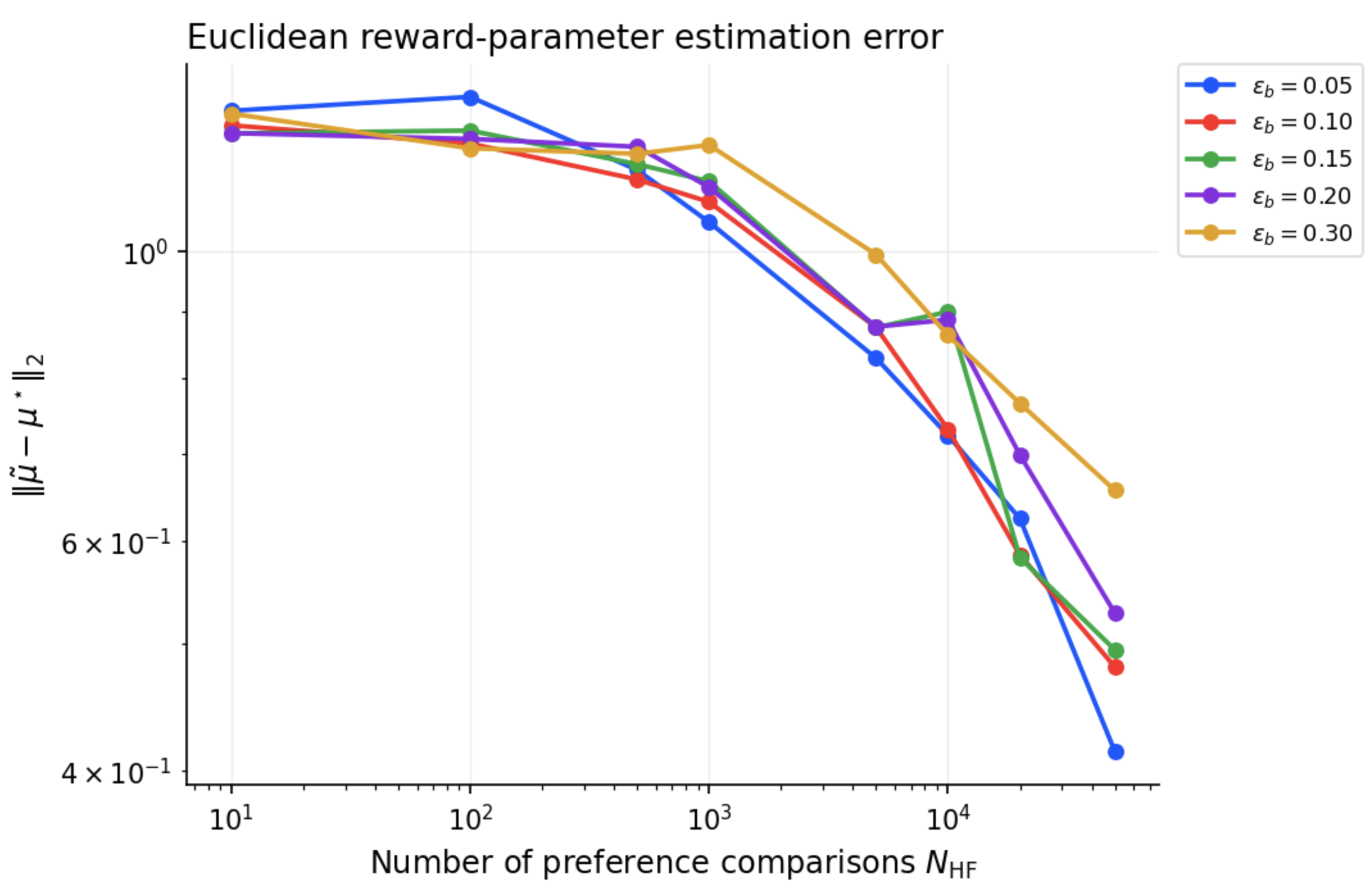}
    \caption{Euclidean reward-parameter error $\|\tilde\mu-\mu^\star\|_2$, reported as an auxiliary diagnostic for the experiment of Figure~\ref{fig:reward-adaptive-main}. The adaptive covariance norm controlled by Theorem~\ref{thm:main-mu} is reported in the main text.}
    \label{fig:reward-euclidean}
\end{figure}

\paragraph{Results.}
Figure~\ref{fig:reward-adaptive-main} reports medians over \(100\) independent trials and shows the theorem-relevant quantity \(\|\tilde\mu-\mu^\star\|_{\tilde\Sigma+\zeta I}\) together with the corresponding theorem RHS. The empirical adaptive-norm error decreases as \(N_{\mathrm{HF}}\) increases, consistent with the reduction of the statistical component in Theorem~\ref{thm:main-mu}. At the same time, for a fixed sample size, larger prescribed belief error \(\epsilon_b\) generally leads to larger estimation error. This is consistent with the perturbation term in the theorem, where belief mismatch enters through the accumulated trajectory-feature error.

The empirical curves remain below the proof-level RHS across the tested regimes. The gap is expected and should not be interpreted as a failure of the theoretical bound, as the theorem is a high-probability guarantee and uses conservative ingredients. Figure~\ref{fig:reward-adaptive-main} also illustrates the main qualitative implication of the theorem. More preference comparisons reduce the statistical error, but they do not remove the effect of a fixed belief approximation error. In the theorem RHS, once the statistical term becomes small, the remaining scale is governed by the belief-induced bias term and the regularization bias. This matches the conceptual message of the paper, that under partial observability with imperfect belief construction, reward learning has an additional error channel that is absent from the fully observed Bradley--Terry setting.

Figure~\ref{fig:reward-euclidean} reports \(\|\tilde\mu-\mu^\star\|_2\) as a diagnostic in the ordinary Euclidean parameter norm. This is not the norm controlled by Theorem~\ref{thm:main-mu}; therefore, crossings and small non-monotonicities across \(\epsilon_b\)-curves should not be overinterpreted. Nevertheless, the Euclidean plot shows the same broad behavior mentioned earlier that increasing \(N_{\mathrm{HF}}\) improves reward-parameter recovery, while larger belief perturbations tend to make estimation harder.

\end{document}